\documentclass[11pt]{article}

\usepackage{amsmath}
\usepackage{amssymb}
\usepackage{amsthm}
\usepackage{bbold}
\usepackage{bm}
\usepackage{colonequals}
\usepackage{authblk}
\usepackage{xcolor}
\usepackage[letterpaper,margin=1in]{geometry}
\usepackage{graphicx}
\usepackage{siunitx}

\definecolor{cblue}{rgb}{0.121569,0.466667,0.705882}
\definecolor{corange}{rgb}{1.000000,0.498039,0.054902}
\definecolor{cgreen}{rgb}{0.172549,0.627451,0.172549}

\usepackage{hyperref}
\hypersetup{
  linkcolor  = cblue,
  citecolor  = cgreen,
  urlcolor   = corange,
  colorlinks = true,
}

\newtheorem{theorem}{Theorem}[section]
\newtheorem{remark}[theorem]{Remark}
\newtheorem{lemma}[theorem]{Lemma}
\newtheorem{definition}[theorem]{Definition}
\newtheorem{proposition}[theorem]{Proposition}
\newtheorem{corollary}[theorem]{Corollary}
\newtheorem{conjecture}[theorem]{Conjecture}

\newcommand{\what}{\widehat}

\newcommand{\CC}{\mathbb{C}}
\newcommand{\EE}{\mathbb{E}}
\newcommand{\RR}{\mathbb{R}}
\renewcommand{\SS}{\mathbb{S}}
\newcommand{\ZZ}{\mathbb{Z}}
\DeclareSymbolFont{bbold}{U}{bbold}{m}{n}
\DeclareSymbolFontAlphabet{\mathbbold}{bbold}
\newcommand{\One}{\mathbbold{1}}

\newcommand{\ba}{\bm a}
\newcommand{\be}{\bm e}
\newcommand{\bq}{\bm q}
\newcommand{\bu}{\bm u}
\newcommand{\bv}{\bm v}
\newcommand{\bw}{\bm w}
\newcommand{\bx}{\bm x}
\newcommand{\bA}{\bm A}
\newcommand{\bB}{\bm B}
\newcommand{\bC}{\bm C}
\newcommand{\bD}{\bm D}
\newcommand{\bE}{\bm E}
\newcommand{\bH}{\bm H}
\newcommand{\bK}{\bm K}
\newcommand{\bM}{\bm M}
\newcommand{\bP}{\bm P}
\newcommand{\bQ}{\bm Q}
\newcommand{\bR}{\bm R}
\newcommand{\bS}{\bm S}
\newcommand{\bT}{\bm T}
\newcommand{\bU}{\bm U}
\newcommand{\bV}{\bm V}
\newcommand{\bW}{\bm W}
\newcommand{\bX}{\bm X}
\newcommand{\bY}{\bm Y}
\newcommand{\bZ}{\bm Z}

\newcommand{\sD}{\mathcal{D}}
\newcommand{\sG}{\mathcal{G}}
\newcommand{\sH}{\mathcal{H}}
\newcommand{\sL}{\mathcal{L}}
\newcommand{\sN}{\mathcal{N}}

\DeclareSymbolFont{sfoperators}{OT1}{cmss}{m}{n}
\DeclareSymbolFontAlphabet{\mathsf}{sfoperators}
\makeatletter
\renewcommand{\operator@font}{\mathgroup\symsfoperators}
\makeatother

\DeclareMathOperator*{\argmax}{arg\,max}
\DeclareMathOperator{\GOE}{GOE}
\DeclareMathOperator{\sym}{sym}
\DeclareMathOperator{\Tr}{Tr}
\DeclareMathOperator{\Var}{Var}
\DeclareMathOperator{\Cov}{Cov}
\DeclareMathOperator{\Unif}{Unif}
\renewcommand\Re{\operatorname{Re}}

\newcommand\numberthis{\addtocounter{equation}{1}\tag{\theequation}}

\title{Lehner's operator norm formulas, semidefinite programming, and spiked matrix models}
\author{Dmitriy Kunisky\thanks{Email: \texttt{kunisky@jhu.edu}.}}
\affil{Department of Applied Mathematics \& Statistics, Johns Hopkins University}
\date{June 12, 2026}

\newcommand{\free}{\mathsf{free}}
\renewcommand{\sc}{\mathsf{sc}}
\newcommand{\rad}{\mathsf{rad}}
\newcommand{\SDP}{\mathsf{SDP}}
\newcommand{\herm}{\mathsf{herm}}
\newcommand{\spec}{\mathsf{spec}}
\newcommand{\Law}{\mathsf{Law}}
\newcommand{\id}{\mathfrak{1}}

\begin{document}

\maketitle
\thispagestyle{empty}

\begin{abstract}
    Lehner (1999) derived elegant formulas for the operator norm $\|\mathfrak{X}\|$ of operators of the form $\mathfrak{X} = \bA_0 \otimes \id + \sum_{i = 1}^n \bA_i \otimes \mathfrak{m}_i$, also easily generalized to the spectral edge $\lambda_{\max}(\mathfrak{X})$, in terms of nonlinear optimization problems over positive definite matrices.
    Here the $\bA_i$ are finite-dimensional Hermitian matrices, the $\mathfrak{m}_i$ are either free semicircular or free Rademacher families of operators, and $\id$ is the identity operator.
    We first show that both of Lehner's nonlinear optimizations can be rewritten as linear semidefinite programs (SDPs), even in the Rademacher case where Lehner's optimization is not itself convex.
    We give the primal and dual forms of these SDPs, derive the complementary slackness relations and consequences thereof, and propose that the SDPs are more stable and accurate than the iterative numerical scheme proposed in Lehner's original work.

    We then apply the SDPs from the semicircular case to spiked matrix models, studied recently via Lehner's formula by Bandeira, Cipolloni, Schr{\"o}der, and van Handel~(2024).
    We give a new proof of the Baik--Ben Arous--P\'{e}ch\'{e} (BBP) transition they establish in models with isotropic (but possibly correlated) Gaussian noise by constructing feasible variables for the associated primal and dual SDPs.
    Combining our construction with a sensitivity interpretation of optimal dual variables, we study the fluctuations of leading eigenvectors of such models.
    We conjecture and give numerical evidence that these fluctuations are Gaussian but anisotropic and non-universal, and that their covariance may be computed in terms of the optimizer of the dual of Lehner's formula, which in turn approximately equals the leading eigenmatrix of a completely positive operator associated to the covariance of the noise model.
\end{abstract}

\clearpage

\tableofcontents
\thispagestyle{empty}

\clearpage

\section{Introduction}
\pagenumbering{arabic}

\subsection{Lehner formulas}

Write $\CC^{d \times d}_{\herm}$ for the set of $d \times d$ Hermitian matrices.
Below we always take $\bA_0, \bA_1, \dots, \bA_n \in \CC^{d \times d}_{\herm}$ for $d, n \geq 1$.
When we discuss operators, they will always be bounded self-adjoint operators on Hilbert spaces of countable dimension.
For $\mathfrak{X}$ such an operator, we write $\spec(\mathfrak{X}) \subseteq \RR$ for its spectrum, which we recall is a compact set.
Thus, we may also speak of $\lambda_{\max}(\mathfrak{X}) \colonequals \max\{\lambda: \lambda \in \spec(\mathfrak{X})\}$, the right edge of the spectrum.
We may also apply these notions to finite-dimensional Hermitian matrices $\mathfrak{X} = \bX$, in which case $\spec(\bX)$ is the set of (finitely many) real eigenvalues and $\lambda_{\max}(\bX)$ is the largest element of this set.
As above, we always use Fraktur letters ($\mathfrak{s}, \mathfrak{r}, \mathfrak{X}$) to denote infinite-dimensional operators, and bold capital letters ($\bA$, $\bX$, $\bY$) to denote finite-dimensional matrices.

The seminal work of \cite{Lehner-1999-NormFreeOperatorMatrixCoefficients} showed the following formulas for $\lambda_{\max}(\mathfrak{X})$ for $\mathfrak{X}$ a sum of tensor products of deterministic Hermitian matrices $\bA_0, \bA_1, \dots, \bA_n$ with certain free families of operators.
Remarkably, these formulas can ``solve out'' the role of the infinite-dimensional operators and leave only explicit finite-dimensional optimization problems parameterized by $\bA_0, \bA_1, \dots, \bA_n$.
We call these the \emph{Lehner formulas}.

\begin{theorem}[Semicircular Lehner formula, {\cite[Corollary 1.5]{Lehner-1999-NormFreeOperatorMatrixCoefficients}}]
    \label{thm:lehner-sc}
    Suppose $(\mathfrak{s}_i)_{i = 1}^n$ is a free semicircular family of operators on a suitable Hilbert space and let $\id$ be the identity operator on that space.
    Then,
    \begin{align*}
      \Lambda_{\sc}(\bA_0, \bA_1, \dots, \bA_n) &\colonequals \lambda_{\max}\left(\bA_0 \otimes \id + \sum_{i = 1}^n \bA_i \otimes \mathfrak{s}_i\right) \\
      &= \inf_{\substack{\bZ \in \CC^{d \times d}_{\herm} \\ \bZ \succ \bm 0}} \lambda_{\max}\left(\bZ + \bA_0 + \sum_{i = 1}^n \bA_i\bZ^{-1}\bA_i\right). \numberthis \label{eq:lehner-sc}
    \end{align*}
    Further, the infimum in \eqref{eq:lehner-sc} can be restricted to those $\bZ$ for which the expression inside $\lambda_{\max}(\cdot)$ equals a scalar multiple of $\bm I_d$.
\end{theorem}

\begin{theorem}[Rademacher Lehner formula, {\cite[Theorem 1.1]{Lehner-1999-NormFreeOperatorMatrixCoefficients}}]
    \label{thm:lehner-rad}
    Suppose $(\mathfrak{r}_i)_{i = 1}^n$ is a free Rademacher family of operators on a suitable Hilbert space and let $\id$ be the identity operator on that space.
    Then,
    \begin{align*}
      &\Lambda_{\rad}(\bA_0, \bA_1, \dots, \bA_n) \\
      &\hspace{1cm}\colonequals \lambda_{\max}\left(\bA_0 \otimes \id + \sum_{i = 1}^n \bA_i \otimes \mathfrak{r}_i\right) \\
      &\hspace{1cm}= \inf_{\substack{\bZ \in \CC^{d \times d}_{\herm} \\ \bZ \succ \bm 0}} \lambda_{\max}\left(\bZ + \bA_0 + \sum_{i = 1}^n \frac{1}{2} \bZ^{1/2}((\bm I_d + 4(\bZ^{-1/2}\bA_i \bZ^{-1/2})^2)^{1/2} - \bm I_d)\bZ^{1/2}\right). \numberthis \label{eq:lehner-rad}
    \end{align*}
    Further, the infimum in \eqref{eq:lehner-rad} can be restricted to those $\bZ$ for which the expression inside $\lambda_{\max}(\cdot)$ equals a scalar multiple of $\bm I_d$.
\end{theorem}
\noindent
We note that our formulation of \eqref{eq:lehner-rad} differs from that of \cite{Lehner-1999-NormFreeOperatorMatrixCoefficients} by a change of variables $\bZ \leftarrow 2\bZ$, which we can already see here makes the Rademacher formula look more similar to the semicircular formula and which we will use to make a common reformulation of both in Section~\ref{sec:perspective-reform}.

\begin{remark}[Norm versus right edge]
    For $\mathfrak{X} = \bA_0 \otimes \id + \sum_{i = 1}^n \bA_i \otimes \mathfrak{m}_i$ for the operators $\mathfrak{m}_i$ appearing in either result above, Lehner's results in \cite{Lehner-1999-NormFreeOperatorMatrixCoefficients} give formulas for the operator norm $\|\mathfrak{X}\|$.
    Lehner's versions are given by just replacing $\lambda_{\max}(\cdot)$ with $\|\cdot\|$ throughout the above.
    But, as has been remarked for instance by \cite{BCSVH-2024-MatrixConcentrationFreeProbability2}, the same exact proofs also give the formulas for $\lambda_{\max}(\mathfrak{X})$ that we have stated; further, as $\mathfrak{X}$ is bounded and self-adjoint we have $\|\mathfrak{X}\| = \max\{\lambda_{\max}(\mathfrak{X}), \lambda_{\max}(-\mathfrak{X})\}$, so the spectral edge formulas immediately recover Lehner's norm formulas.
\end{remark}

This paper concerns the structure of the formulas \eqref{eq:lehner-sc} and \eqref{eq:lehner-rad}.
To make the motivation for studying these more concrete, before summarizing our main results let us review the interest in Lehner's formulas in the theories of von Neumann algebras, random matrices, and random graphs.

\subsection{Models and applications of operator series}
\label{sec:appl}

First, we note that the operators $\mathfrak{X}$ discussed above can be realized concretely, as discussed in~\cite{Lehner-1999-NormFreeOperatorMatrixCoefficients}.
Namely, a free semicircular family $(\mathfrak{s}_i)_{i = 1}^n$ can be realized on the full Fock space $\mathcal{T}(\CC^n)$ as the operators $\mathfrak{s}_i = \mathfrak{l}_i + \mathfrak{l}_i^*$ for creation and annihilation operators $\mathfrak{l}_i$ and $\mathfrak{l}_i^*$, in the notation of \cite{Lehner-1999-NormFreeOperatorMatrixCoefficients}.
And, a free Rademacher family $(\mathfrak{r}_i)_{i = 1}^n$ can be taken to be the collection of images of generators of a free product $G = (\ZZ / 2\ZZ)^{*n}$ under its left-regular representation, operators on the Hilbert space $\ell^2(G)$.
There has been considerable interest in the literature in the question of efficiently computing the norms of operators formed in such ways.
For some relevant references, see, e.g., \cite{Kesten-1958-SymmetricRandomWalksGroups,AO-1976-ComputingNormsCStar,Haagerup-1978-NonNuclearCStar,HP-1993-BoundedLinearOperatorsCStar,Lehner-1998-FreeOperatorsOperatorCoefficients,Lehner-1999-NormFreeOperatorMatrixCoefficients,Buchholz-1999-NormConvolutionFreeGroups,Lehner-2001-ComputationSpectraFreeProbability,FNT-2014-ComputeOperatorNorm,PvH-2025-ExtremeSingularValuesFree}, as well as discussion in the book \cite{Pisier-2003-IntroductionOperatorSpaceTheory}.

Further, instead of spectral edges of infinite-dimensional operators, both the semicircular and Rademacher cases of Lehner's formula can be viewed naturally as limits of spectral edges of certain random matrices.
\begin{definition}[Gaussian orthogonal ensemble]
    \label{def:goe}
    $\GOE(m)$ is the law of the random matrix $\bW \in \RR^{m \times m}_{\sym}$ having $W_{ij} = W_{ji} \sim \sN(0, \frac{1}{m}(1 + \One\{i = j\}))$ independently for all $1 \leq i \leq j \leq m$.
\end{definition}

\begin{theorem}[\cite{Schultz-2005-StrongConvergenceRealSymplectic,CM-2011-StrongAsymptoticFreenessHaar}]
    \label{thm:lehner-random}
    Let $\bA_0, \bA_1, \dots, \bA_n \in \CC^{d \times d}_{\herm}$.
    Then, the following hold, with convergences in probability as $m \to \infty$:
    \begin{enumerate}
        \item Let $\bW_1^{(m)}, \dots, \bW^{(m)}_n \sim \GOE(m)$ independently. Then,
            \[ \lambda_{\max}\left(\bA_0 \otimes \bm I_m + \sum_{i = 1}^n \bA_i \otimes \bW_i^{(m)}\right) \to \Lambda_{\sc}(\bA_0, \bA_1, \dots, \bA_n). \]
        \item Let $\bP_1^{(m)}, \dots, \bP_n^{(m)}$ be projection matrices to independent Haar-distributed subspaces of $\CC^m$ of dimension $m/2$ for each even $m$.
            Write $\bR_i^{(m)} \colonequals \bm I_m - 2\bP_i^{(m)}$ for each $m \geq 1$ and $i \in [n]$ (forming orthogonal reflection matrices).
            Then, taking the limit only over even $m$,
            \[ \lambda_{\max}\left(\bA_0 \otimes \bm I_m + \sum_{i = 1}^n \bA_i \otimes \bR_i^{(m)}\right) \to \Lambda_{\rad}(\bA_0, \bA_1, \dots, \bA_n). \]
    \end{enumerate}
\end{theorem}
\noindent
In both parts of Theorem~\ref{thm:lehner-random}, the expectations of the left-hand sides also converge to the right-hand side by standard arguments showing uniform integrability of the spectral edge of the associated random matrices.
See also \cite{BC-2019-EigenvaluesRandomLifts,VanHandel-2025-StrongConvergenceSurvey} for additional discussion of the framework of \emph{strong convergence} generalizing results of this kind for various random matrices.

The above gives infinite- and finite-dimensional models that realize Lehner's formulas.
While the random matrix models of Theorem~\ref{thm:lehner-random} are perhaps somewhat artificial, the formulas can also be used to understand other finite-dimensional random matrices more relevant to applications.

In the semicircular case, the recent line of work \cite{BB-2021-SpectralNormGaussianCorrelated, BBVH-2021-MatrixConcentrationFreeProbability, BCSVH-2024-MatrixConcentrationFreeProbability2} showed that, under mild assumptions, the spectrum of the random matrix
\[ \bX^{(m)} \colonequals \bA_0 \otimes \bm I_m + \sum_{i = 1}^n \bA_i \otimes \bW_i^{(m)} \in \CC^{dm \times dm}_{\herm} \]
for large $m$, whose ``limit'' in various senses is indeed the operator $\mathfrak{X} = \bA_0 \otimes \id + \sum_{i = 1}^n \bA_i \otimes \mathfrak{s}_i$ (called $\bX_{\free}$ in the notation of the above references), resembles that of
\begin{equation}
    \label{eq:X-gauss}
    \bX = \bA_0 + \sum_{i = 1}^n g_i \bA_i \in \CC^{d \times d}_{\herm}
\end{equation}
for $g_1, \dots, g_n \sim \sN(0, 1)$ i.i.d.\ Gaussian \emph{scalars}.
The latter kind of model is called a \emph{Gaussian matrix series}, and by standard manipulations \emph{any} Hermitian $\bX$ with jointly Gaussian entries can be represented in this way.
Combining this comparison principle with Theorem~\ref{thm:lehner-sc} which calculates $\lambda_{\max}(\mathfrak{X}) = \Lambda_{\sc}(\bA_0, \bA_1, \dots, \bA_n)$ yields the following.
\begin{theorem}[Corollary~2.3 of \cite{BCSVH-2024-MatrixConcentrationFreeProbability2}]
    \label{thm:BCSVH}
    For $\bX$ defined as in \eqref{eq:X-gauss}, define the quantities
\begin{align*}
\sigma(\bX)^2
  &\colonequals \|\EE(\bX - \EE \bX)^2\| \\
  &= \left\| \sum_{i = 1}^n \bA_i^2 \right\|, \numberthis \\
v(\bX)^2
  &\colonequals \|\EE \mathsf{vec}(\bX - \EE \bX)\mathsf{vec}(\bX - \EE \bX)^{*}\| \\
  &= \left\|\sum_{i = 1}^n \mathsf{vec}(\bA_i) \mathsf{vec}(\bA_i)^*\right\|, \numberthis \label{eq:BBVH-v} \\
\sigma_*(\bX)^2
  &\colonequals \sup_{\|\bu\| = \|\bw\| = 1}\EE[|\bu^{*}(\bX - \EE \bX)\bw|^2] \\
  &= \sup_{\|\bu\| = \|\bw\| = 1} \sum_{i = 1}^n \left| \langle \bu, \bA_i \bw \rangle \right|^2, \numberthis \label{eq:BBVH-sigma-star} \\
  \widetilde v(\bX)^2
  &\colonequals
v(\bX) \cdot \sigma(\bX). \numberthis
\end{align*}
    There is a universal constant $C > 0$ such that
    \[
        \left|\EE\lambda_{\max}(\bX) - \Lambda_{\sc}(\bA_0, \bA_1, \dots, \bA_n)\right|
        \leq
        C \,\widetilde{v}(\bX) (\log d)^{3/4}.
    \]
    Also, if $\bX = \bX^{(d)} \in \CC^{d \times d}_{\herm}$ is a sequence of such random matrices with $\bX^{(d)} = \bA_0^{(d)} + \sum_{i = 1}^n g_i\bA_i^{(d)}$ for some $n = n(d)$, and one has, as $d \to \infty$,
    \[
        \widetilde{v}(\bX^{(d)})(\log d)^{3/4} \to 0,
    \]
    then the following convergence in probability holds as $d \to \infty$:
    \[
        |\lambda_{\max}(\bX^{(d)}) - \Lambda_{\sc}(\bA_0^{(d)}, \bA_1^{(d)}, \dots, \bA_n^{(d)})| \to 0.
    \]
\end{theorem}
\noindent
We will discuss in Sections~\ref{sec:intro-spiked} and~\ref{sec:bbp} how our reformulations of Lehner's semicircular formula clarify some aspects of the application of the above to spiked matrix models by \cite{BCSVH-2024-MatrixConcentrationFreeProbability2}.

The Rademacher case is most naturally associated to questions about finite and infinite graphs.
As shown by \cite{GVK-2019-SpectraInfiniteGraphsFreenessAmalgamation}, the adjacency operators of universal covering trees $T$ of arbitrary finite graphs $G$, and more generally periodic Jacobi operators on such trees, can be written as $\bX = \sum_{i = 1}^n \bA_i \otimes \mathfrak{r}_i$ for a free Rademacher family $(\mathfrak{r}_i)_{i = 1}^n$, where $n$ is the number of edges in the graph.
Roughly speaking, the adjacency operator of such a $T$ is just the adjacency matrix of $G$, but where every edge's entry is replaced by a free Rademacher element.
In particular, Lehner's formula can be used to compute the spectral radius of such trees.
\begin{theorem}[\cite{GVK-2019-SpectraInfiniteGraphsFreenessAmalgamation}]
    \label{thm:universal-cover-rademacher}
    Let $G$ be a finite connected graph with $n$ edges, and let $T$ be its (infinite) universal covering tree.
    Then there exist explicit Hermitian matrices $\bA_1, \dots, \bA_n \in \CC^{|V(G)| \times |V(G)|}_{\herm}$ such that the adjacency operator $\mathfrak{a}_T$ of $T$ is unitarily equivalent to $\sum_{i = 1}^n \bA_i \otimes \mathfrak{r}_i$.
    Consequently,
    \[
        \lambda_{\max}(\mathfrak{a}_T) = \Lambda_{\rad}(\bm 0, \bA_1, \dots, \bA_n).
    \]
\end{theorem}

Further, \cite{BC-2019-EigenvaluesRandomLifts} showed that the adjacency matrices of large \emph{random lifts} \cite{ALMR-2001-RandomLifts} of $G$ converge strongly to this adjacency operator, and thus Lehner's formula can compute the limiting spectral radius of large random lifts of arbitrary finite graphs.
We omit the somewhat intricate definitions of these notions below; see the reference for details.
\begin{theorem}[\cite{BC-2019-EigenvaluesRandomLifts}]
    \label{thm:lift-rademacher}
    Let $G$ be a finite connected graph, let $T$ be its universal covering tree, and let $G^{(m)}$ be a uniformly random $m$-lift of $G$.
    Let $\bA_{G^{(m)}} \in \{0, 1\}^{m|V(G)| \times m|V(G)|}_{\sym}$ denote the adjacency matrix of $G^{(m)}$, and let $\bP^{(m)} \in \RR^{m|V(G)| \times m|V(G)|}_{\sym}$ be the orthogonal projection matrix to the orthogonal complement of the subspace of functions on $V(G^{(m)})$ that are constant on the fibers of the covering map of $G$ by $G^{(m)}$.
    Then, the following convergence in probability holds:
    \[ \lambda_{\max}(\bP^{(m)}\bA_{G^{(m)}}\bP^{(m)}) \to \lambda_{\max}(\mathfrak{a}_T) = \Lambda_{\rad}(\bm 0, \bA_1, \dots, \bA_n), \]
    where the $\bA_i$ are as in Theorem~\ref{thm:universal-cover-rademacher}.
\end{theorem}

In the simplest cases, Theorem~\ref{thm:universal-cover-rademacher} with Lehner's formula shows that the spectral edge and operator norm of the adjacency operator of the infinite $d$-regular tree are $2\sqrt{d - 1}$ (recovering an earlier result of Kesten \cite{Kesten-1958-SymmetricRandomWalksGroups}), and a special case of the main result of \cite{BC-2019-EigenvaluesRandomLifts} (recovering an earlier result of Friedman \cite{Friedman-2003-AlonConjecture}) shows that the spectral edge and spectral radius of non-trivial eigenvalues of a large random $d$-regular graph, which can be modeled as a random lift of the graph on a single vertex with $d$ self-loops, indeed converges in probability to $2\sqrt{d - 1}$.

\subsection{Main results}

In addition to the operator-theoretic motivation, the Lehner formulas may therefore also be seen as giving explicit finite-dimensional optimizations for determining the limiting behavior of various interesting sequences of random matrices and graphs of growing dimension.
Our goal is to show that these complicated-looking optimizations are in fact instances of \emph{semidefinite programming (SDP)} optimization problems and to derive various consequences of this observation.

\subsubsection{Semidefinite programming formulations}
We first study how the Lehner formulas can be expressed as SDPs.
In the semicircular case, we find that $\Lambda_{\sc}(\bA_0, \bA_1, \dots, \bA_n)$ can rather straightforwardly be written in two different ways as an SDP whose parameters depend on the $\bA_i$.
\begin{theorem}[``Small'' semicircular SDPs]
    \label{thm:sc-sdp-small}
    For any $\bA_0, \bA_1, \dots, \bA_n \in \CC^{d \times d}_{\herm}$,
    \begin{align}
      \Lambda_{\sc}(\bA_0, \bA_1, \dots, \bA_n)
      &= \left\{ \begin{array}{ll} \text{infimum of} & c \\ \text{subject to} & \left[\begin{array}{cc} c \bm I_d - \bA_0 - \Phi(\bZ) & \bm I_d \\ \bm I_d & \bZ \end{array}\right] \succeq \bm 0, \\ & c \in \RR, \bZ \in \CC^{d \times d}_{\herm} \end{array} \right\} \label{eq:sc-sdp-small-primal} \\
      &= \left\{ \begin{array}{ll} \text{maximum of} & \langle \bA_0, \bm S \rangle + 2\Re(\Tr(\bm T)) \\ \text{subject to} & \left[\begin{array}{cc} \bm S & -\bm T \\ -\bm T^{*} & \Phi(\bS) \end{array}\right] \succeq \bm 0, \\ & \bS \in \CC^{d \times d}_{\herm}, \bT \in \CC^{d \times d}, \\ & \Tr(\bS) = 1 \end{array}\right\} \label{eq:sc-sdp-small-dual} \\
      &= \max_{\substack{\bS \succeq \bm 0 \\ \Tr(\bS) = 1}} \left\{\langle \bA_0, \bS \rangle + 2\|\bS^{1/2} \Phi(\bS)^{1/2}\|_*\right\}. \label{eq:sc-sdp-small-dual-2}
    \end{align}
    The first two optimizations are Lagrangian duals of one another.
    The maximum in \eqref{eq:sc-sdp-small-dual} and \eqref{eq:sc-sdp-small-dual-2} is always achieved.
    If $\bigcap_{i = 1}^n \ker(\bA_i) = \{\bm 0\}$, then the infimum in \eqref{eq:sc-sdp-small-primal} is also achieved (as is the infimum in the original definition \eqref{eq:lehner-sc} of $\Lambda_{\sc}$).
    In \eqref{eq:sc-sdp-small-dual-2}, $\|\bY\|_* \colonequals \Tr((\bY\bY^{*})^{1/2})$ is the nuclear norm.
\end{theorem}

\begin{remark}[Prior work]
    Ramon van Handel in a private communication described to us also obtaining the dual form \eqref{eq:sc-sdp-small-dual-2} of the semicircular Lehner formula, using the Sion minimax theorem instead of SDP formulations.
    Some calculations with this dual appear in \cite{BBVH-2021-MatrixConcentrationFreeProbability, BCSVH-2024-MatrixConcentrationFreeProbability2}, though there it is only given in special cases where the matrix positivity constraints can be reduced to scalar ones because of special structure of the $\bA_i$.
\end{remark}

A simple variant of our approach to deriving Theorem~\ref{thm:sc-sdp-small} also gives the following alternative SDP formulation:
\begin{theorem}[``Large'' semicircular SDPs]
    \label{thm:sc-sdp-large}
    For any $\bA_0, \bA_1, \dots, \bA_n \in \CC^{d \times d}_{\herm}$,
    \begin{align*}
      &\Lambda_{\sc}(\bA_0, \bA_1, \dots, \bA_n) \\
      &= \left\{\begin{array}{ll} \text{minimum of} & c \\ \text{subject to} & \left[\begin{array}{ccccc} c\bm I_d - \bA_0 - \bZ & \bA_1 & \bA_2 & \cdots & \bA_n \\ \bA_1 & \bZ & \bm 0 & \cdots & \bm 0 \\ \bA_2 & \bm 0 & \bZ & \cdots & \bm 0 \\
                                                                                       \vdots & \vdots & \vdots & \ddots & \vdots \\ \bA_n & \bm 0 & \bm 0 & \cdots & \bZ \end{array}\right] \succeq \bm 0, \\ & c \in \RR, \bZ \in \CC^{d \times d}_{\herm} \end{array}\right\} \numberthis \label{eq:sc-sdp-large-primal} \\
      \intertext{}
      &= \left\{\begin{array}{ll} \text{maximum of} & \left\langle \left[\begin{array}{cccc} \bA_0 & -\bA_1 & \cdots & -\bA_n \\ -\bA_1 & \bm 0 & \cdots & \bm 0 \\ \vdots & \vdots & \ddots & \vdots \\ -\bA_n & \bm 0 & \cdots & \bm 0 \end{array}\right], \left[\begin{array}{cccc} \bm \Gamma^{[0, 0]} & \bm \Gamma^{[0, 1]} & \cdots & \bm \Gamma^{[0, n]} \\ \bm \Gamma^{[1, 0]} & \bm \Gamma^{[1, 1]} & \cdots & \bm\Gamma^{[1, n]} \\ \vdots & \vdots & \ddots & \vdots \\ \bm\Gamma^{[n, 0]} & \bm\Gamma^{[n, 1]} & \cdots & \bm \Gamma^{[n, n]} \end{array}\right]\right\rangle
\\ \text{subject to} & \bm\Gamma^{[i, j]} \in \CC^{d \times d} \text{ and } (\bm\Gamma^{[i, j]})^* = \bm\Gamma^{[j, i]} \text{ for all } 0 \leq i, j \leq n, \\ & \bm \Gamma = [\bm \Gamma^{[i, j]}]_{i, j = 0}^n \succeq \bm 0, \\ & \bm \Gamma^{[0, 0]} = \sum_{i = 1}^n \bm \Gamma^{[i, i]}, \\ & \Tr(\bm \Gamma^{[0, 0]}) = 1 \end{array}\right\}. \numberthis \label{eq:sc-sdp-large-dual}
    \end{align*}
    The two SDPs are Lagrangian duals of one another.
    The minimum in~\eqref{eq:sc-sdp-large-primal} and the maximum in~\eqref{eq:sc-sdp-large-dual} are both always achieved.
\end{theorem}
\begin{remark}[Sign of $\bA_i$]
    \label{rem:large-sdp-sign}
    We have that $\bD\bm\Gamma\bD$ for $\bD$ the diagonal orthogonal matrix $\bD = \mathsf{Diag}(\bm I_d, -\bm I_d, \dots, -\bm I_d)$ is feasible for the dual problem \eqref{eq:sc-sdp-large-dual} if and only if $\bm\Gamma$ is.
    Changing variables $\bm\Gamma \leftarrow \bD\bm\Gamma\bD$, we see that we may replace each $-\bA_i$ with $\bA_i$ for $i \in [n]$ in \eqref{eq:sc-sdp-large-dual} without changing the value of the problem; however, the resulting SDP is not the Lagrangian dual of \eqref{eq:sc-sdp-large-primal} anymore.
\end{remark}

\begin{remark}[Merits of large SDPs]
    While higher-dimensional, the large SDPs seem to have a few favorable properties compared to the small SDPs of Theorem~\ref{thm:sc-sdp-small}.
    First, we note that the issue with the infimum of \eqref{eq:sc-sdp-small-primal} possibly not being achieved does not occur here; the minimum of \eqref{eq:sc-sdp-large-primal} is always achieved, even when the infimum in the original Lehner formula is not achieved (for instance in the example of Remark~\ref{rem:inf-achieved} below).
    The dual large SDP \eqref{eq:sc-sdp-large-dual} also has the nice feature that, along with the previous Remark~\ref{rem:large-sdp-sign}, if we let $\mathcal{G} \subset \CC^{d(n + 1) \times d(n + 1)}_{\herm}$ be the set of feasible block matrices $\bm\Gamma$ in \eqref{eq:sc-sdp-large-dual} and define
    \[ \widetilde{\bA} \colonequals \left[\begin{array}{cccc} \bA_0 & \bA_1 & \cdots & \bA_n \\ \bA_1 & \bm 0 & \cdots & \bm 0 \\ \vdots & \vdots & \ddots & \vdots \\ \bA_n & \bm 0 & \cdots & \bm 0 \end{array}\right] \in \CC^{d(n + 1) \times d(n + 1)}_{\herm}, \]
    then this formulation shows that $\Lambda_{\sc}(\bA_0, \bA_1, \dots, \bA_n)$ is precisely the support function of the set $\sG$ (whose definition does not depend on the $\bA_i$) evaluated on $\widetilde{\bA}$.
\end{remark}

After proving these formulations in Section~\ref{sec:sc-sdp}, we give several supplementary results.
First, in Section~\ref{sec:slackness}, we describe the \emph{complementary slackness} relations between primal and dual optimizers.
This gives a partial alternative explanation of the fact that Lehner's original formula \eqref{eq:lehner-sc} can be restricted to those $\bZ$ for which $\bZ + \bA_0 + \sum_{i = 1}^n \bA_i\bZ^{-1}\bA_i$ is a multiple of the identity.
Second, in Section~\ref{sec:dist-to-opt}, we give a quantitative version of complementary slackness that is also the basis of our approach to spiked matrix models below.
Lastly, in Section~\ref{sec:explicit-bounds}, by producing explicit primal and dual certificates for the above SDPs, we give upper and lower bounds on $\Lambda_{\sc}$ that are simple explicit functions of the $\bA_i$, not requiring any optimization.

In the Rademacher case the situation at first appears quite different.
Here, Lehner's optimization is \emph{not} convex, and thus does not admit a direct formulation as an SDP:
\begin{theorem}[Non-convexity of Lehner's Rademacher objective]
    \label{thm:rad-non-convex}
    There exist $\bA_0, \bA_1, \dots, \bA_n \in \CC^{d \times d}_{\herm}$ such that the function $F: \{\bZ \in \CC^{d \times d}_{\herm}: \bZ \succ \bm 0\} \to \RR$ defined by
    \[ F(\bZ) \colonequals \lambda_{\max}\left(\bZ + \bA_0 + \sum_{i = 1}^n \frac{1}{2}\bZ^{1/2}((\bm I_d + 4(\bZ^{-1/2}\bA_i \bZ^{-1/2})^2)^{1/2} - \bm I_d)\bZ^{1/2}\right), \]
    the objective function of the optimization \eqref{eq:lehner-rad}, is not convex.
    These may further be chosen to satisfy $\bA_i \succ \bm 0$ for all $i = 0, 1, \dots, n$.
\end{theorem}

Yet, perhaps surprisingly, it \emph{is} in fact possible to give an SDP formulation of $\Lambda_{\rad}$ as well, though with a more involved construction involving additional auxiliary variables for each $\bA_i$.
\begin{theorem}[Rademacher SDP]
    \label{thm:rad-sdp}
    For any $\bA_0, \bA_1, \dots, \bA_n \in \CC^{d \times d}_{\herm}$,
    \begin{align}
      \Lambda_{\rad}(\bA_0, \bA_1, \dots, \bA_n)
      &= \left\{\begin{array}{ll} \text{minimum of} & c \\ \text{subject to} & \bZ \succeq \bm 0, \\ & \left[\begin{array}{cc} \bY_i & \bA_i \\ \bA_i & \bY_i + \bZ \end{array}\right] \succeq \bm 0 \text{ for all } i \in [n], \\ & \bZ + \bA_0 + \sum_{i = 1}^n \bY_i \preceq c \bm I_d, \\ & c \in \RR, \bZ, \bY_i \in \CC^{d \times d}_{\herm}, \text{ for all } i \in [n] \end{array}\right\} \label{eq:rad-sdp-primal} \\
      \intertext{}
      &= \left\{\begin{array}{ll} \text{maximum of} & \langle \bA_0, \bm \Delta \rangle + \sum_{i = 1}^n \langle \bA_i, \bR_i + \bR_i^{*} \rangle \\ \text{subject to} & \left[\begin{array}{cc} \bS_i & -\bR_i \\ -\bR_i^{*} & \bT_i \end{array}\right] \succeq \bm 0 \text{ for all } i \in [n], \\ & \bS_i + \bT_i = \bm\Delta \text{ for all } i \in [n], \\ & \bm\Delta - \sum_{i = 1}^n \bT_i \succeq \bm 0, \\ & \Tr(\bm\Delta) = 1, \\ & \bm\Delta, \bS_i, \bT_i \in \CC^{d \times d}_{\herm}, \bR_i \in \CC^{d \times d} \text{ for all } i \in [n] \end{array}\right\}. \label{eq:rad-sdp-dual}
    \end{align}
    The two SDPs are Lagrangian duals of one another.
    The minimum in~\eqref{eq:rad-sdp-primal} and the maximum in~\eqref{eq:rad-sdp-dual} are both always achieved.
\end{theorem}
\noindent
Since they are not relevant to our statistical applications below, we do not explore the structure of the SDPs in Theorem~\ref{thm:rad-sdp} as thoroughly here, though variants of the ideas we develop in the semicircular case in Section~\ref{sec:sc-sdp} can likely be mimicked in the Rademacher case as well.

Finally, for both the semicircular and Rademacher Lehner formulas, we discuss in Section~\ref{sec:numerical} why solving the associated SDPs with off-the-shelf convex optimization solvers in practice behaves much more stably than the numerical scheme Lehner proposes in \cite{Lehner-1999-NormFreeOperatorMatrixCoefficients}.
We suggest that, provided the SDPs are not prohibitively large, solving the SDPs should be the preferred numerical method for computing $\Lambda_{\sc}$ and $\Lambda_{\rad}$.

\subsubsection{Spiked matrix models}
\label{sec:intro-spiked}

For the second portion of our results, we return to Lehner's semicircular formula.
We focus on the following application of Theorem~\ref{thm:BCSVH} given by \cite{BCSVH-2024-MatrixConcentrationFreeProbability2} to \emph{spiked matrix models}, low-rank matrices perturbed by additive noise.
This result describes a quite general version of the much studied \emph{Baik--Ben Arous--P\'{e}ch\'{e} (BBP) transition} (named for the work of \cite{BBAP-2005-LargestEigenvalueSampleCovariance}, originating with observations of \cite{Johnstone-2001-LargestEigenvaluePCA}) in such models.
See, e.g., the survey treatment in \cite{Capitaine-2017-HDRRandomMatrices,JP-2018-PCASurvey} for general background on these phenomena.
For the sake of simplicity, we switch in these applications to focusing on real symmetric rather than complex Hermitian matrices.

\begin{theorem}[Special case of Theorem~3.1 of \cite{BCSVH-2024-MatrixConcentrationFreeProbability2}]
    \label{thm:BCSVH-BBP}
    Define the function
    \begin{equation}
        B(\theta) \colonequals \left\{\begin{array}{ll} 2 & \text{if } \theta \leq 1, \\ \theta + \theta^{-1} & \text{if } \theta > 1 \end{array}\right\}.
        \label{eq:bbp-B}
    \end{equation}
    Let $\bW = \bW^{(d)} \in \RR^{d \times d}_{\sym}$ be a sequence of real symmetric random matrices with centered jointly Gaussian entries.
    Write
    \[ \bW^{(d)} = \sum_{i = 1}^n g_i \bA_i^{(d)} \]
    for some $n = n(d)$, deterministic $\bA_i^{(d)} \in \RR^{d \times d}_{\sym}$, and independent $g_i \sim \sN(0, 1)$.
    Let $\bv = \bv^{(d)} \in \SS^{d - 1}$ be deterministic, let $\theta \geq 0$ not depend on $d$, and set
    \begin{align*}
      \bA_0^{(d)} &\colonequals \theta\bv\bv^{\top}, \\
      \bX^{(d)} &\colonequals \bA_0^{(d)} + \bW^{(d)}.
    \end{align*}
    Suppose that
    \begin{align}
        \EE[(\bW^{(d)})^2] &= \bm I_d \text{ for all } d \geq 1, \label{eq:BCSVH-BBP-1} \\
        v(\bW^{(d)}) &= o(\log(d)^{-3/2}) \text{ as } d \to \infty, \label{eq:BCSVH-BBP-2}
    \end{align}
    where $v(\cdot)$ is as defined in \eqref{eq:BBVH-v}.
    Then, the following hold, as $d \to \infty$:
    \begin{align}
      \Lambda_{\sc}(\bA_0^{(d)}, \bA_1^{(d)}, \dots, \bA_n^{(d)}) &\to B(\theta), \label{eq:bbp-Lambda}\\
      \lambda_{\max}(\bX^{(d)}) &\to B(\theta) \text{ in probability.} \label{eq:bbp-rand}
    \end{align}
    Further, let $\what{\bv} = \what{\bv}^{(d)}$ be the leading eigenvector of $\bX^{(d)}$.
    Then, under the same assumptions we have, as $d \to \infty$,
    \begin{equation}
        \label{eq:bbp-v}
        \langle \bv, \what{\bv} \rangle^2 \to B^{\prime}(\theta) = \left\{\begin{array}{ll} 0 & \text{if } \theta \leq 1, \\ 1 - \theta^{-2} & \text{if } \theta > 1 \end{array}\right\} \text{ in probability}.
    \end{equation}
\end{theorem}

In Section~\ref{sec:bbp-iso}, we give a new proof of \eqref{eq:bbp-Lambda} (from which \eqref{eq:bbp-rand} follows immediately using Theorem~\ref{thm:BCSVH}).
Our proof is somewhat more cumbersome than that of \cite{BCSVH-2024-MatrixConcentrationFreeProbability2}, but is arguably more natural: the proof of \cite{BCSVH-2024-MatrixConcentrationFreeProbability2} proves an upper bound on $\Lambda_{\sc}$ by constructing a suitable $\bZ$ for Lehner's formula \eqref{eq:lehner-sc}, but proves a lower bound by showing with a proof by contradiction that a $\bZ$ achieving a considerably smaller value cannot exist.
We instead follow a more standard style of argument in convex optimization by both producing an explicit $\bZ$ to prove an upper bound via \eqref{eq:lehner-sc} or equivalently the primal SDP \eqref{eq:sc-sdp-small-primal}, and producing an explicit \emph{dual certificate} $\bS$ to prove a lower bound via the reduced dual \eqref{eq:sc-sdp-small-dual-2}.

Our construction of a dual certificate $\bS$ is somewhat involved but is summarized in Definition~\ref{def:dual-cert}.
It is given not quite in closed form, but rather as the leading eigenmatrix of a certain completely positive operator on $\CC^{d \times d}_{\herm}$ constructed from $\theta, \bv$, and $\bA_1, \dots, \bA_n$.
It turns out that, beyond just being a useful proof technique, studying this dual near-optimizer yields its own substantive insights.
We make a heuristic argument concerning the dual optimizer below, leading to an intriguing new conjecture on spiked matrix models.

First, a well-known principle in convex optimization allows us to interpret optimal dual variables in such situations as measuring certain notions of sensitivity of the value of an optimization problem to changes in its parameters.
In our case, we have the following:
\begin{theorem}[Danskin's Theorem, Proposition B.22 of \cite{Bertsekas-1997-NonlinearProgramming}]
    \label{thm:sdp-diff}
    Fix $\bA_1, \dots, \bA_n \in \CC^{d \times d}_{\herm}$.
    Suppose that, for some $\bA_0 \in \CC^{d \times d}_{\herm}$, the dual problem~\eqref{eq:sc-sdp-small-dual-2} has a unique optimizer $\bS^{\star}$.
    Then the function
    \[
        \bM \mapsto \Lambda_{\sc}(\bM, \bA_1, \dots, \bA_n)
    \]
    is differentiable at $\bA_0$ as a function on $\CC^{d \times d}_{\herm}$, viewed as a real vector space, and its derivative with respect to the trace inner product is
    \[
        \frac{\partial\Lambda_{\sc}}{\partial \bA_0}(\bA_0, \bA_1, \dots, \bA_n) = \bS^{\star}.
    \]
\end{theorem}
\noindent
See also Section~5.6.3 of \cite{BV-2004-ConvexOptimization} for more discussion of such sensitivity interpretations.

We apply this observation to spiked matrix models as follows.
For $\bX = \bA_0 + \bW$, if $\what{\bv}$ is the leading eigenvector of $\bX$ and the leading eigenvalue $\lambda_{\max}(\bX)$ is simple, then a standard result of linear algebra (see, e.g., \cite{Magnus-1985-DifferentiatingEigenvaluesEigenvectors}) gives that the function $F(\bA_0) \colonequals \lambda_{\max}(\bA_0 + \bW)$ is differentiable at $\bA_0$ and
\[ \frac{\partial F}{\partial \bA_0}(\bA_0) = \what{\bv}\what{\bv}^{\top} \]
for $\what{\bv}$ the leading eigenvector as above.
Since by Theorem~\ref{thm:BCSVH} we expect to have
\[ \Lambda_{\sc}(\bA_0, \bA_1, \dots, \bA_n) \approx \EE \lambda_{\max}(\bX), \]
if we can take derivatives on either side of this approximation, the dual optimizer in \eqref{eq:sc-sdp-small-dual-2} is unique, and our construction $\bS^{\mathsf{est}}$ is close to that optimizer, we arrive at the heuristic prediction
\[ \EE[\what{\bv}\what{\bv}^{\top}] = \EE\left[\frac{\partial F}{\partial \bA_0}(\bA_0)\right] \approx \frac{\partial \Lambda_{\sc}}{\partial \bA_0}(\bA_0, \bA_1, \dots, \bA_n) = \bS^{\star}. \]
We formulate this precisely as the following first conjecture.
\begin{conjecture}
    \label{conj:spiked-matrix}
    In the setting of Theorem~\ref{thm:BCSVH-BBP}, in particular assuming~\eqref{eq:BCSVH-BBP-1} and~\eqref{eq:BCSVH-BBP-2}, suppose that $\theta > 1$, that $d_0 \geq 1$ is fixed, and that $\bR = \bR^{(d)} \in \RR^{d_0 \times d}$ are such that $d \cdot \bR\bS^{\star}\bR^{\top} \to \bm\Gamma \in \RR^{d_0 \times d_0}_{\sym}$ and $\|\bR\| = O(1)$ as $d \to \infty$.
    Suppose also that the dual Lehner formula~\eqref{eq:sc-sdp-small-dual-2} evaluated with $\bA_0 = \theta \bv\bv^{\top}$ has a unique optimizer $\bS^{\star}$ for each $d \geq 1$.
    Then, with the eigenvector $\what{\bv}$ defined such that $\langle \what{\bv}, \bv \rangle > 0$, we have the convergence as $d \to \infty$
    \[ d \cdot \bR \,\EE[\what{\bv}\what{\bv}^{\top}] \bR^{\top} \to \bm \Gamma. \]
    Further, the same holds with the same limit $\bm\Gamma$ if $\bS^{\star}$ is replaced by $\bS^{\mathsf{est}}$ as defined in Definition~\ref{def:dual-cert}, provided the conditions of Proposition~\ref{prop:Phi-irreducible-transfer} are satisfied so that $\bS^{\mathsf{est}}$ is well-defined.
\end{conjecture}
\noindent
This is a conservative conjecture about convergence of rescaled finite-dimensional projections of the second moment matrix $\EE[\what{\bv}\what{\bv}^{\top}]$; it is natural to hope that some notion of distance between $\EE[\what{\bv}\what{\bv}^{\top}]$ and $\bS^{\star}$ or $\bS^{\mathsf{est}}$ converges to zero as well, but it is unclear what the correct such notion would be and we do not venture such a formulation here.

We also propose a more statistically meaningful conjecture, deriving from this prediction of $\EE[\what{\bv}\what{\bv}^{\top}]$ a prediction of the covariance structure of fluctuations in $\what{\bv}$.
First, note that the eigenvector $\what{\bv}$ is only well-defined up to sign, so let us choose it such that $\langle \what{\bv}, \bv \rangle > 0$.
By \eqref{eq:bbp-v}, we expect this inner product to be close to $\sqrt{1 - \theta^{-2}}$, and thus we expect
\begin{equation}
    \label{eq:what-bv-exp}
    \EE[\what{\bv}] \approx \sqrt{1 - \theta^{-2}}\, \bv,
\end{equation}
and therefore
\[ \Cov[\what{\bv}] = \EE[\what{\bv}\what{\bv}^{\top}] - \EE[\what{\bv}]\EE[\what{\bv}]^{\top} \approx \bS^{\star} - (1 - \theta^{-2}) \bv\bv^{\top} \equalscolon \bm\Sigma. \]
So, the dual optimizer $\bS^{\star}$ in Lehner's formula encodes the covariance of the fluctuations of $\what{\bv}$ around the expectation given in \eqref{eq:what-bv-exp}, and thus contains information about the second-order behavior of this top eigenvector beyond the first-order convergence give in the BBP transition as in \eqref{eq:bbp-v}.

Absent any obvious obstruction, it is natural to further suppose that these fluctuations should be Gaussian, as occurs in the case of Wigner noise matrices (see, e.g., \cite{CDM-2018-DeformedWignerEigenvectorFluctuations,CK-2025-EigenvectorFluctuationsSpikedMatrix} for such results).
Under this assumption, the covariance fully determines the fluctuations, which we may conjecturally characterize as follows:
\begin{conjecture}[Spiked matrix model eigenvector fluctuations]
    \label{conj:spiked-matrix-fluct}
    In the setting of Theorem~\ref{thm:BCSVH-BBP}, in particular assuming~\eqref{eq:BCSVH-BBP-1} and~\eqref{eq:BCSVH-BBP-2}, suppose that $\theta > 1$, that $d_0 \geq 1$ is fixed, and that $\bR = \bR^{(d)} \in \RR^{d_0 \times d}$ are such that $d \cdot \bR\bm\Sigma\bR^{\top} \to \bm\Gamma \in \RR^{d_0 \times d_0}_{\sym}$ and $\|\bR\| = O(1)$ as $d \to \infty$, and $\bR \bv = \bR^{(d)}\bv^{(d)} = \bm 0$ for each $d \geq 1$.
    Suppose also that the dual Lehner formula~\eqref{eq:sc-sdp-small-dual-2} evaluated with $\bA_0 = \theta \bv\bv^{\top}$ has a unique optimizer $\bS^{\star}$ for each $d \geq 1$.
    Then, with the eigenvector $\what{\bv}$ defined such that $\langle \what{\bv}, \bv \rangle > 0$, we have the weak convergence as $d \to \infty$
    \[ \Law\left(\sqrt{d} \cdot \bR\left(\what{\bv} - \sqrt{1 - \theta^{-2}}\,\bv\right)\right) \to \sN(\bm 0, \bm\Gamma). \]
    Further, the same holds with the same limit $\bm\Gamma$ if $\bm\Sigma$ is replaced by $\bm\Sigma^{\mathsf{est}} \colonequals \bS^{\mathsf{est}} - (1 - \theta^{-2})\bv\bv^{\top}$, where $\bS^{\mathsf{est}}$ is as defined in Definition~\ref{def:dual-cert}, provided the conditions of Proposition~\ref{prop:Phi-irreducible-transfer} are satisfied so that $\bS^{\mathsf{est}}$ is well-defined.
\end{conjecture}
\noindent
As with Conjecture~\ref{conj:spiked-matrix}, we restrict our predictions here to only rescaled finite-dimensional marginal covariances, though in Figure~\ref{fig:spiked-covariance-table} we observe that, qualitatively, we clearly have $\bQ \Cov[\what{\bv}] \bQ^{\top} \approx \bQ\bm\Sigma\bQ^{\top}$ for $\bQ \in \RR^{(d - 1) \times d}$ merely projecting away from the $\bv$ direction, even in rather small dimension ($d = 40$ for those experiments).

\begin{remark}
    Note that we restrict our attention here to test directions that are orthogonal to the signal $\bv$ by making the assumption $\bR\bv = \bm 0$.
    While experiments and heuristic calculations indicate that fluctuations in the $\bv$ direction should be asymptotically Gaussian as well, the above reasoning does not give a correct estimate of their variance (or their covariance with other directions).
    The reason for this is as follows: we expect to have $\EE[\what{\bv}] = \sqrt{1 - \theta^{-2}}\,\bv + \bm\delta$ for an error vector $\bm\delta$ with $\|\bm\delta\| = O(d^{-1/2})$, so that $\langle \bm\delta, \bv \rangle = O(d^{-1})$ and thus $\EE[\langle \what{\bv}, \bv \rangle] = \sqrt{1 - \theta^{-2}} + O(d^{-1})$.
    But, since we also expect to have $\Var[\langle \what{\bv}, \bv \rangle] = O(d^{-1})$, when we compute $\Var[\langle \what{\bv}, \bv \rangle] = \EE[\langle \what{\bv}, \bv \rangle^2] - (\EE[\langle \what{\bv}, \bv \rangle])^2$, the error term in $\EE[\langle \what{\bv}, \bv \rangle]$ makes a leading order contribution.
    Thus, to predict the magnitude of fluctuations in the $\bv$ direction, we would need to identify the second-order $O(d^{-1})$ behavior of the expected correlation $\EE[\langle \what{\bv}, \bv \rangle]$ of the signal $\bv$ and its estimate $\what{\bv}$.
    This does not appear to be known in the literature for general noise models of the kind we consider, and also, curiously, does not seem to be encoded in the objects in our SDP analysis the way that $\EE[\what{\bv}\what{\bv}^{\top}]$ is.
\end{remark}

We give further discussion of and numerical evidence for Conjecture~\ref{conj:spiked-matrix-fluct} in Section~\ref{sec:conj-evidence}.
To the best of our knowledge, eigenvector fluctuations in spiked matrix models with general correlated noise have not been considered before in the literature, and we hope these first steps will motivate further study of their interesting behavior.
Our conjecture differs from previous results on such fluctuations with noise given by Wigner matrices (having i.i.d.\ entries up to symmetry) in two important regards:
\begin{enumerate}
\item Unlike the analogous isotropic Gaussian limits in prior work such as \cite{CH-2012-FluctuationsEigenvectorSpikedMatrix,LCC-2024-GaussianFluctuationsEigenvectors,CK-2025-EigenvectorFluctuationsSpikedMatrix} for the case of Gaussian Wigner matrices (drawn from the Gaussian orthogonal or unitary ensembles), the fluctuations of $\what{\bv}$ can be anisotropic, with the projected covariances $\Gamma$ depending non-trivially on the structure of $(\bv, \bA_1, \dots, \bA_n)$, and thus behaving non-universally across different isotropic noise models.
\item Unlike the case of Wigner matrix noise with non-Gaussian entry distribution as studied by~\cite{CDM-2018-DeformedWignerEigenvectorFluctuations,CK-2025-EigenvectorFluctuationsSpikedMatrix}, we do \emph{not} find that the case of localized $\bv$ (with its mass concentrated on a few entries) differs from the case of delocalized $\bv$ (with its mass distributed more evenly over all entries): even if, say, $\bv = \be_1$ and $\bR = \be_1\be_1^{\top}$ measures fluctuations of $\what{\bv}$ in the $\bv$ direction, Conjecture~\ref{conj:spiked-matrix} still makes a Gaussian prediction.
    We believe this Gaussianity can only break down when the noise matrix $\bW^{(d)}$ itself becomes biased in the $\bv$ direction, say by having $\sigma_*(\bW^{(d)}) = \Omega(1)$.
\end{enumerate}

\section{Preliminaries}

\subsection{Notation}

We write $\CC^{d \times d}_{\herm}$ for the set of $d \times d$ complex Hermitian matrices and $\RR^{d \times d}_{\sym}$ for the set of $d \times d$ real symmetric matrices.
We write $\langle \bX, \bY \rangle \colonequals \Tr(\bX^{*} \bY) = \sum_{i, j} \overline{X_{ij}}Y_{ij}$ for the trace inner product of matrices with the same dimensions, and $\|\bX\|_F \colonequals \sqrt{\langle \bX, \bX \rangle}$ for the Frobenius norm, the norm associated to this inner product.
We write $\|\bX\|$ for the operator norm and $\|\bX\|_* \colonequals \Tr((\bX\bX^*)^{1/2})$ for the nuclear or Schatten-1 norm.

\subsection{Linear algebra}

We review the following well-known constructions and results.
\begin{proposition}[Schur complement criterion]
    \label{prop:schur}
    Suppose $\bA \in \CC^{m \times m}_{\herm}$, $\bB \in \CC^{m \times n}$, $\bC \in \CC^{n \times n}_{\herm}$.
    Then, we have
    \[ \left[\begin{array}{cc} \bA & \bB \\ \bB^{*} & \bC \end{array}\right] \succ \bm 0 \]
    if and only if
    \[ \bA \succ \bm 0 \,\,\,\,\text{ and }\,\,\,\, \bC - \bB^{*}\bA^{-1}\bB \succ \bm 0. \]
    The matrix $\bC - \bB^{*}\bA^{-1}\bB$ is called the \emph{Schur complement} of the original block matrix with respect to the block $\bA$.
\end{proposition}

\begin{proposition}[Proposition 1.3.2 of \cite{Bhatia-2009-PositiveDefiniteMatrices}]
    \label{prop:contraction}
    Suppose $\bA, \bB \in \CC^{n \times n}_{\herm}$ have $\bA, \bB \succeq \bm 0$ and $\bT \in \CC^{n \times n}$.
    Then,
    \[ \left[\begin{array}{cc} \bA & \bT \\ \bT^{*} & \bB \end{array}\right] \succeq \bm 0 \]
    if and only if there is some $\bR$ with $\|\bR\| \leq 1$ such that $\bT = \bA^{1/2}\bR\bB^{1/2}$.
\end{proposition}

\begin{definition}[Spectral functions]
    For $\bX \in \CC^{n \times n}_{\herm}$, let $\bX = \sum_{i = 1}^n \lambda_i(\bX) \bu_i\bu_i^{*}$ be the spectral decomposition.
    Then, for $f: \RR \to \RR$ we write $f(\bX) \colonequals \sum_{i = 1}^n f(\lambda_i(\bX)) \bu_i\bu_i^{*}$.
    We say that:
    \begin{itemize}
    \item $f$ is \emph{operator monotone} if $f(\bX) \preceq f(\bY)$ for all $n \geq 1$ and $\bX, \bY \in \CC^{n \times n}_{\herm}$ with $\bX \preceq \bY$.
    \item $f$ is \emph{operator convex} if $f(t\bX + (1 - t)\bY) \preceq tf(\bX) + (1 - t)f(\bY)$ for all $t \in [0, 1]$, $n \geq 1$, and $\bX, \bY \in \CC^{n \times n}_{\herm}$.
    \end{itemize}
    Further, we say $f$ is operator monotone or convex on a subset $I \subseteq \RR$ if the corresponding definition holds for all $\bX, \bY$ with $\spec(\bX), \spec(\bY) \subseteq I$.
\end{definition}

\subsubsection{Non-commutative perspectives}
\label{sec:perspective}

The following notion that has been studied in operator theory and linear algebra turns out to be closely related to the structure of Lehner's formulas.

\begin{definition}[Non-commutative perspective]
    \label{def:perspective}
    For $\bA, \bZ \in \CC^{n \times n}_{\herm}$ with $\bZ \succ \bm 0$, we define the \emph{non-commutative $f$-perspective} transformation as
    \[ P_f(\bA, \bZ) \colonequals \bZ^{1/2} f\left(\bZ^{-1/2}\bA \bZ^{-1/2}\right) \bZ^{1/2}. \]
\end{definition}

\begin{proposition}[\cite{ENE-2011-PerspectiveMatrixConvex}]
    \label{prop:perspective-convex}
    If $f$ is operator convex on an open interval $I \subseteq \RR$ (possibly infinite on one or both sides), then $P_f$ is a convex function on the set of input pairs
    \[ D_{I} \colonequals \{(\bA, \bZ) \in (\CC^{d \times d}_{\herm})^2: \bZ \succ \bm 0, \spec(\bZ^{-1/2}\bA\bZ^{-1/2}) \subseteq I\}. \]
\end{proposition}
\noindent
See also \cite{Effros-2009-MatrixConvexityQuantumInequalities,EH-2014-NonCommutativePerspectives,FSP-2019-SDPMatrixLogarithm} for related discussion and special cases of the above.

\subsubsection{Completely positive operators}
\label{sec:cp}

We will see that an important role in many of our derivations is played by various completely positive operators associated to the parameters $\bA_1, \dots, \bA_n$.
We give a few basic parts of the theory surrounding such operators here.
\begin{definition}[Completely positive maps]
    \label{def:cp}
    Let $\sL: \CC^{d \times d} \to \CC^{d \times d}$ be a linear map.
    We say that $\sL$ is \emph{positive} if $\sL(\bS) \succeq \bm 0$ for every $\bS \succeq \bm 0$.
    We say that $\sL$ is \emph{completely positive} if, for every $m \geq 1$ and every positive semidefinite block matrix $[\bS_{ij}]_{i, j = 1}^m \succeq \bm 0$, we have
    \[
        [\sL(\bS_{ij})]_{i, j = 1}^m \succeq \bm 0.
    \]
\end{definition}

\begin{definition}[Irreducibility and primitivity]
    \label{def:cp-irred}
    Suppose $\sL: \CC^{d \times d} \to \CC^{d \times d}$ is completely positive and has the form
    \[
        \sL(\bS) = \sum_{i = 1}^n \bB_i \bS \bB_i^*,
    \]
    sometimes called a \emph{Kraus representation}.
    We say that $\sL$ is \emph{irreducible} if there is no nonzero proper subspace $L \subset \CC^d$ such that $\bB_i L \subseteq L$ for every $i \in [n]$.
    We say that $\sL$ is \emph{primitive} if there is some $m \geq 1$ such that $\sL^m(\bS) \succ \bm 0$ for every nonzero $\bS \succeq \bm 0$.
\end{definition}

\begin{theorem}[Quantum Perron-Frobenius theorem \cite{EHK-1978-SpectralPropertiesPositiveMaps}]
    \label{thm:quantum-pf}
    Let $\sL: \CC^{d \times d} \to \CC^{d \times d}$ be a nonzero completely positive map, and let $r(\sL)$ denote its spectral radius as a linear operator on $\CC^{d \times d}$.
    Then there is some nonzero $\bS \succeq \bm 0$ such that $\sL(\bS) = r(\sL)\bS$.
    If $\sL$ is irreducible, then $r(\sL) > 0$, this eigenmatrix may be chosen positive definite, and it is the unique positive semidefinite eigenmatrix associated to the eigenvalue $r(\sL)$ up to rescaling.
    In this case, the eigenvalue $r(\sL)$ is algebraically simple.
    If $\sL$ is primitive, then every other eigenvalue $\lambda$ of $\sL$ satisfies $|\lambda| < r(\sL)$.
\end{theorem}

\subsection{Convex optimization}

We will use standard general results on convex optimization without recalling the details below.
In particular, we will use Slater's condition for strong Lagrangian duality and the Karush-Kuhn-Tucker~(KKT) conditions, in particular the complementary slackness condition, for convex optimization problems as presented in \cite[Chapter 5]{BV-2004-ConvexOptimization}.
We apply these to SDPs; details of this application are given in \cite{VB-1996-SDP}, for instance.
We will work with SDPs involving complex numbers and Hermitian matrices, which is a slightly non-standard setting, but is easily converted to standard real SDP; see \cite[Exercise 4.42]{BV-2004-ConvexOptimization} or \cite{HZ-2007-ComplexMatrixDecomposition} for details.

\section{Reformulation via non-commutative perspectives}
\label{sec:perspective-reform}

We first make an observation that is common to both of Lehner's formulas.
We have mentioned in Section~\ref{sec:perspective} that Lehner's formulas are related to non-commutative perspective transformations.
Indeed, it turns out that Lehner's two formulas can be phrased in a unified form using the following two choices of $f$ in the non-commutative perspective:
\begin{align*}
  f_{\sc}(t) &\colonequals t^2, \\
  f_{\rad}(t) &\colonequals \frac{-1 + \sqrt{1 + 4t^2}}{2}.
\end{align*}
The following is immediate from the statement of the formulas.
\begin{proposition}[Lehner formulas via non-commutative perspectives]
    \label{prop:lehner-perspective}
    For any $\bA_0, \bA_1, \dots, \bA_n \in \CC^{d \times d}_{\herm}$, we have
    \begin{align*}
      \Lambda_{\sc}(\bA_0, \bA_1, \dots, \bA_n)
      &= \inf_{\bZ \succ \bm 0} \lambda_{\max}\left(\bZ + \bA_0 + \sum_{i = 1}^n P_{f_{\sc}}(\bA_i, \bZ)\right), \\
      \Lambda_{\rad}(\bA_0, \bA_1, \dots, \bA_n)
      &= \inf_{\bZ \succ \bm 0} \lambda_{\max}\left(\bZ + \bA_0 + \sum_{i = 1}^n P_{f_{\rad}}(\bA_i, \bZ)\right).
    \end{align*}
\end{proposition}

\begin{remark}[Relation to free probability transforms]
    The functions $f(t)$ are closely related to the spectral distributions of the associated operators; indeed, in both cases, $f(t) / t$ is \emph{$R$-transform} of this distribution, the Wigner semicircle distribution supported on $[-2, 2]$ for the case of $f_{\sc}(t)$ and the Rademacher distribution $\Unif(\{\pm 1\})$ for the case of $f_{\rad}(t)$.
\end{remark}

Per Proposition~\ref{prop:perspective-convex}, it is then of interest to understand whether $f_{\sc}$ and $f_{\rad}$ are operator convex (it is easily verified that both are scalar convex functions).
The answers are as follows:
\begin{proposition}[Example V.1.3 of \cite{Bhatia-2013-MatrixAnalysis}]
    $f_{\sc}$ is operator convex on all of $\RR$.
\end{proposition}
\begin{proof}
    We calculate directly: for $t \in [0, 1]$,
    \begin{align*}
      t\bX^2 + (1 - t)\bY^2 - \left(t\bX + (1 - t)\bY\right)^2 = t(1 - t)(\bX - \bY)^2 \succeq \bm 0,
    \end{align*}
    completing the proof.
\end{proof}

\begin{proposition}
    \label{prop:frad}
    $f_{\rad}$ is not operator convex on any non-empty open interval.
    Further, there exists $\bA \succeq \bm 0$ such that $p(\bZ) \colonequals P_{f_{\rad}}(\bA, \bZ)$ is not a convex function $p: \RR^{d \times d}_{\succeq \bm 0} \to \RR^{d \times d}_{\succeq \bm 0}$ with respect to the Loewner partial order.
\end{proposition}
\begin{proof}
    As mentioned earlier, $f_{\rad}$ is strictly convex (as a scalar function), and is also smooth, so by \cite[Theorem 2.3]{HT-2007-DifferentialMatrixConvexFunctions}, it is operator convex on $2 \times 2$ Hermitian matrices with spectra in some open interval $I$ if and only if, for all $t \in I$, we have
    \[ \left[\begin{array}{cc} \frac{1}{2}f^{(2)}_{\rad}(t) & \frac{1}{6}f^{(3)}_{\rad}(t) \\ \frac{1}{6}f^{(3)}_{\rad}(t) & \frac{1}{24} f^{(4)}_{\rad}(t) \end{array}\right] \succeq \bm 0. \]
    However, a direct computation gives that the determinant of the matrix on the left-hand side is $-1 / (1 + 4t^2)^5 < 0$ for all $t \in \RR$, and the first result follows.

    For the second part, a numerical calculation confirms that if we set
    \[ \bA \colonequals \left[\begin{array}{cc} 2 & 1 \\ 1 & 2 \end{array}\right], \hspace{0.5cm} \bX \colonequals \left[\begin{array}{cc} 5 & 2 \\ 2 & 2 \end{array}\right], \hspace{0.5cm} \bY \colonequals \left[\begin{array}{cc} 9 & 8 \\ 8 & 11 \end{array}\right], \]
    then
    \[ \lambda_2\left(\frac{1}{2}(P_{f_{\rad}}(\bA, \bX) + P_{f_{\rad}}(\bA, \bY)) - P_{f_{\rad}}\left(\bA, \frac{1}{2}(\bX + \bY)\right)\right) < -10^{-3} < 0, \]
    and $\bA, \bX, \bY \succ \bm 0$.
\end{proof}

Later in Section~\ref{sec:rad}, we will show, as should now be expected, that the entire scalar-valued objective function of Lehner's formula in the Rademacher case is also not a convex function on the feasible set $\{\bZ \succ \bm 0\}$.

\section{Semidefinite program formulations for semicircular formula}
\label{sec:sc-sdp}

We draw a distinction between two different SDP reformulations of the semicircular Lehner formula, calling them ``small'' and ``large'' according to the dimension of auxiliary decision variables that must be introduced in the optimization.
The reason that there are two different ways to do this is that we may write a bound on the objective function $\lambda_{\max}(\bA_0 + \bZ + \sum_{i = 1}^n\bA_i \bZ^{-1}\bA_i)$ first as $\bA_0 + \bZ + \sum_{i = 1}^n\bA_i \bZ^{-1}\bA_i \preceq c\bm I_d$, and then may express this as a semidefinite constraint by using the ``Schur complement trick'' of Proposition~\ref{prop:schur} on the third term.
However, we could also first swap the roles of $\bZ$ and $\bZ^{-1}$ in Lehner's formula and use the Schur complement trick on the \emph{second} term.
As our results describe, the resulting SDPs are of course closely related but subtly different in certain details of their behavior.

It will be useful to define the operator
\begin{equation}
    \label{eq:Phi}
    \Phi(\bS) = \Phi(\bS; \bA_1, \dots, \bA_n) \colonequals \sum_{i = 1}^n \bA_i \bS\bA_i.
\end{equation}
This depends on the parameters $\bA_i$, but we omit this dependence in the notation for the sake of simplicity.
Note that, in the Gaussian series model $\bX = \bA_0 + \sum_{i = 1}^n g_i \bA_i$ discussed in Section~\ref{sec:appl}, we also have the intrinsic description
\[ \Phi(\bS) = \EE[(\bX - \EE \bX) \bS (\bX - \EE \bX)]. \]

These operators will play an important role throughout, so let us establish a few of their basic properties, which are simple to verify.
\begin{proposition}
    \label{prop:Phi-properties}
    For any $\bA_1, \dots, \bA_n \in \CC^{d \times d}_{\herm}$ and $\Phi: \CC^{d \times d}_{\herm} \to \CC^{d \times d}_{\herm}$ as defined in \eqref{eq:Phi}, the following hold.
    \begin{enumerate}
    \item $\Phi$ is linear.
    \item $\Phi$ is self-adjoint with respect to the Frobenius inner product: $\langle \Phi(\bS), \bT \rangle = \langle \bS, \Phi(\bT)\rangle$.
    \item $\Phi$ is completely positive (Definition~\ref{def:cp}).\footnote{In fact, every linear and completely positive $\Phi: \CC^{d \times d}_{\herm} \to \CC^{d \times d}_{\herm}$ must admit the closely related \emph{Kraus representation} $\Phi(\bS) = \sum_{i = 1}^n \bV_i^* \bS\bV_i$ for some $\bV_i \in \CC^{d \times d}$ \cite[Theorem 3.1.1]{Bhatia-2009-PositiveDefiniteMatrices}.}
    \item For all $\bS$, $\Tr(\Phi(\bS)) = \langle \sum_{i = 1}^n \bA_i^2, \bS\rangle$.
    \end{enumerate}
\end{proposition}

\subsection{Small semicircular SDP: Proof of Theorem \ref{thm:sc-sdp-small}}

We mention the following technical point before giving additional details on the statement.
\begin{proposition}
    \label{prop:Z-pd}
    Any $\bZ \in \CC^{d \times d}_{\herm}$ feasible for \eqref{eq:sc-sdp-small-primal} has $\bZ \succ \bm 0$.
\end{proposition}
\begin{proof}
    We have $\bZ \succeq \bm 0$ as $\bZ$ is a diagonal block of a matrix constrained to be positive semidefinite.
    Suppose $\bZ\bx = \bm 0$.
    Then, for sufficiently small $\epsilon > 0$ we would have
    \[ \left[\begin{array}{cc} -\epsilon \bx \\ \bx \end{array}\right]^*\left[\begin{array}{cc} c \bm I_d - \bA_0 - \Phi(\bZ) & \bm I_d \\ \bm I_d & \bZ \end{array}\right]\left[\begin{array}{cc} -\epsilon \bx \\ \bx \end{array}\right] \leq \epsilon^2 \|\bx\|^2\|c\bm I_d - \bA_0 - \Phi(\bZ)\| - 2\epsilon\|\bx\|^2 < 0, \]
    a contradiction.
\end{proof}

\begin{remark}[Existence of minimizer in Lehner formula]
    \label{rem:inf-achieved}
    It is not unconditionally true that the infimum in \eqref{eq:sc-sdp-small-primal} or the original semicircular Lehner formula \eqref{eq:lehner-sc} is achieved.
    Indeed, consider the example $d = 2$, $n = 1$, and
    \[ \bA_0 = \left[\begin{array}{rr} 1 & 0 \\ 0 & -2\end{array}\right], \,\,\,\,\, \bA_1 = \left[\begin{array}{rr} 0 & 0 \\ 0 & 1\end{array}\right]. \]
    We then have
    \[ \Phi(\bZ) = \bA_1\bZ\bA_1 = \left[\begin{array}{cc} 0 & 0 \\ 0 & Z_{22} \end{array}\right]. \]

    Since $\Phi(\bZ) + \bZ^{-1} \succ \bm 0$ for all $\bZ \succ \bm 0$ we have $\lambda_{\max}(\bA_0 + \Phi(\bZ) + \bZ^{-1}) > 1$.
    On the other hand, taking
    \[ \bZ = \left[\begin{array}{cc} z & 0 \\ 0 & 1 \end{array}\right] \]
    with $z \to \infty$ has $\lambda_{\max}(\bA_0 + \Phi(\bZ) + \bZ^{-1}) = 1 + z^{-1}$, so in fact $\Lambda_{\sc}(\bA_0, \bA_1) = 1$ and this value is not achieved by any $\bZ$ in \eqref{eq:lehner-sc}.
    The same holds in the SDP in \eqref{eq:sc-sdp-small-primal}, since every feasible $\bZ$ in the SDP still has $\bZ \succ \bm 0$ by Proposition~\ref{prop:Z-pd}.

    This example also shows that we cannot in general relax the condition $\bigcap_{i = 1}^n \ker(\bA_i) = \{\bm 0\}$ to $\bigcap_{i = 0}^n \ker(\bA_i) = \{\bm 0\}$, which might look more natural since if it does not hold then the entire operator $\bX = \bA_0 \otimes \mathfrak{1} + \sum_{i = 1}^n \bA_i \otimes \mathfrak{s}_i$ would be isometrically equivalent to one where the $\bA_i$ are of lower dimension.
    The former condition is more relevant since it controls the ``coercivity'' properties of the $\Phi$ operator.
\end{remark}

\begin{proof}[Proof of Theorem \ref{thm:sc-sdp-small}]
    For the sake of brevity, we omit the domain constraints like $c \in \RR$ and $\bZ \in \CC^{d \times d}_{\herm}$ below; they are always as written in the statement of the Theorem.
    Starting from the Lehner formula \eqref{eq:lehner-sc}, we change variables $\bZ \leftarrow \bZ^{-1}$ to rewrite
    \begin{align*}
      \Lambda_{\sc}(\bA_0, \bA_1, \dots, \bA_n)
      &= \inf_{\bZ \succ \bm 0} \lambda_{\max}(\bA_0 + \Phi(\bZ) + \bZ^{-1}) \\
      &= \left\{\begin{array}{ll} \text{infimum of} & c \\ \text{subject to} & \bZ \succ \bm 0, \\ & c\bm I_d - \bA_0 - \Phi(\bZ) - \bZ^{-1} \succ \bm 0 \end{array}\right\}.
    \end{align*}
    By Proposition~\ref{prop:schur},
    \[
        \bZ \succ \bm 0 \text{ and } c\bm I_d - \bA_0 - \Phi(\bZ) - \bZ^{-1} \succ \bm 0
        \,\, \text{ if and only if } \,\,
        \left[\begin{array}{cc} c \bm I_d - \bA_0 - \Phi(\bZ) & \bm I_d \\ \bm I_d & \bZ \end{array}\right] \succ \bm 0.
  \]
  Thus we may further rewrite
  \begin{align*}
      \Lambda_{\sc}(\bA_0, \bA_1, \dots, \bA_n)
      &= \left\{\begin{array}{ll} \text{infimum of} & c \\ \text{subject to} & \left[\begin{array}{cc} c \bm I_d - \bA_0 - \Phi(\bZ) & \bm I_d \\ \bm I_d & \bZ \end{array}\right] \succ \bm 0 \end{array} \right\} \numberthis \label{eq:sc-small-sdp-1-2}
                                       \intertext{and so we also have}
                                                    &\geq \left\{\begin{array}{ll} \text{infimum of} & c \\ \text{subject to} & \left[\begin{array}{cc} c \bm I_d - \bA_0 - \Phi(\bZ) & \bm I_d \\ \bm I_d & \bZ \end{array}\right] \succeq \bm 0 \end{array}\right\} \numberthis \label{eq:sc-small-sdp-2-2}
  \end{align*}
  by relaxing the feasible set.
  On the other hand, suppose that $(c, \bZ)$ are feasible for \eqref{eq:sc-small-sdp-2-2}.
  By Proposition~\ref{prop:Z-pd} we have $\bZ \succ \bm 0$.
  So, for $(c, \bZ)$ feasible for \eqref{eq:sc-small-sdp-2-2}, again by Proposition~\ref{prop:schur}, $(c + \epsilon, \bZ)$ is feasible for \eqref{eq:sc-small-sdp-1-2} for any $\epsilon > 0$.
  Thus, the values of \eqref{eq:sc-small-sdp-1-2} and \eqref{eq:sc-small-sdp-2-2} are in fact equal.

  The second SDP \eqref{eq:sc-sdp-small-dual} is the Lagrangian dual of \eqref{eq:sc-sdp-small-primal} with the maximum \emph{a priori} replaced by a supremum; we give the routine calculation demonstrating this in Appendix~\ref{app:dual-small}.
  Note also that the first SDP is strictly feasible, as we see by taking $\bZ = \bm I_d$ and $c > 0$ sufficiently large.
  Therefore Slater's condition holds, and it follows that strong duality holds (i.e., \eqref{eq:sc-sdp-small-primal} and \eqref{eq:sc-sdp-small-dual} have the same optimal value) and that the maximum in \eqref{eq:sc-sdp-small-dual} is achieved (and so it is indeed a maximum and not just a supremum).

  Now consider the sufficient condition for the infimum in \eqref{eq:sc-sdp-small-primal} to also be achieved: suppose that $\bigcap_{i = 1}^n \ker(\bA_i) = \{\bm 0\}$.
  We always have $\Phi(\bZ) \succeq \bm 0$ for all $\bZ \succeq \bm 0$ by complete positivity of $\Phi$.
  Under the above assumption we claim that we also have $\Phi(\bZ) \neq \bm 0$ whenever $\bZ \succeq \bm 0$ and $\bZ \neq \bm 0$.
  Indeed, if $\bm 0 = \Phi(\bZ) = \sum_{i = 1}^n (\bA_i\bZ^{1/2})(\bA_i \bZ^{1/2})^*$ then $\bA_i\bZ^{1/2} = \bm 0$ for all $i \in [n]$.
  Then, every column of $\bZ^{1/2}$ belongs to $\bigcap_{i = 1}^n \ker(\bA_i) = \{\bm 0\}$, so $\bZ^{1/2} = \bm 0$ and $\bZ = \bm 0$.
  Thus, $\lambda_{\max}(\Phi(\bZ))$ is a strictly positive function on the compact set $\{\bZ: \bZ \succeq \bm 0, \lambda_{\max}(\bZ) = 1\}$, and so is bounded below on this set.
  Therefore, by linearity of $\Phi$, we also have $\lambda_{\max}(\Phi(\bZ)) \geq \alpha \lambda_{\max}(\bZ)$ for some $\alpha > 0$, for all $\bZ \succeq \bm 0$.
  In particular, the objective function $\lambda_{\max}(\bA_0 + \Phi(\bZ) + \bZ^{-1})$ diverges both as $\bZ$ approaches the boundary of the positive semidefinite cone and as $\lambda_{\max}(\bZ) \to \infty$, and thus the infimum is achieved in both the original \eqref{eq:lehner-sc} and in \eqref{eq:sc-sdp-small-primal}.

  Finally, for the third optimization \eqref{eq:sc-sdp-small-dual-2}, we eliminate the variable $\bm T$ from the second SDP.
  Fix $\bm S \succeq \bm 0$ with $\Tr(\bm S) = 1$.
  By Proposition~\ref{prop:contraction}, the constraint
  \[
      \left[\begin{array}{cc} \bm S & -\bm T \\ -\bm T^{*} & \Phi(\bm S) \end{array}\right] \succeq \bm 0
  \]
  is equivalent to the existence of $\bm R \in \CC^{d \times d}$ with $\|\bm R\| \leq 1$ and $\bm T = \bm S^{1/2} \bm R \, \Phi(\bm S)^{1/2}$.
  Using the dual form of the nuclear norm,
    \[
      \|\bY\|_*
      = \sup_{\|\bm R\| \leq 1} \Re(\langle\bR, \bY\rangle),
    \]
    we get, for fixed $\bS$ as above,
    \begin{align*}
      \sup\left\{2\Re(\Tr(\bm T)) : \left[\begin{array}{cc} \bm S & -\bm T \\ -\bm T^{*} & \Phi(\bm S) \end{array}\right] \succeq \bm 0\right\}
      &= 2\sup_{\|\bm R\| \leq 1} \Re(\Tr(\bm R^* \, \Phi(\bm S)^{1/2}\bm S^{1/2})) \\
      &= 2\|\Phi(\bm S)^{1/2}\bm S^{1/2}\|_* \\
      &= 2\|\bm S^{1/2}\Phi(\bm S)^{1/2}\|_*,
    \end{align*}
    giving the third formula \eqref{eq:sc-sdp-small-dual-2}.
    The last equality uses that $\bY\bY^{*}$ and $\bY^{*}\bY$ have the same nonzero eigenvalues, so $\|\bY\|_* = \|\bY^*\|_*$.
\end{proof}

\subsection{Large semicircular SDP: Proof of Theorem \ref{thm:sc-sdp-large}}

\begin{proof}[Proof of Theorem \ref{thm:sc-sdp-large}]
    Let
    \[
      \bH(c, \bZ)
      \colonequals
      \left[\begin{array}{ccccc} c\bm I_d - \bA_0 - \bZ & \bA_1 & \bA_2 & \cdots & \bA_n \\ \bA_1 & \bZ & \bm 0 & \cdots & \bm 0 \\ \bA_2 & \bm 0 & \bZ & \cdots & \bm 0 \\
                                                                                     \vdots & \vdots & \vdots & \ddots & \vdots \\ \bA_n & \bm 0 & \bm 0 & \cdots & \bZ \end{array}\right].
    \]
    We first identify the first SDP \eqref{eq:sc-sdp-large-primal} with Lehner's formula \eqref{eq:lehner-sc}.
    Suppose first that $\bZ \succ \bm 0$.
    Applying Proposition~\ref{prop:schur} to the lower-right block-diagonal matrix
    \[
      \mathsf{Diag}(\bZ, \dots, \bZ) = \bm I_n \otimes \bZ
    \]
    yields
    \[
      \bH(c, \bZ) \succ \bm 0
      \,\, \text{ if and only if } \,\,
      c\bm I_d - \bA_0 - \bZ - \sum_{i = 1}^n \bA_i \bZ^{-1} \bA_i \succ \bm 0.
    \]
    Therefore, if one further restricts \eqref{eq:sc-sdp-large-primal} to require $\bH(c, \bZ) \succ \bm 0$, its optimal value is exactly the value $\Lambda_{\sc}(\bA_0, \bA_1, \dots, \bA_n)$ of Lehner's formula.
    For any $\bZ$ feasible for \eqref{eq:sc-sdp-large-primal} we have $\bZ \succeq \bm 0$ and also $\bH(c + 2\epsilon, \bZ + \epsilon \bm I_d) = \bH(c, \bZ) + \epsilon \bm I_{d(n + 1)} \succ \bm 0$ for all $\epsilon > 0$.
    Taking $\epsilon \to 0$ then shows that \eqref{eq:sc-sdp-large-primal} with the original constraint $\bH(c, \bZ) \succeq \bm 0$ also equals $\Lambda_{\sc}(\bA_0, \bA_1, \dots, \bA_n)$.

  The second SDP \eqref{eq:sc-sdp-large-dual} is the Lagrangian dual of \eqref{eq:sc-sdp-large-primal} with the maximum replaced by a supremum; we show this in Appendix~\ref{app:dual-large}.
  Note also that the first SDP is strictly feasible, as we see by taking $\bZ = \bm I_d$ and $c > 0$ sufficiently large.
  Therefore Slater's condition holds, and it follows that strong duality holds and that the maximum in \eqref{eq:sc-sdp-large-dual} is achieved.
  (These steps are just as in the proof of Theorem~\ref{thm:sc-sdp-small} above.)
  Moreover, the dual \eqref{eq:sc-sdp-large-dual} is also strictly feasible, as we see by taking $\bm\Gamma^{[0, 0]} = \frac{1}{d}\bm I_d$, $\bm\Gamma^{[i,i]} = \frac{1}{dn}\bm I_d$ for all $i \in [n]$, and all off-diagonal blocks $\bm\Gamma^{[i, j]} = \bm 0$.
  Thus by Slater's condition the minimum in \eqref{eq:sc-sdp-large-primal} is also achieved.
\end{proof}

\subsection{Complementary slackness}
\label{sec:slackness}

We now derive relationships between the optimizers in the primal and dual programs above using the \emph{complementary slackness} relation given in the KKT conditions for optimality.
\begin{lemma}[Complementary slackness for small SDPs]
    \label{lem:sc-slackness}
    In the setting of Theorem~\ref{thm:sc-sdp-small}, suppose that the infimum in the primal problem is achieved by some $(c, \bZ)$ and that $(\bS, \bT)$ achieve the maximum of the dual problem \eqref{eq:sc-sdp-small-dual}.
    The following then hold:
    \begin{align}
      \bZ &\succ \bm 0, \\
      \bS &= \bZ \Phi(\bS) \bZ, \label{eq:sc-slackness-S} \\
      \bT &= \bZ \Phi(\bS) = \bS\bZ^{-1}, \label{eq:sc-slackness-T} \\
      (c \bm I_d - \bA_0 - \Phi(\bZ) - \bZ^{-1}) \bS &= \bm 0. \label{eq:sc-slackness-K}
                                                     \intertext{If $\bS \succ \bm 0$, then we further have}
                                                     c \bm I_d - \bA_0 - \Phi(\bZ) - \bZ^{-1} &= \bm 0, \numberthis \label{eq:original-slackness} \\
      \bZ &= \Phi(\bS)^{-1/2} (\Phi(\bS)^{1/2} \bS \Phi(\bS)^{1/2})^{1/2} \Phi(\bS)^{-1/2}. \label{eq:sc-slackness-geomean}
    \end{align}
\end{lemma}
\noindent
We note that the bottom right expression sometimes goes by the name of the \emph{matrix geometric mean} of $\bS$ and $\Phi(\bS)^{-1}$ \cite[Section 4.1]{Bhatia-2009-PositiveDefiniteMatrices}, a special case of the non-commutative perspective with respect to the function $t \mapsto t^{1/2}$.
\begin{proof}
    Fix $(c, \bZ)$ and $(\bS, \bT)$ primal and dual optimizers, respectively.
    Define
    \begin{align*}
      \bK &\colonequals c \bm I_d - \bA_0 - \Phi(\bZ) - \bZ^{-1}, \\
        \bM &\colonequals \left[\begin{array}{cc} c \bm I_d - \bA_0 - \Phi(\bZ) & \bm I_d \\ \bm I_d & \bZ \end{array}\right], \\
        \bY &\colonequals \left[\begin{array}{cc} \bS & -\bT \\ -\bT^* & \Phi(\bS) \end{array}\right].
    \end{align*}
    By feasibility, $\bM, \bY \succeq \bm 0$, and by Proposition~\ref{prop:Z-pd} we have $\bZ \succ \bm 0$.
    So, by Proposition~\ref{prop:schur}, $\bK \succeq \bm 0$.

    Because the primal problem is strictly feasible, the KKT conditions are necessary and sufficient for optimality.
    In particular, complementary slackness gives $\langle \bM, \bY \rangle = 0$.
    Since $\bM, \bY \succeq \bm 0$, this implies $\bM\bY = \bY\bM = \bm 0$.
    Expanding these matrix multiplications in blocks, we find the equalities
    \begin{align*}
      (c\bm I_d - \bA_0 - \Phi(\bZ))\bS &= \bT^*, \\
      (c\bm I_d - \bA_0 - \Phi(\bZ))\bT &= \Phi(\bS), \\
      \bS &= \bZ\bT^*, \\
      \bT &= \bZ\Phi(\bS), \\
      \bS(c\bm I_d - \bA_0 - \Phi(\bZ)) &= \bT, \\
      \bS &= \bT\bZ, \\
      \bT^*(c\bm I_d - \bA_0 - \Phi(\bZ)) &= \Phi(\bS), \\
      \bT^* &= \Phi(\bS)\bZ.
    \end{align*}
    Since $\bZ \succ \bm 0$, we may solve for $\bT$ to get \eqref{eq:sc-slackness-T}.
    Substituting $\bT = \bZ\Phi(\bS)$ into $\bS = \bT\bZ$ gives \eqref{eq:sc-slackness-S}.
    We also obtain \eqref{eq:sc-slackness-K}, since
    \[ \bK\bS = (c\bm I_d - \bA_0 - \Phi(\bZ))\bS - \bZ^{-1}\bS = \bT^* - \bZ^{-1}\bS = \bm 0. \]

    Finally, if $\bS \succ \bm 0$, then we may cancel $\bS$ above giving $\bK = \bm 0$, which is the statement \eqref{eq:original-slackness}.
    Also, we have $\Phi(\bS) = \bZ^{-1}\bS\bZ^{-1} \succ \bm 0$.
    Let
    \[
        \bW \colonequals \Phi(\bS)^{1/2}\bZ\Phi(\bS)^{1/2}.
    \]
    Then $\bW \succ \bm 0$, and \eqref{eq:sc-slackness-S} becomes
    \[
        \bW^2 = \Phi(\bS)^{1/2}\bS\Phi(\bS)^{1/2}.
    \]
    Taking the (positive definite) square root of both sides gives
    \[
        \bW = (\Phi(\bS)^{1/2}\bS\Phi(\bS)^{1/2})^{1/2},
    \]
    and conjugating by $\Phi(\bS)^{-1/2}$ yields \eqref{eq:sc-slackness-geomean}.
\end{proof}

\begin{remark}
    Recall that Lehner's formula in Theorem~\ref{thm:lehner-sc} has the additional feature that one may restrict to $\bZ$ for which $\bA_0 + \Phi(\bZ) + \bZ^{-1}$ is a multiple of the identity.
    It is very tempting to argue that this is a consequence of the complementary slackness relations: if it were true that there is \emph{always} a dual optimizer $\bS$ having full rank, then indeed that would follow from Lemma~\ref{lem:sc-slackness}.

    However, this need not be the case; for example, take $d = 2$, $n = 1$, and
    \[ \bA_0 = \be_1\be_1^{\top} = \left[\begin{array}{rr} 1 & 0 \\ 0 & 0\end{array}\right], \,\,\,\,\, \bA_1 = \bm I_2 = \left[\begin{array}{rr} 1 & 0 \\ 0 & 1\end{array}\right]. \]
    Then, $\Phi(\bS) = \bS$ and the condensed version \eqref{eq:sc-sdp-small-dual-2} of the dual SDP reduces to
    \[ \Lambda_{\sc}(\bA_0, \bA_1) = \max_{\substack{\bS \succeq \bm 0 \\ \Tr(\bS) = 1}} \left\{S_{11} + 2\right\} \]
    whose unique optimizer is $\bS = \bA_0 = \be_1\be_1^{\top}$, a rank-deficient matrix.

    This example also demonstrates that it need not be the case that \emph{every} primal optimizer $\bZ$ has $\bA_0 + \Phi(\bZ) + \bZ^{-1}$ a multiple of the identity.
    Indeed, in this case $\bZ = \mathsf{Diag}(1, z)$ is such an optimizer for any $z > 0$ such that $z + z^{-1} \leq 3$, and only the case $z + z^{-1} = 3$ makes $\bA_0 + \Phi(\bZ) + \bZ^{-1}$ a multiple of the identity.
\end{remark}

\begin{lemma}[Complementary slackness for large SDPs]
    In the setting of Theorem~\ref{thm:sc-sdp-large}, suppose that the minimum in the primal problem is achieved by some $(c, \bZ)$ and the maximum in the dual problem by some $\bm\Gamma$.
    The following then hold:
    \begin{align}
      (c \bm I_d - \bA_0 - \bZ)\bm\Gamma^{[0, j]} &= -\sum_{i = 1}^n \bA_i \bm\Gamma^{[i,j]} \text{ for all } 0 \leq j \leq n, \label{eq:sc-slackness-large-1} \\
      \bZ\bm\Gamma^{[i, j]} &= -\bA_i\bm\Gamma^{[0, j]} \text{ for all } 1 \leq i \leq n, 0 \leq j \leq n, \label{eq:sc-slackness-large-2}  \\
      \bZ\bm\Gamma^{[0, 0]}\bZ &= \Phi(\bm\Gamma^{[0, 0]}). \label{eq:sc-slackness-large-Z-G}
    \end{align}
\end{lemma}
\noindent
We note that, if we set $\bS = \bm\Gamma^{[0, 0]}$ and have $\bZ \succ \bm 0$, then \eqref{eq:sc-slackness-large-Z-G} is the same as \eqref{eq:sc-slackness-S} from the complementary slackness relations of the small SDP, except that $\bZ$ is replaced by $\bZ^{-1}$, which precisely matches the same replacement that we make at the beginning of our proof of Theorem~\ref{thm:sc-sdp-small}.
\begin{proof}
    Fix $(c, \bZ)$ and $\bm \Gamma$ primal and dual optimizers, respectively.
    Define
    \[ \bH
      \colonequals
      \left[\begin{array}{ccccc} c\bm I_d - \bA_0 - \bZ & \bA_1 & \bA_2 & \cdots & \bA_n \\ \bA_1 & \bZ & \bm 0 & \cdots & \bm 0 \\ \bA_2 & \bm 0 & \bZ & \cdots & \bm 0 \\
              \vdots & \vdots & \vdots &  \ddots & \vdots \\ \bA_n & \bm 0 & \bm 0 & \cdots & \bZ \end{array}\right]. \]
  As before, by feasibility we have $\bH, \bm\Gamma \succeq \bm 0$.
  Again the primal problem is strictly feasible, the KKT conditions are necessary and sufficient for optimality, and complementary slackness gives $\langle \bH, \bm \Gamma \rangle = 0$ and therefore $\bH\bm\Gamma = \bm\Gamma\bH = \bm 0$.
  Expanding blockwise gives, precisely the relations~\eqref{eq:sc-slackness-large-1} and~\eqref{eq:sc-slackness-large-2}.

  Further, setting $i = j$ in \eqref{eq:sc-slackness-large-2} and taking the conjugate transpose gives that, for each $i \in [n]$,
  \[ \bm\Gamma^{[i, i]} \bZ = -\bm\Gamma^{[i, 0]} \bA_i. \]
  Summing over $i$ and using that, by the constraint on $\bm\Gamma$, we have $\bm\Gamma^{[0, 0]} = \sum_{i = 1}^n \bm\Gamma^{[i, i]}$, we have
  \[ \bm\Gamma^{[0, 0]}\bZ = -\sum_{i = 1}^n \bm\Gamma^{[i, 0]} \bA_i. \]
  Finally, multiplying on the left by $\bZ$ and using \eqref{eq:sc-slackness-large-2} again gives
  \[ \bZ\bm\Gamma^{[0, 0]}\bZ = -\sum_{i = 1}^n \bZ \bm\Gamma^{[i, 0]} \bA_i = \sum_{i = 1}^n \bA_i \bm\Gamma^{[0, 0]} \bA_i = \Phi(\bm\Gamma^{[0, 0]}), \]
  giving the final result \eqref{eq:sc-slackness-large-Z-G}.
\end{proof}

\subsection{Estimating distance from optimality}
\label{sec:dist-to-opt}

We now aim to deduce, motivated by the structure of the complementary slackness relations, an ``automatic'' estimate on how close a given primal feasible point $(c, \bZ)$ of the small SDP in Theorem~\ref{thm:sc-sdp-small} is to optimal.
The simple idea is that, if $\bZ$ is a primal optimizer and $\bS$ a dual optimizer, then from Lemma~\ref{lem:sc-slackness} we expect to have $\bS = \bZ\Phi(\bS)\bZ$.
In fact, it turns out that even if we find such $\bS$ for an arbitrary feasible $(c, \bZ)$, we obtain a convenient bound on the distance from $c$ to the true optimal value.

\begin{lemma}
    \label{lem:approx-slackness}
    Let $\bA_0, \bA_1, \dots, \bA_n \in \CC^{d \times d}_{\herm}$, let $(c, \bZ)$ be feasible for the primal small SDP \eqref{eq:sc-sdp-small-primal}.
    Recall that by Proposition~\ref{prop:Z-pd} this implies $\bZ \succ \bm 0$.
    Define $\Psi_{\bZ}: \CC^{d \times d}_{\herm} \to \CC^{d \times d}_{\herm}$ by
    \[
        \Psi_{\bZ}(\bS) \colonequals \bZ \Phi(\bS)\bZ.
    \]
    Suppose that $\bS \succeq \bm 0$, $\Tr(\bS) = 1$, and
    \[
        \Psi_{\bZ}(\bS) = \bZ \Phi(\bS)\bZ = \rho \bS
    \]
    for some $\rho > 0$ (that is, $\bS$ is an eigenmatrix of the operator $\Psi_{\bZ}$ with eigenvalue $\rho$).
    Then,
    \begin{equation}
        \label{eq:approx-slackness}
        c - \langle c \bm I_d - \bA_0 - \Phi(\bZ) - \bZ^{-1}, \bS \rangle - (\sqrt{\rho} - 1)^2 \langle \bZ^{-1}, \bS\rangle \leq \Lambda_{\sc}(\bA_0, \bA_1, \dots, \bA_n) \leq c.
    \end{equation}
\end{lemma}
\noindent
To interpret the result, suppose $\rho = 1$, as we have for an actual pair of primal and dual optimizers.
Then, the second term in the lower bound of \eqref{eq:approx-slackness} vanishes.
Note that, as appeared in our argument in Lemma~\ref{lem:sc-slackness}, $c \bm I_d - \bA_0 - \Phi(\bZ) - \bZ^{-1} \succeq \bm 0$, and further this is a natural ``slack variable'' for the primal small SDP, given by a Schur complement of the matrix constrained to be positive semidefinite.
So, Lemma~\ref{lem:approx-slackness} says that the distance of $\bZ$ to optimality is governed by how little the eigenmatrix $\bS$ ``overlaps'' with this slack variable.

\begin{proof}
    Note that $\Lambda_{\sc}(\bA_0, \bA_1, \dots, \bA_n) \leq c$ immediately by feasibility of $(c, \bZ)$.
    Expanding the proposed lower bound using that $\Tr(\bS) = 1$ and the eigenmatrix relation for $\bS$, we have
    \begin{align*}
      &c - \langle c \bm I_d - \bA_0 - \Phi(\bZ) - \bZ^{-1}, \bS \rangle - (\rho^{1/2} - 1)^2 \langle \bZ^{-1}, \bS \rangle \\
      &= \langle \bA_0, \bS \rangle + \langle \Phi(\bZ), \bS \rangle + \langle \bZ^{-1}, \bS \rangle - (\rho + 1 - 2\rho^{1/2}) \langle \bZ^{-1}, \bS \rangle \\
      &= \langle \bA_0, \bS \rangle + \langle \bZ, \Phi(\bS) \rangle - (\rho - 2\rho^{1/2}) \langle \bZ^{-1}, \bS \rangle \\
      &= \langle \bA_0, \bS \rangle + \rho \langle \bZ, \bZ^{-1}\bS\bZ^{-1} \rangle - (\rho - 2\rho^{1/2}) \langle \bZ^{-1}, \bS \rangle \\
      &= \langle \bA_0, \bS \rangle + 2\rho^{1/2} \langle \bZ^{-1}, \bS \rangle.
    \end{align*}

    Define
    \[
        \bT \colonequals \rho^{-1/2}\bZ\Phi(\bS) = \rho^{1/2}\bS\bZ^{-1},
    \]
    then $(\bS, \bT)$ is feasible for the expanded dual problem \eqref{eq:sc-sdp-small-dual}.
    So, we have the lower bound
    \[ \Lambda_{\sc}(\bA_0, \bA_1, \dots, \bA_n) \geq \langle \bA_0, \bS \rangle + 2\Re(\Tr(\bT)) = \langle \bA_0, \bS \rangle + 2\rho^{1/2}\langle \bZ^{-1}, \bS \rangle, \]
    giving the result.
\end{proof}

\begin{remark}
    We do not do it here, but one can write a similar statement ``in reverse'' that shows that, given $\bS$ of full rank, if one takes $\bZ$ to be defined by the complementarity relation \eqref{eq:sc-slackness-geomean}, then one can use that $\bZ$ as a primal certificate to bound the distance of $\bS$ from optimality.
    This seems less useful, since even in the relatively straightforward spiked matrix model setting we consider in Section~\ref{sec:bbp-iso}, already we run into the situation that a near-optimal $\bZ$ is easy to construct while we do not know a closed form for a near-optimal $\bS$.
\end{remark}

We also give some ancillary results that help find suitable eigenmatrices $\bS$ of operators $\Psi_{\bZ}$.
In the setting of Lemma~\ref{lem:approx-slackness}, we define the additional operator
\[ \widetilde{\Psi}_{\bZ}(\bS) \colonequals \bZ^{1/2}\Phi(\bZ^{1/2}\bS\bZ^{1/2})\bZ^{1/2}. \]
This satisfies the following properties:
\begin{proposition}
    \label{prop:Psi}
    For any $\bA_1, \dots, \bA_n, \bZ \in \CC^{d \times d}_{\herm}$ with $\bZ \succ \bm 0$, the following hold:
    \begin{enumerate}
    \item $\widetilde{\Psi}_{\bZ}$ is linear.
    \item $\widetilde{\Psi}_{\bZ}$ is self-adjoint with respect to the Frobenius inner product: $\langle \widetilde{\Psi}_{\bZ}(\bS), \bT \rangle = \langle \bS, \widetilde{\Psi}_{\bZ}(\bT)\rangle$.
    \item $\widetilde{\Psi}_{\bZ}$ is completely positive (Definition~\ref{def:cp}).
    \item Let $\rho \colonequals \lambda_{\max}(\widetilde{\Psi}_{\bZ})$, the largest eigenvalue when $\widetilde{\Psi}_{\bZ}$ is viewed as a self-adjoint operator on $\CC^{d \times d}_{\herm}$.
        Then, there exists $\widetilde{\bS} \succeq \bm 0$ an eigenmatrix with eigenvalue $\rho$, i.e., with $\widetilde{\Psi}_{\bZ}(\widetilde{\bS}) = \rho\widetilde{\bS}$.
        Setting $\bS \colonequals \bZ^{1/2} \widetilde{\bS} \bZ^{1/2}$, we also have $\bS \succeq \bm 0$ and $\Psi_{\bZ}(\bS) = \rho\bS$.
    \end{enumerate}
\end{proposition}
\noindent
The key point is that, as one may verify, $\Psi_{\bZ}$ is \emph{not} in general self-adjoint, and the self-adjoint variant $\widetilde{\Psi}_{\bZ}$ is easier to work with.
Further, $\widetilde{\Psi}_{\bZ}$ is similar to $\Psi_{\bZ}$, as it is given by conjugating $\Psi_{\bZ}$ by the operation $\bS \mapsto \bZ^{-1/2}\bS\bZ^{-1/2}$.
For this reason the last part lets us translate eigenmatrices of one operator to those of the other, and in particular gives us a means of producing positive semidefinite eigenmatrices of $\Psi_{\bZ}$.

\begin{proof}
    The first three results are easy to verify directly.
    For the last part, the only part that is not immediate is that there exists a positive semidefinite eigenmatrix with eigenvalue $\rho$.
    By the variational description of the largest eigenvalue we have that for any
    \[ \widetilde{\bS} \in \argmax_{\|\widetilde{\bS}\|_F = 1}\, \langle \widetilde{\bS}, \widetilde{\Psi}_{\bZ}(\widetilde{\bS}) \rangle \]
    we will have $\widetilde{\Psi}_{\bZ}(\widetilde{\bS}) = \rho\widetilde{\bS}$, so it suffices to show that there exists a maximizer $\widetilde{\bS}$ that is positive semidefinite.
    Let $\widetilde{\bS}$ be any maximizer, and suppose $\widetilde{\bS}^{\pm} \succeq \bm 0$ are such that $\widetilde{\bS} = \widetilde{\bS}^+ - \widetilde{\bS}^-$ and $\langle \widetilde{\bS}^+, \widetilde{\bS}^- \rangle = 0$.
    Define $|\widetilde{\bS}| \colonequals \widetilde{\bS}^+ + \widetilde{\bS}^- \succeq \bm 0$ and note that $\|\,|\widetilde{\bS}|\,\|_F^2 = \|\widetilde{\bS}\|_F^2 = \|\widetilde{\bS}^+\|_F^2 + \|\widetilde{\bS}^-\|_F^2 = 1$.
    We have
    \begin{align*}
      \langle \widetilde{\bS}, \widetilde{\Psi}_{\bZ}(\widetilde{\bS}) \rangle
      &= \langle \widetilde{\bS}^+ - \widetilde{\bS}^-, \widetilde{\Psi}_{\bZ}(\widetilde{\bS}^+ - \widetilde{\bS}^-) \rangle \\
      &= \langle \widetilde{\bS}^+, \widetilde{\Psi}_{\bZ}(\widetilde{\bS}^+) \rangle + \langle \widetilde{\bS}^-, \widetilde{\Psi}_{\bZ}(\widetilde{\bS}^-) \rangle - \langle \widetilde{\bS}^-, \widetilde{\Psi}_{\bZ}(\widetilde{\bS}^+) \rangle - \langle \widetilde{\bS}^+, \widetilde{\Psi}_{\bZ}(\widetilde{\bS}^-) \rangle
        \intertext{Here all matrices inside the inner products are positive semidefinite by the complete positivity of $\widetilde{\Psi}_{\bZ}$, so we have}
      &\leq \langle \widetilde{\bS}^+ + \widetilde{\bS}^-, \widetilde{\Psi}_{\bZ}(\widetilde{\bS}^+ + \widetilde{\bS}^-) \rangle \\
      &= \langle |\widetilde{\bS}|, \widetilde{\Psi}_{\bZ}(|\widetilde{\bS}|) \rangle,
    \end{align*}
    thus $|\widetilde{\bS}|$ is also a maximizer.
\end{proof}

Thus, given any $\bZ \succ \bm 0$, we may always apply Lemma~\ref{lem:approx-slackness} with $\rho = \lambda_{\max}(\widetilde{\Psi}_{\bZ})$ by taking the $\bS \succeq \bm 0$ produced in the last result of Proposition~\ref{prop:Psi} and normalizing it to have $\Tr(\bS) = 1$.
If we insist on having $\rho = 1$ exactly, we may always renormalize $\bZ \leftarrow \bZ / \sqrt{\lambda_{\max}(\widetilde{\Psi}_{\bZ})}$ which will achieve this, though as we will see in Section~\ref{sec:bbp-iso} it can be easier to argue differently that a suitable $\bZ$ achieving $\rho = 1$ exists.

\subsection{Explicit bounds on spectral edge}
\label{sec:explicit-bounds}

We suppose in this section for the sake of simplicity that $\bA_0 = \bm 0$.
We explore explicit bounds on $\Lambda_{\sc}(\bm 0, \bA_1, \dots, \bA_n)$ that can be written in terms of simple quantities associated to the $\bA_i$.
Analogous bounds that allow $\bA_0 \neq \bm 0$ can be obtained by using that
\[ \Lambda_{\sc}(\bA_0, \bA_1, \dots, \bA_n) \leq \lambda_{\max}(\bA_0) + \Lambda_{\sc}(\bm 0, \bA_1, \dots, \bA_n). \]
The benchmark bound to which we compare ours is the following.
\begin{theorem}[Proof of Theorem 9.9.5 in \cite{Pisier-2003-IntroductionOperatorSpaceTheory}]
    \label{thm:pisier}
    Let $\bA_1, \dots, \bA_n \in \CC^{d \times d}_{\herm}$, and define
    \[ \sigma \colonequals \left\|\sum_{i = 1}^n \bA_i^2\right\|^{1/2}. \]
    Then,
    \[ \sigma \leq \Lambda_{\sc}(\bm 0, \bA_1, \dots, \bA_n) \leq 2\sigma. \]
\end{theorem}
\noindent
The proof in the reference uses manipulations involving the construction of the associated operator $\mathfrak{X}$ in Fock space.
But, there also is a simple proof of the upper bound using Lehner's formula or our SDP formulations: write
\[ \bV^2 \colonequals \sum_{i = 1}^n \bA_i^2, \]
so that $\sigma = \|\bV\|$.
Take $\bZ = t\bm I_d$ in Lehner's formula \eqref{eq:lehner-sc} for some $t > 0$.
Since $\Phi(\bm I_d) = \bV^2$, this gives
\[ \Lambda_{\sc}(\bm 0, \bA_1, \dots, \bA_n) \leq \lambda_{\max}(t\bV^2 + t^{-1}\bm I_d) = t\sigma^2 + t^{-1}. \]
Choosing the optimal $t = \sigma^{-1}$ gives the upper bound.

Curiously, though, we have not been able to find a dual SDP construction that reproduces the lower bound of $\sigma$.
By strong duality, one must exist, and it is a curious problem to find it; perhaps indirect constructions of the kind we use in Section~\ref{sec:bbp-iso} below may achieve this.
On the other hand, we can use our tools to produce different bounds that are tighter in some cases and indicate a natural criterion for when they should be nearly tight, unlike the above which always is loose by a factor of 2 on one side or the other.

\subsubsection{Bound from \texorpdfstring{$\Phi$}{Phi} operator}
Our first alternative bound is based on the associated operator $\Phi: \CC^{d \times d}_{\herm} \to \CC^{d \times d}_{\herm}$ defined as above,
\[ \Phi(\bS) \colonequals \sum_{i = 1}^n \bA_i \bS\bA_i. \]
This is a linear operator on $\CC^{d \times d}_{\herm}$ that is self-adjoint with respect to the Frobenius inner product (Proposition~\ref{prop:Phi-properties}).
Under this structure, we may speak of the eigenvalues of $\Phi$, and in particular we have the usual variational description
\[ \lambda_{\max}(\Phi) = \max_{\|\bS\|_F = 1} \langle \bS, \Phi(\bS) \rangle. \]
A way to view this in terms of matrices is to observe that
\[ \mathsf{vec}(\Phi(\bS)) = \sum_{i = 1}^n \mathsf{vec}(\bA_i\bS\bA_i) = \left(\sum_{i = 1}^n \bA_i \otimes \bA_i\right)\mathsf{vec}(\bS), \]
where $\bA_i \otimes \bA_i \in \CC^{d^2 \times d^2}_{\herm}$ denotes the Kronecker product of matrices (note that this is different than how this notation is used in, for example, the related work \cite{BB-2021-SpectralNormGaussianCorrelated}).
Thus we have
\[ \lambda_{\max}(\Phi) = \lambda_{\max}\left(\sum_{i = 1}^n \bA_i \otimes \bA_i \right). \]
Again, this is distinct from the $v$ parameter of \cite{BB-2021-SpectralNormGaussianCorrelated, BBVH-2021-MatrixConcentrationFreeProbability} given in our \eqref{eq:BBVH-v}, which in our notation is instead given by
\[ v = \lambda_{\max}\left(\sum_{i = 1}^n \mathsf{vec}(\bA_i)\mathsf{vec}(\bA_i)^{\top}\right). \]

We first establish a few preliminaries.
\begin{proposition}
    \label{prop:eigenmatrix}
    There exists $\bS$ with $\Phi(\bS) = \lambda_{\max}(\Phi) \bS$ having $\bS \succeq \bm 0$.
\end{proposition}
\begin{proof}
    This follows by the complete positivity of $\Phi$ by the same argument as in Proposition~\ref{prop:Psi}, since we always have $\langle |\bS|, \Phi(|\bS|) \rangle \geq \langle \bS, \Phi(\bS) \rangle$.
\end{proof}

\begin{proposition}
    \label{prop:lam-max-bound}
    $\lambda_{\max}(\Phi) \leq \sigma^2$.
\end{proposition}
\begin{proof}
    Suppose $\|\bS\|_F = 1$.
    We have
    \begin{align*}
      \langle \bS, \Phi(\bS) \rangle
      &= \sum_{i = 1}^n \Tr(\bA_i \bS\bA_i\bS) \\
      \intertext{and we have by a standard trace inequality that}
      &\leq \sum_{i = 1}^n \Tr(\bA_i^2 \bS^2) \\
      &= \left\langle \sum_{i = 1}^n \bA_i^2, \bS^2 \right\rangle \\
      &\leq \sigma^2 \Tr(\bS^2) \\
      &= \sigma^2,
    \end{align*}
    and the result follows by maximizing over $\bS$.
\end{proof}

The following is our first new bound on $\Lambda_{\sc}(\bm 0, \bA_1, \dots, \bA_n)$.
\begin{corollary}
    \label{cor:spectral-bound}
    In the above setting, suppose that there is $\bS$ such that $\Phi(\bS) = \lambda_{\max}(\Phi)\bS$ and having $\bS \succ \bm 0$.
    Then,
    \[ 2\sqrt{\lambda_{\max}(\Phi)} \leq \Lambda_{\sc}(\bm 0, \bA_1, \dots, \bA_n) \leq \left(\sqrt{\frac{\lambda_{\min}(\bS)}{\lambda_{\max}(\bS)}} + \sqrt{\frac{\lambda_{\max}(\bS)}{\lambda_{\min}(\bS)}}\right) \sqrt{\lambda_{\max}(\Phi)}. \]
    Further, when $\Phi(\bm I_d) = \sum_{i = 1}^n \bA_i^2 = C \bm I_d$ for some $C > 0$, then we may take $\bS = \bm I_d$ above, in which case both inequalities above are equalities and we obtain the exact formula
    \[ \Lambda_{\sc}(\bm 0, \bA_1, \dots, \bA_n) = 2\sqrt{\lambda_{\max}(\Phi)} = 2\sigma. \]
\end{corollary}
\noindent
The latter claim about isotropic models is already known, as discussed for instance in \cite{BBVH-2021-MatrixConcentrationFreeProbability}, and can be seen as following from Lehner's formula or from the characterization of the spectral distribution of the operator $\mathfrak{X} = \sum_{i = 1}^n \bA_i \otimes \mathfrak{s}_i$ as obeying a semicircle law of suitable width in this case.
However, unlike the bounds in Theorem~\ref{thm:pisier}, the bound of Corollary~\ref{cor:spectral-bound} decays gracefully for models that are ``nearly isotropic'' in the sense that the top eigenvector of $\Phi$ (which is itself a matrix) is close in spectrum to the identity matrix.

\begin{proof}[Proof of Corollary~\ref{cor:spectral-bound}]
    Let $\bS$ be as in the statement, normalized such that $\Tr(\bS) = 1$.
    This $\bS$ is feasible for the dual SDP form of the Lehner formula \eqref{eq:sc-sdp-small-dual-2}.
    Plugging it in, we have
    \begin{align*}
      \lambda_{\max}(\bX)
      &\geq 2\|\bS^{1/2} \Phi(\bS)^{1/2}\|_* \\
      &= 2\|\bS^{1/2} (\lambda_{\max}(\Phi) \bS)^{1/2}\|_* \\
      &= 2\sqrt{\lambda_{\max}(\Phi)} \cdot \|\bS\|_* \\
      &= 2\sqrt{\lambda_{\max}(\Phi)},
    \end{align*}
    giving the lower bound (using that $\|\bS\|_* = \Tr(\bS) = 1$).

    For the upper bound, take $\bZ = t\bS$ for $\bS$ as above and $t > 0$ to be selected in the original Lehner formula \eqref{eq:lehner-sc}.
    We have
    \begin{align*}
      \lambda_{\max}(\bX)
      &\leq \lambda_{\max}(\bZ^{-1} + \Phi(\bZ)) \\
      &\leq \lambda_{\max}\left(\frac{1}{t}\bS^{-1} + t \lambda_{\max}(\Phi) \bS\right) \\
      &= \max\left\{\frac{1}{t \lambda_{\min}(\bS)} + t \lambda_{\max}(\Phi) \lambda_{\min}(\bS),\, \frac{1}{t \lambda_{\max}(\bS)} + t \lambda_{\max}(\Phi) \lambda_{\max}(\bS)\right\}.
    \end{align*}
    Minimizing over $t$ then gives the upper bound for $\lambda_{\max}(\bX)$, and an identical argument gives the same bound on $\lambda_{\max}(-\bX)$, whereby the same bound applies to $\|\bX\|$.

    Finally, when $\Phi(\bm I_d) = C \bm I_d$, then $\bm I_d$ is an eigenvector of $\Phi$ with eigenvalue $C = \sigma^2$.
    By Proposition~\ref{prop:lam-max-bound}, $\lambda_{\max}(\Phi) \leq \sigma^2$, so in fact $\lambda_{\max}(\Phi) = C = \sigma^2$ in this case.
    The result then follows by substituting into the bounds established above.
\end{proof}

\subsubsection{Bound from \texorpdfstring{$\bm V$}{V} matrix}

The second inequality will be based instead on the matrix mentioned above,
\[ \bV^2 \colonequals \Phi(\bm I_d) = \sum_{i = 1}^n \bA_i^2. \]
\begin{theorem}
    \label{thm:V-bounds}
    In the above setting, we have
    \[ 2\,\frac{\Tr(\bV^{3})}{\Tr(\bV^2)} \leq \Lambda_{\sc}(\bm 0, \bA_1, \dots, \bA_n) \leq 2\sigma = 2\|\bV\|. \]
    If $\bV = C\bm I_d$ for some $C > 0$, then both inequalities are tight.
\end{theorem}
\noindent
These bounds are close to tight when the model is nearly isotropic in a different sense, that of $\bV$ having a flat spectrum (rather than the properties of the optimizer $\bS$ in Corollary~\ref{cor:spectral-bound}).
\begin{proof}
    We take the dual witness $\bm\Gamma$ in \eqref{eq:sc-sdp-large-dual} given by
    \[ \bm\Gamma = \frac{1}{\mathrm{Tr}(\bV^2)}\left[\begin{array}{c} \bV \\ -\bA_1 \\ \vdots \\ -\bA_n \end{array}\right]\left[\begin{array}{c} \bV \\ -\bA_1 \\ \vdots \\ -\bA_n \end{array}\right]^{*}, \]
    which will satisfy the constraints, having $\bm\Gamma \succeq \bm 0$ by construction and
    \begin{align*}
      \bm\Gamma^{[0, 0]} &= \frac{1}{\Tr(\bV^2)}\bV^2, \\
      \bm\Gamma^{[i, i]} &= \frac{1}{\Tr(\bV^2)} \bA_i^2 \text{ for } 1 \leq i \leq n, \\
      \bm\Gamma^{[0, i]} &= -\frac{1}{\Tr(\bV^2)} \bV \bA_i \text{ for } 1 \leq i \leq n.
    \end{align*}
    As we are assuming $\bm A_0 = \bm 0$, this then gives the bound
    \[ \Lambda_{\sc}(\bm 0, \bA_1, \dots, \bA_n) \geq -2\Re\left(\sum_{i = 1}^n \langle \bm\Gamma^{[0, i]}, \bA_i \rangle\right) = \frac{2}{\Tr(\bV^2)} \Re\left(\sum_{i = 1}^n \Tr(\bV \bA_i^2)\right) = 2\,\frac{\Tr(\bV^{3})}{\Tr(\bV^2)}, \]
    as claimed.
\end{proof}

\begin{remark}
    There is a more general family of construction taking, for some $\bB_i \in \CC^{d \times d}$ for each $1 \leq i \leq n$,
    \[ \bm\Gamma = \frac{1}{\mathrm{Tr}(\sum_{i = 1}^n \bB_i\bB_i^*)}\left[\begin{array}{c} (\sum_{i = 1}^n \bB_i\bB_i^*)^{1/2} \\ -\bB_1 \\ \vdots \\ -\bB_n \end{array}\right]\left[\begin{array}{c} (\sum_{i = 1}^n \bB_i\bB_i^*)^{1/2} \\ -\bB_1 \\ \vdots \\ -\bB_n \end{array}\right]^{*}. \]
    This is feasible in \eqref{eq:sc-sdp-large-dual} for the same reasons as in the proof above, and it shows that
    \[ \Lambda_{\sc}(\bm 0, \bA_1, \dots, \bA_n) \geq \frac{2}{\mathrm{Tr}(\sum_{i = 1}^n \bB_i\bB_i^*)} \sum_{i = 1}^n\Re\left(\Tr\left(\bA_i \left(\sum_{i = 1}^n \bB_i\bB_i^*\right)^{1/2} \bB_i^*\right)\right). \]
    For example, optimizing over $\sum_{i = 1}^n \bB_i\bB_i^* = \bm I_d$ gives $\Lambda_{\sc}(\bm 0, \bA_1, \dots, \bA_n) \geq \frac{2}{d}\|\bA\|_*$, where $\bA \in \CC^{dn \times d}$ is the matrix of the $\bA_i$ stacked together.
    This bound looks natural, but is easily seen to never be superior to the one in Theorem~\ref{thm:V-bounds}.
    It is also reasonable, for example, to take $\bB_i \colonequals f(\bA_i)$ for various spectral functions $f$, and each such choice gives a different variant of the bound of Theorem~\ref{thm:V-bounds}.
\end{remark}

\section{Semidefinite program formulation for Rademacher formula}
\label{sec:rad}

\subsection{Non-convexity: Proof of Theorem~\ref{thm:rad-non-convex}}

First, let us complete the discussion begun in Section~\ref{sec:perspective}, where we showed that the non-commutative perspective associated to the Rademacher case does not yield a convex matrix-valued function.
Here, we show that in fact the entire (scalar-valued) objective of Lehner's formula in this case is not convex.
\begin{proof}[Proof of Theorem~\ref{thm:rad-non-convex}]
    We recall the statement, reformulated according to the notation of Section~\ref{sec:perspective-reform}: we will show that there exist $\bA_0, \bA_1, \dots, \bA_n \in \CC^{d \times d}_{\herm}$ such that the function $F: \{\bZ \in \CC^{d \times d}_{\herm}: \bZ \succ \bm 0\} \to \RR$ defined by $F(\bZ) \colonequals \lambda_{\max}(\bZ + \bA_0 + \sum_{i = 1}^n P_{f_{\rad}}(\bA_i, \bZ))$, the objective function of the optimization \eqref{eq:lehner-rad}, is not convex.
    These may further be chosen to satisfy $\bA_i \succ \bm 0$ for all $i = 0, 1, \dots, n$.

    We take $d = 2$, $n = 1$, and $\bA_0 = \bm 0$.
    In this case, we are reduced to considering the function
    \[ F_{\bA}(\bZ) \colonequals \lambda_{\max}(\bZ + P_{f_{\rad}}(\bA, \bZ)). \]
    for $\bA = \bA_1$.
    Taking
    \[ \bA = \left[\begin{array}{cc} 21 & 2 \\ 2 & 2 \end{array}\right], \hspace{0.5cm} \bY = \left[\begin{array}{cc} 10 & 3 \\ 3 & 3 \end{array}\right], \hspace{0.5cm} \bZ = \left[\begin{array}{cc} 10 & 3 \\ 3 & 2 \end{array}\right], \]
    a numerical computation shows that
    \[ \frac{1}{2}(F_{\bA}(\bY) + F_{\bA}(\bZ)) - F_{\bA}\left(\frac{1}{2}(\bY + \bZ)\right) < -10^{-4} < 0, \]
    and $\bA, \bY, \bZ \succ \bm 0$.
    To take $\bA_0 \succ \bm 0$ as well, we may instead set $\bA_0 = \epsilon \bm I_d$ for an arbitrary sufficiently small $\epsilon > 0$.
\end{proof}

\subsection{Rademacher SDP: Proof of Theorem~\ref{thm:rad-sdp}}
\label{sec:pf-rad-sdp}

We recall the SDP in question in Theorem~\ref{thm:rad-sdp}, which we denote
\[ \mathsf{SDP}_{\rad}(\bA_0, \bA_1, \dots, \bA_n) \colonequals \left\{\begin{array}{ll} \text{minimum of} & c \\ \text{subject to} & \bZ \succeq \bm 0, \\ & \left[\begin{array}{cc} \bY_i & \bA_i \\ \bA_i & \bY_i + \bZ \end{array}\right] \succeq \bm 0 \text{ for all } i \in [n], \\ & \bZ + \bA_0 + \sum_{i = 1}^n \bY_i \preceq c \bm I_d, \\ & c \in \RR, \bZ, \bY_i \in \CC^{d \times d}_{\herm}, \text{ for all } i \in [n] \end{array}\right\}. \]
In this section, we suppose that we have fixed $\bA_0, \bA_1, \dots, \bA_n \in \CC^{d \times d}_{\herm}$, and we write $\SDP_{\rad}$ and $\Lambda_{\rad}$ without arguments for the sake of brevity.

We prove Theorem~\ref{thm:rad-sdp} by proving below that $\SDP_{\rad} \leq \Lambda_{\rad}$ (Lemma~\ref{lem:rad-sdp-1}) and $\SDP_{\rad} \geq \Lambda_{\rad}$ (Lemma~\ref{lem:rad-sdp-2}).
First, we give the few final details aside from these two results.
\begin{proof}[Proof of Theorem~\ref{thm:rad-sdp}]
    From Lemmas~\ref{lem:rad-sdp-1} and~\ref{lem:rad-sdp-2} it follows that $\Lambda_{\rad} = \SDP_{\rad}$.
    The second SDP \eqref{eq:rad-sdp-dual} is the Lagrangian dual of \eqref{eq:rad-sdp-primal} with the maximum replaced by a supremum; we show this in Appendix~\ref{app:rad}.
    The primal SDP is strictly feasible by taking $\bZ = t\bm I_d$ and $\bY_i = t\bm I_d$ for $t > 0$ large enough, and then taking $c$ large enough.
    The dual SDP is strictly feasible by taking $\bm\Delta = d^{-1}\bm I_d$, $\bR_i = \bm 0$, $\bT_i = \alpha \bm\Delta$, and $\bS_i = (1 - \alpha)\bm\Delta$ for any $0 < \alpha < 1/n$.
    Both programs have finite value.
    Applying Slater's condition in both directions gives strong duality, and the primal minimum and dual maximum are both achieved.
\end{proof}

We now give the proofs of the two main inequalities.
The first direction is easier.
\begin{lemma}
    \label{lem:rad-sdp-1}
    $\SDP_{\rad} \leq \Lambda_{\rad}$.
\end{lemma}
\begin{proof}
    Suppose that $\bZ \succ \bm 0$.
    Introduce the variables
    \[ \bM_i \colonequals \bZ^{-1/2}\bA_i \bZ^{-1/2} \]
    so that
    \[ P_{f_{\rad}}(\bA_i, \bZ) = \bZ^{1/2} f_{\rad}(\bM_i) \bZ^{1/2}. \]
    Note that
    \begin{align*}
      &\left[\begin{array}{cc} P_{f_{\rad}}(\bA_i, \bZ) & \bA_i \\ \bA_i & P_{f_{\rad}}(\bA_i, \bZ) + \bZ \end{array}\right] \\
      &\hspace{1.5cm} = \left[\begin{array}{cc} \bZ^{1/2} & \bm 0 \\ \bm 0 & \bZ^{1/2} \end{array}\right]\left[\begin{array}{cc} f_{\rad}(\bM_i) & \bM_i \\ \bM_i & f_{\rad}(\bM_i) + \bm I_d \end{array}\right]\left[\begin{array}{cc} \bZ^{1/2} & \bm 0 \\ \bm 0 & \bZ^{1/2} \end{array}\right],
    \end{align*}
    and here the middle matrix is positive semidefinite, since it is similar to a block-diagonal matrix of two-by-two blocks of the form
    \[ \left[\begin{array}{cc} f_{\rad}(\lambda) & \lambda \\ \lambda & f_{\rad}(\lambda) + 1 \end{array}\right] \]
    over $\lambda$ the eigenvalues of $\bM_i$, and any matrix of the above form is positive semidefinite (having determinant zero and non-negative trace).
    Thus we have found
    \[ \left[\begin{array}{cc} P_{f_{\rad}}(\bA_i, \bZ) & \bA_i \\ \bA_i & P_{f_{\rad}}(\bA_i, \bZ) + \bZ \end{array}\right] \succeq \bm 0. \]
    Thus, in $\mathsf{SDP}_{\rad}$ we have that, for any given $\bZ \succ \bm 0$, $\bY_i = P_{f_{\rad}}(\bA_i, \bZ)$ satisfy the second positivity constraint, and the minimum value of $c$ for these $\bZ$ and $\bY_i$ is precisely $\lambda_{\max}(\bA_0 + \bZ + \sum_{i = 1}^n P_{f_{\rad}}(\bA_i, \bZ))$, the objective function of $\Lambda_{\rad}$.
    So, we have $\SDP_{\rad} \leq \Lambda_{\rad}$, as claimed.
\end{proof}

The opposite bound is subtler and more interesting.
Indeed, as we discuss further in Section~\ref{sec:ai}, this result was not expected by the author and was proved autonomously by the GPT-5.5~Pro large language model when prompted to instead find an example showing that it does \emph{not} hold.
We present the two proofs that this model discovered, one using operator-theoretic reasoning and based on the original relationship in Lehner's formula between $\Lambda_{\rad}$ and operator norms, and the other treating $\Lambda_{\rad}$ purely as a finite-dimensional optimization problem.
The first proof is simpler, but the second reveals some interesting structure in the Rademacher Lehner formula and the non-commutative perspective of $f_{\rad}$.
\begin{lemma}
    \label{lem:rad-sdp-2}
    $\SDP_{\rad} \geq \Lambda_{\rad}$.
\end{lemma}

\begin{proof}[Proof 1 of Lemma~\ref{lem:rad-sdp-2}]
    We use the Lehner formula \eqref{eq:lehner-rad}, which gives
    \[ \Lambda_{\rad} = \lambda_{\max}\left(\bA_0 \otimes \id + \sum_{i = 1}^n \bA_i \otimes \mathfrak{r}_i\right). \]
    We use the concrete realization of the $\mathfrak{r}_i$ mentioned above and in \cite{Lehner-1999-NormFreeOperatorMatrixCoefficients}: let $G = (\ZZ / 2\ZZ)^{* n}$ be the free product, write $g_1, \dots, g_n$ for the generators, and realize $\mathfrak{r}_i$ on the Hilbert space $\sH = \ell^2(G)$ as left multiplication by $g_i$.
    Write $\mathfrak{p}_i$ for the projection operator to the closed span of words which, when reduced by the relations $g_i^2 = 1$, have $g_i$ as their first term for $1 \leq i \leq n$, write $\mathfrak{p}_0$ for the projection to the span of the empty word, and let $\mathfrak{q}_i = \mathfrak{1} - \mathfrak{p}_i$.
    We have $\sum_{i = 0}^n \mathfrak{p}_i = \mathfrak{1}$ and $\mathfrak{p}_i + \mathfrak{q}_i = \mathfrak{1}$ for all $0 \leq i \leq n$.
    Note that, for $1 \leq i \leq n$, left multiplication by $\mathfrak{r}_i$ maps the subspace $\mathfrak{p}_i\sH$ to $\mathfrak{q}_i\sH$ and vice-versa.
    Thus, there is a unitary $\mathfrak{u}_i$ such that, relative to the decomposition $\sH = \mathfrak{q}_i\sH \oplus \mathfrak{p}_i\sH$, we have the block decompositions
    \[ \mathfrak{r}_i = \left[\begin{array}{cc} \mathfrak{0} & \mathfrak{u}_i^* \\ \mathfrak{u}_i & \mathfrak{0} \end{array}\right], \quad
      \mathfrak{q}_i = \left[\begin{array}{cc} \mathfrak{1} & \mathfrak{0} \\ \mathfrak{0} & \mathfrak{0} \end{array}\right], \quad
      \mathfrak{p}_i = \left[\begin{array}{cc} \mathfrak{0} & \mathfrak{0} \\ \mathfrak{0} & \mathfrak{1} \end{array}\right]. \]

    Returning to the main argument, let $(c, \bZ, \bY_1, \dots, \bY_n)$ be feasible for $\SDP_{\rad}$.
    Conjugating the block constraint involving $\bY_i$ appropriately, we have
    \[ \left[\begin{array}{cc} \bY_i & -\bA_i \\ -\bA_i & \bY_i + \bZ \end{array}\right] \succeq \bm 0, \]
    and combining with the above observation gives
    \[ \bY_i \otimes \mathfrak{q}_i + (\bY_i + \bZ) \otimes \mathfrak{p}_i - \bA_i \otimes \mathfrak{r}_i \succeq \mathfrak{0}, \]
    and rearranging this gives
    \[ \bA_i \otimes \mathfrak{r}_i \preceq \bY_i \otimes \mathfrak{1} + \bZ \otimes \mathfrak{p}_i. \]
    Summing over $i$ and using that $\sum_{i = 1}^n \mathfrak{p}_i = \mathfrak{1} - \mathfrak{p}_0 \preceq \mathfrak{1}$,
    \[ \bA_0 \otimes \mathfrak{1} + \sum_{i = 1}^n \bA_i \otimes \mathfrak{r}_i = \left(\bA_0 + \sum_{i = 1}^n\bY_i\right) \otimes \mathfrak{1} + \bZ \otimes \left(\sum_{i = 1}^n \mathfrak{p}_i \right) \preceq \left(\bA_0 + \sum_{i = 1}^n\bY_i + \bZ\right) \otimes \mathfrak{1}. \]
    Taking $\lambda_{\max}$ on either side and using Lehner's formula then gives
    \[ \Lambda_{\rad} \leq \lambda_{\max}\left(\bA_0 + \sum_{i = 1}^n\bY_i + \bZ\right) \leq c, \]
    and the result follows since this holds for all feasible points of $\SDP_{\rad}$.
\end{proof}

The second proof is instead based on several special properties of the function $f_{\rad}$ and the associated non-commutative perspective.
First, we show that the non-commutative perspective of $f_{\rad}$ can be characterized as the solution of a Riccati-like matrix equation.
\begin{proposition}
    \label{prop:Pfrad-riccati}
    Let $\bA, \bW \in \CC^{d \times d}_{\herm}$ have $\bW \succ \bm 0$.
    Then, for $\bP \succeq \bm 0$, $\bP = P_{f_{\rad}}(\bA, \bW)$ if and only if $\bP = \bA(\bP + \bW)^{-1}\bA$.
\end{proposition}
\begin{proof}
    Define $\what{\bA} \colonequals \bW^{-1/2}\bA \bW^{-1/2}$ and $\what{\bP} \colonequals \bW^{-1/2}\bP \bW^{-1/2}$.
    Conjugating the condition $\bP = \bA(\bP + \bW)^{-1}\bA$ by $\bW^{-1/2}$ gives
    \[ \what{\bP} = \what{\bA}(\bm I_d + \what{\bP})^{-1}\what{\bA}. \]
    Conjugating this in turn by $(\bm I_d + \what{\bP})^{-1/2}$ gives, if we set $\bC \colonequals (\bm I_d + \what{\bP})^{-1/2}\what{\bA}(\bm I_d + \what{\bP})^{-1/2}$, the equivalent
    \[ \bC^2 = (\bm I_d + \what{\bP})^{-1/2}\what{\bP} (\bm I_d + \what{\bP})^{-1/2} = \what{\bP}(\bm I_d + \what{\bP})^{-1}. \]
    We invert the above relation and then substitute for $\what{\bA}$ to obtain
    \begin{align*}
      \what{\bP}
      &= \bC^2(\bm I_d - \bC^2)^{-1}, \\
      \what{\bA}
      &= (\bm I_d + \what{\bP})^{1/2} \bC (\bm I_d + \what{\bP})^{1/2} \\
      &= \bC(\bm I_d - \bC^2)^{-1}.
    \end{align*}
    We may then verify by functional calculus that $\what{\bP} = f(\what{\bA})$, and conjugating by $\bW^{1/2}$ gives that $\bP = P_{f_{\rad}}(\bA, \bW)$.
    This shows that $\bP = \bA(\bP + \bW)^{-1}\bA$ implies that $\bP = P_{f_{\rad}}(\bA, \bW)$, and reversing the same steps shows the converse.
\end{proof}

The basic idea that we will apply below is that, by analogy with Riccati equations, we may solve this equation for $\bP$ and thereby compute $P_{f_{\rad}}(\bA, \bW)$ with a fixed point iteration (given the structure of the equation we use, this may also be viewed as an infinite matrix-valued continued fraction for the non-commutative perspective).
We will actually use this in slightly different form in the main proof, but, since it also follows from prior work, we state a simple standalone version here for the sake of intuition.
\begin{proposition}[Lemma 5.1 and Proposition 5.1 of \cite{FL-1996-HermitianSolutionsRiccati}; Algorithm 2.4 of \cite{GL-1999-IterativeSolutionMatrixEquations}]
    \label{prop:Pfrad-iter}
    In the setting of Proposition~\ref{prop:Pfrad-riccati}, define the sequence
    \begin{align*}
      \bP^{(0)} &\colonequals \bm 0, \\
      \bP^{(k + 1)} &\colonequals \bA (\bP^{(k)} + \bW)^{-1} \bA.
    \end{align*}
    Then, $\bP^{(k)} \to P_{f_{\rad}}(\bA, \bW)$ as $k \to \infty$.
    Further,
    \[ \bP^{(0)} \preceq \bP^{(2)} \preceq \bP^{(4)} \preceq \cdots \preceq P_{f_{\rad}}(\bA, \bW) \preceq \cdots \preceq \bP^{(5)} \preceq \bP^{(3)} \preceq \bP^{(1)}. \]
\end{proposition}

\begin{proof}[Proof 2 of Lemma~\ref{lem:rad-sdp-2}]
    Let $(c, \bZ, \bY_1, \dots, \bY_n)$ be strictly feasible for $\SDP_{\rad}$, so that
    \[ \bZ \succ \bm 0, \quad \left[\begin{array}{cc} \bY_i & \bA_i \\ \bA_i & \bY_i + \bZ \end{array}\right] \succ \bm 0, \quad c\bm I_d - \bA_0 \succ \bZ  + \sum_{i = 1}^n \bY_i. \]

    Let us digress to explain the motivation of the argument to follow.
    We would like to simply reverse our argument for Lemma~\ref{lem:rad-sdp-1} by saying that the second condition above implies $\bY_i \succeq P_{f_{\rad}}(\bA_i, \bZ)$.
    But, this implication cannot hold: if it did, then combined with our argument for Lemma~\ref{lem:rad-sdp-1} (and after attending to technicalities involving weak versus strict Loewner ordering) it would show that the above condition is equivalent to $\bY_i \succeq P_{f_{\rad}}(\bA_i, \bZ)$ and thus gives a semidefinite characterization of the epigraph of the function $P_{f_{\rad}}$.
    In Proposition~\ref{prop:frad} we found that, for certain $\bA_i$, this is not a convex function of $\bZ$, so such a characterization cannot exist.
    Instead, we will use an iterative algorithm inspired by Proposition~\ref{prop:Pfrad-iter} to build $\bW \succ \bm 0$ a ``replacement'' for $\bZ$ and $\bP_i \preceq \bY_i$ replacements for $\bY_i$ such that $\bP_i = \bA_i(\bW + \bP_i)^{-1}\bA_i$ and thus, by Proposition~\ref{prop:Pfrad-riccati}, $\bP_i = P_{f_{\rad}}(\bA_i, \bW)$.
    This will let us relate the above feasible point of $\SDP_{\rad}$ to one of $\Lambda_{\rad}$ to complete the proof.

    Continuing with the proof, note that the second condition above implies that, for each $i \in [n]$,
    \[ \bY_i \succeq \bA_i(\bY_i + \bZ)^{-1}\bA_i. \]
    In this sense, $\bY_i$ is a ``supersolution'' of the equation in Proposition~\ref{prop:Pfrad-riccati}.
    Define the iterates
    \begin{align*}
      \bY_i^{(0)} &\colonequals \bY_i, \\
      \bY_i^{(k + 1)} &\colonequals \bA_i \left(c\bm I_d - \bA_0 - \sum_{j \in [n] \setminus \{i\}}\bY_j^{(k)}\right)^{-1}\bA_i.
    \end{align*}
    For $k = 0$, we have
    \[ c\bm I_d - \bA_0 - \sum_{j \in [n] \setminus \{i\}}\bY_j^{(k)} = (\bY_i + \bZ) + \left(c\bm I_d - \bA_0 - \bZ - \sum_{j = 1}^n \bY_j\right) \succ \bY_i + \bZ, \]
    and thus $\bY_i^{(1)} \preceq \bY_i^{(0)}$.
    By induction it then follows that $\bY_i^{(k + 1)} \preceq \bY_i^{(k)}$ for each $k \geq 0$, so the $\bY_i^{(k)}$ are positive semidefinite and decreasing in the Loewner order.

    Thus, for each $i$, the $\bY_i^{(k)}$ must converge, so let $\bP_i \colonequals \lim_{k \to \infty} \bY_i^{(k)}$.
    And, set $\bW \colonequals c\bm I_d - \bA_0 - \sum_{i = 1}^n \bP_i \succeq \bZ \succ \bm 0$.
    Taking the limit in the iterations gives that, for each $i \in [n]$,
    \[ \bP_i = \bA_i \left(c\bm I_d - \bA_0 - \sum_{j \in [n] \setminus \{i\}}\bP_j\right)^{-1}\bA_i = \bA_i(\bW + \bP_i)^{-1}\bA_i. \]
    So, by Proposition~\ref{prop:Pfrad-riccati}, $\bP_i = P_{f_{\rad}}(\bA_i, \bW)$.
    Rearranging the above then gives
    \[ c\bm I_d = \bA_0 + \bW + \sum_{i = 1}^n P_{f_{\rad}}(\bA_i, \bW). \]
    Thus, the objective function of $\Lambda_{\rad}$ at the feasible point $\bW \succ \bm 0$ is $c$, so $\Lambda_{\rad} \leq c$.

    This holds for $c$ a component of any strictly feasible point of $\SDP_{\rad}$.
    Any feasible point is a limit of strictly feasible points for $\SDP_{\rad}$, and so we find $\Lambda_{\rad} \leq \SDP_{\rad}$, as claimed.
\end{proof}

\section{Remarks on numerical evaluation}
\label{sec:numerical}

We next give some discussion of computing $\Lambda_{\sc}$ and $\Lambda_{\rad}$ in practice.
For the sake of simplicity, we focus here on $\Lambda_{\sc}(\bA_0, \bA_1, \dots, \bA_n)$; the same observations with minor variations also hold for $\Lambda_{\rad}$, and when there are significant differences we will point them out below.

Recall that, if we write
\[ K(\bZ) \colonequals \bA_0 + \bZ + \sum_{i = 1}^n \bA_i\bZ^{-1}\bA_i, \]
Lehner's formula \eqref{eq:lehner-sc} says that $\Lambda_{\sc}$ is the infimum of $\lambda_{\max}(K(\bZ))$ over $\bZ \succ \bm 0$, where we may also restrict to those $\bZ$ for which $K(\bZ)$ is a multiple of the identity.
Lehner also describes a numerical procedure for evaluating this formula in \cite{Lehner-1999-NormFreeOperatorMatrixCoefficients}, whose premise is to use the latter observation and navigate the curve of $\bZ$ for which $K(\bZ)$ is a multiple of the identity.
In the semicircular case, the procedure may be summarized as follows.
For a large scalar $c$, first solve $K(\bZ) = c\bm I_d$ by the fixed-point iteration
\[ \bZ^{(t + 1)} = c\bm I_d - \bA_0 - \sum_{i = 1}^n \bA_i(\bZ^{(t)})^{-1}\bA_i. \]
Then, move along the implicit curve $K(\bZ) = c\bm I_d$ in the direction of smaller $c$.
The tangent direction is found by inverting the derivative
\[ \sD K_{\bZ}(\bH)
  =
  \bH - \sum_{i = 1}^n \bA_i\bZ^{-1}\bH\bZ^{-1}\bA_i,
\]
since if $\bH_{\bZ}$ solves $\sD K_{\bZ}(\bH_{\bZ}) = \bm I_d$, then decreasing $c$ corresponds locally to moving in the direction $-\bH_{\bZ}$.
After such a step, we use a Newton iteration to project back to the curve where $K(\bZ)$ is a multiple of the identity, and loop these steps until convergence.

To the best of our knowledge, no theoretical guarantees have been proven for this numerical scheme in the literature.
Parts of it admit natural local guarantees; for instance, the fixed point iteration above is a contraction provided $c$ is sufficiently large, and Newton's method is locally quadratically convergent when $\sD K_{\bZ}$ is invertible and the initial point is close enough to the set where $K(\bZ)$ is a multiple of the identity \cite[Proposition 6.1]{Lehner-1999-NormFreeOperatorMatrixCoefficients}.
But, it is unclear how to combine these observations into a global convergence theorem for the entire method.
Further, various edge cases can arise: the line search moving in the $-\bH_{\bZ}$ direction can fail, the Newton projection afterwards can leave the positive definite cone, and the derivative $\sD K_{\bZ}$ can become singular as we approach the optimal value.
Indeed, experimenting with a direct implementation of this algorithm we have encountered all of these issues, which at the very least require some ad hoc modifications to detect or for the iteration to recover from (\cite{Lehner-1999-NormFreeOperatorMatrixCoefficients} also discusses some similar issues and various possible fixes).

The following is a simple example of how numerical instability can appear in this method as it nears optimality.
Suppose that $\sum_{i = 1}^n \bA_i^2 = \bm I_d$ (an assumption appearing in Theorem~\ref{thm:BCSVH-BBP}, for example) and $\bA_0 = \bm 0$.
Then $\bZ = \bm I_d$ is optimal for $\Lambda_{\sc}$, but
\[ \sD K_{\bm I_d}(\bm I_d) = \bm I_d - \sum_{i = 1}^n \bA_i^2 = \bm 0. \]
Therefore, the derivative operator at this optimal point becomes singular, and correspondingly becomes ill-conditioned near the optimum.

In contrast to all of this, the SDP formulation in Theorem~\ref{thm:sc-sdp-small} is a routine convex conic optimization problem, solved stably and accurately by off-the-shelf SDP solvers and enjoying standard convergence guarantees and measurements of optimality by primal-dual residuals and duality gaps.
In the semicircular case the SDP is also quite small, requiring only a single $d \times d$ decision variable.
The SDP of Theorem~\ref{thm:rad-sdp} for the Rademacher case enjoys the same advantages but is costlier, requiring $n + 1$ such variables.
Thus, especially for small $d$, the SDP formulations of these optimizations seem to be a much more reliable means of evaluating the Lehner formulas.
Also, while it is unfair to compare solving these SDPs through highly optimized solvers (we ran experiments with the Mosek solver, for instance) versus a direct implementation of the iterative method (which could also sometimes enter the above unstable divergent regimes of behavior, again making it difficult to time fairly), we mention that for small dimension around $d \leq 20$ the SDPs were also considerably faster to solve, by factors of 5 to 10 of wall clock time.

\section{Applications to BBP transitions in spiked matrix models}
\label{sec:bbp}

We now give a new proof of one of the main applications of Lehner's formula of \cite{BCSVH-2024-MatrixConcentrationFreeProbability2}, cited here as Theorem~\ref{thm:BCSVH-BBP}.
Our proof will be slightly more complicated than that in the reference, but we propose that it is perhaps conceptually clearer and may be easier to generalize to other situations.
That work derives numerous results about Gaussian random matrices $\bX$ by constructing the corresponding operator $\mathfrak{X}$ ($\bX_{\free}$ in their notation) and controlling $\lambda_{\max}(\mathfrak{X})$.
If this is the goal, our pairs of primal and dual SDPs suggest a natural strategy: using for example the original Lehner formula and the dual from Theorem~\ref{thm:sc-sdp-small}, we should look for a pair of primal and dual variables $\bZ, \bS \in \CC^{d \times d}_{\herm}$ satisfying the feasibility conditions $\bZ \succ \bm 0$, $\bS \succeq \bm 0$, $\Tr(\bS) = 1$ that approximately satisfy strong duality,
\[ \langle \bA_0, \bS \rangle + 2\|\bS^{1/2}\Phi(\bS)^{1/2}\|_* \approx \lambda_{\max}(\bA_0 + \bZ^{-1} + \Phi(\bZ)). \]
Since we must have
\[ \langle \bA_0, \bS \rangle + 2\|\bS^{1/2}\Phi(\bS)^{1/2}\|_* \leq \lambda_{\max}(\mathfrak{X}) \leq \lambda_{\max}(\bA_0 + \bZ^{-1} + \Phi(\bZ)), \]
if the above holds then we have approximately identified $\lambda_{\max}(\mathfrak{X})$.
Alternatively, Lemma~\ref{lem:approx-slackness} gives an ``automated'' approach to such bounds based on a particular construction of a dual $\bS$ starting from the primal $\bZ$ alone.

However, the proofs of \cite{BCSVH-2024-MatrixConcentrationFreeProbability2} do not quite use this approach.
In particular, while they do prove an upper bound on $\lambda_{\max}(\mathfrak{X})$ by constructing an explicit primal certificate $\bZ$, instead of showing the existence of an explicit dual certificate $\bS$ of a nearly-matching lower bound, they argue the non-existence of a primal certificate $\bZ^{\prime}$ with a value substantially lower than the $\bZ$ they construct.

Our goal here is to give a different proof of one of their main results by the arguably more natural argument of producing both primal and dual certificates.
We use our general Lemma~\ref{lem:approx-slackness}, and hope that this tool and its application here might be useful in more general models.
To review the setting, we work with \emph{spiked matrix models}, random matrices
\[ \bX = \theta \bv\bv^{\top} + \sum_{i = 1}^n g_i \bA_i \]
for $\theta \geq 0$, $\bv$ a fixed ``signal'' unit vector, and $g_1, \dots, g_n \sim \sN(0, 1)$ i.i.d.
We focus here on the rank-one case, but $\theta \bv\bv^{\top}$ could be replaced with a general low-rank matrix and similar methods should still apply.
We seek to understand, for a given low-rank $\bv$, how large $\theta$ must be in order for such a model to have an \emph{outlier} eigenvalue, a value of $\lambda_{\max}(\bX)$ much larger than the typical value when $\theta = 0$.
The tools of \cite{BBVH-2021-MatrixConcentrationFreeProbability, BCSVH-2024-MatrixConcentrationFreeProbability2}, in particular Theorem~\ref{thm:BCSVH}, relate the behavior of this model to the operator
\[ \mathfrak{X} = \theta \bv\bv^{\top} \otimes \mathfrak{1} + \sum_{i = 1}^n \bA_i \otimes \mathfrak{s}_i, \]
for which we can compute $\lambda_{\max}(\mathfrak{X}) = \Lambda_{\sc}(\theta\bv\bv^{\top}, \bA_1, \dots, \bA_n)$.

\subsection{Sensitivity interpretation of SDP dual certificates}
\label{sec:dual-interp}

Reviewing the argument from Section~\ref{sec:intro-spiked}, let us give some intuition for how we expect the dual certificates we construct to behave, since they have an interesting meaning in the context of spiked matrix models.

Consider for a moment a more general model $\bX = \bX(\bA_0, \bA_1, \dots, \bA_n) = \bA_0 + \sum_{i = 1}^n g_i\bA_i$.
Provided the largest eigenvalue of such $\bX$ is simple almost surely, we have
\[ \frac{\partial}{\partial \bA_0} \EE[\lambda_{\max}(\bX)] = \EE\left[\frac{\partial}{\partial \bA_0} \lambda_{\max}(\bX)\right] = \EE[\bv_{\max}(\bX)\bv_{\max}(\bX)^{\top}], \]
where $\bv_{\max}(\bX)$ is the unit eigenvector associated to the eigenvalue $\lambda_{\max}(\bX)$.
Now, if Theorem~\ref{thm:BCSVH} applies, we have $\EE[\lambda_{\max}(\bX)] \approx \Lambda_{\sc}(\bA_0, \bA_1, \dots, \bA_n)$, and thus we expect to have
\[ \frac{\partial}{\partial \bA_0} \Lambda_{\sc}(\bA_0, \bA_1, \dots, \bA_n) \approx \EE[\bv_{\max}(\bX)\bv_{\max}(\bX)^{\top}]. \]
On the other hand, provided we can differentiate the value of the SDP with respect to $\bA_0$, by Theorem~\ref{thm:sdp-diff} we have that, for $\bS^{\star}$ the dual optimizer, we expect to have
\[ \bS^{\star} \approx \EE[\bv_{\max}(\bX)\bv_{\max}(\bX)^{\top}]. \]
That is, the meaning of the dual optimizer $\bS^{\star}$ in such applications is that it is the expected value of the projection in the leading eigendirection of the random $\bX$.

Returning to the specific spiked matrix context, suppose that $\bA_0 = \theta \bv\bv^{\top}$ for $\|\bv\| = 1$.
On the one hand, we may use this observation to guide our construction of dual certificates, at least in simple cases.
In particular, for the most symmetric noise models such as GOE matrices (treated below in Section~\ref{sec:bbp-goe}), we expect to have
\[ \bS^{\star} \approx \EE[\bv_{\max}(\bX)\bv_{\max}(\bX)^{\top}] \approx \alpha(\theta) \bv\bv^{\top} + \beta(\theta)(\bm I_d - \bv\bv^{\top}), \]
which says that $\bv_{\max}(\bX)$ appears isotropically distributed aside from a possibly large component in the direction of the rank-one perturbation $\bv$ that is made to $\bX$.
We will see below (Remark~\ref{rem:bbp-S}) that, at least in the simplest noise model, this is indeed precisely what happens and that the value of $\alpha(\theta)$ matches that established in previous random matrix theory literature.
On the other hand, in more complicated models we will appeal to Lemma~\ref{lem:approx-slackness} and see that it describes nearly-optimal $\bS^{\star}$, which we use to conversely make predictions about $\bv_{\max}(\bX)$ in Conjectures~\ref{conj:spiked-matrix} and~\ref{conj:spiked-matrix-fluct}.

\subsection{Warmup: Gaussian orthogonal ensemble noise}
\label{sec:bbp-goe}

Let $\bW = \bW^{(d)} \sim \GOE(d)$ and let $\bv = \bv^{(d)} \in \SS^{d - 1} \subset \RR^d$ be an arbitrary sequence of deterministic unit vectors.
Let $\theta \geq 0$ be a constant not depending on $d$.
We then consider the sequence of random matrices
\[ \bX = \bX^{(d)} = \theta \bv\bv^{\top} + \bW. \]
We will give a proof of the first part \eqref{eq:goe-bbp-eval} of the following well-known statement.
\begin{theorem}[\cite{FP-2007-LargestEigenvalueWigner,CDMF-2009-DeformedWigner}]
    \label{thm:bbp}
    The following convergences in probability holds as $d \to \infty$:
    \begin{align}
      \lambda_{\max}(\bX^{(d)}) &\to B(\theta) \colonequals \left\{\begin{array}{ll} 2 & \text{if } \theta \leq 1, \\ \theta + \theta^{-1} & \text{if } \theta > 1 \end{array}\right\}, \label{eq:goe-bbp-eval} \\
      |\langle \bv_{\max}(\bX^{(d)}), \bv^{(d)}\rangle|^2 &\to B^{\prime}(\theta) = \left\{\begin{array}{ll} 0 & \text{if } \theta \leq 1, \\ 1 - \theta^{-2} & \text{if } \theta > 1 \end{array}\right\}. \label{eq:goe-bbp-evec}
    \end{align}
\end{theorem}

As mentioned before, the theory of \cite{BBVH-2021-MatrixConcentrationFreeProbability, BCSVH-2024-MatrixConcentrationFreeProbability2} gives tools for comparing this $\bX$ to the following operator.
We expand $\bW$ in a Gaussian series: let $g_{ij} \sim \sN(0, 1)$ be i.i.d.\ for \(1 \leq i \leq j \leq d\), and set
\[ \bE_{ij} \colonequals \frac{1}{\sqrt{d(1 + \One\{i = j\})}}(\be_i\be_j^{\top} + \be_j \be_i^{\top}). \]
Then
\[ \Law(\bW) = \Law\left(\sum_{1 \leq i \leq j \leq d} g_{ij} \bE_{ij}\right). \]
We also have
\[ \EE \bW^2 = \sum_{1 \leq i \leq j \leq d} \bE_{ij}^2 = \frac{d + 1}{d} \bm I_d. \]
We consider
\[ \mathfrak{X} = \mathfrak{X}^{(d)} \colonequals \theta \bv\bv^{\top} \otimes \mathfrak{1} + \sum_{1 \leq i \leq j \leq d} \bE_{ij} \otimes \mathfrak{s}_{ij} \]
for $(\mathfrak{s}_{ij})_{1 \leq i \leq j \leq d}$ a free semicircular family.
We will show the following statement.
\begin{theorem}
    \label{thm:bbp-free}
    For any $\theta \geq 0$, we have $\lim_{d \to \infty} \lambda_{\max}(\mathfrak{X}^{(d)}) = B(\theta)$.
\end{theorem}
\noindent
The eigenvalue limit \eqref{eq:goe-bbp-eval} of Theorem~\ref{thm:bbp} then follows from Theorem~\ref{thm:bbp-free} together with Theorem~\ref{thm:BCSVH} and some routine calculations of the parameters $\sigma(\bX), \sigma_*(\bX)$, and $v(\bX)$ (as defined in the statement of Theorem~\ref{thm:BCSVH}) in this model.

For convenience, we define for use below
\begin{align*}
  \bP = \bP^{(d)} &\colonequals \bv\bv^{\top}, \\
  \bQ = \bQ^{(d)} &\colonequals \bm I_d - \bP,
\end{align*}
two orthogonal projections that partition the identity, $\bP + \bQ = \bm I_d$.

\begin{proof}[Proof of Theorem~\ref{thm:bbp-free}]
    Following our previous notation, we define the operator
    \[ \Phi(\bS) \colonequals \sum_{1 \leq i \leq j \leq d} \bE_{ij}\bS\bE_{ij} = \frac{1}{d}(\bS + \Tr(\bS)\bm I_d), \]
    the right-hand side following by a straightforward computation.

    We now use Lehner's formula in its primal and dual forms to prove upper and lower bounds on $\lambda_{\max}(\mathfrak{X})$.
    For the upper bound, we use Lehner's semicircular formula \eqref{eq:lehner-sc} directly.
    Consider feasible points $\bZ \succ \bm 0$ of the form
    \[ \bZ = \alpha \bP + \beta \bQ \]
    for $\alpha, \beta > 0$ to be chosen later.
    We have $\Tr(\bZ) = \alpha + (d - 1)\beta$, and so
    \[ \Phi(\bZ) = \frac{1}{d}(\bZ + \Tr(\bZ)\bm I_d) = \frac{2\alpha + (d - 1)\beta}{d}\bP + \frac{\alpha + d\beta}{d}\bQ. \]
    Also, $\bZ^{-1} = \alpha^{-1}\bP + \beta^{-1}\bQ$, so
    \begin{align*}
      \lambda_{\max}(\mathfrak{X})
      &\leq \lambda_{\max}\left(\theta \bP + \bZ^{-1} + \Phi(\bZ)\right) \\
      &= \lambda_{\max}\left(\left(\theta + \frac{1}{\alpha} + \frac{2\alpha + (d - 1)\beta}{d}\right)\bP + \left(\frac{1}{\beta} + \frac{\alpha + d\beta}{d}\right)\bQ\right) \\
      &= \max\left\{\theta + \frac{1}{\alpha} + \frac{2\alpha + (d - 1)\beta}{d}, \frac{1}{\beta} + \frac{\alpha + d\beta}{d}\right\}
        \intertext{We choose $\alpha \colonequals \sqrt{d}$, which makes this bounded by}
      &\leq \max\left\{\theta + \frac{d - 1}{d}\beta, \frac{1}{\beta} + \beta\right\} + \frac{3}{\sqrt{d}}
        \intertext{and we choose $\beta \colonequals \min\{1, \theta^{-1}\}$, which finally gives}
      &\leq B(\theta) + \frac{3}{\sqrt{d}}.
    \end{align*}

    For the lower bound, we use the dual small SDP in its condensed form \eqref{eq:sc-sdp-small-dual-2}, which gives
    \[ \lambda_{\max}(\mathfrak{X}) = \max_{\substack{\bS \succeq \bm 0 \\ \Tr(\bS) = 1}} \left\{\langle \theta\bP, \bS \rangle + 2\|\bS^{1/2} \Phi(\bS)^{1/2}\|_*\right\}. \]
    We consider a dual feasible point $\bS$ as before, but here because of the constraint $\Tr(\bS) = 1$ there is only one degree of freedom, so for $\alpha \in [0, 1]$ and $d \geq 2$, we consider
    \[
        \bS \colonequals \alpha \bP + \frac{1-\alpha}{d-1}\bQ.
    \]
    Any such $\bS$ is feasible, and we calculate
    \begin{align*}
      \bS^{1/2} &= \sqrt{\alpha}\,\bP + \sqrt{\frac{1 - \alpha}{d - 1}}\,\bQ, \\
      \Phi(\bS)^{1/2} &= \frac{1}{\sqrt{d}}\left(\sqrt{1 + \alpha}\,\bP + \sqrt{1 + \frac{1-\alpha}{d-1}}\,\bQ\right).
    \end{align*}
    Using that $\bP\bQ = \bm 0$, we get
    \begin{align*}
      \lambda_{\max}(\mathfrak{X})
      &\geq \langle \theta \bP, \bS \rangle + 2\|\bS^{1/2}\Phi(\bS)^{1/2}\|_* \\
      &= \theta\alpha + \frac{2}{\sqrt{d}}\left(\sqrt{\alpha(1 + \alpha)} + \sqrt{(1 - \alpha)(d - \alpha)}\right)
        \intertext{We choose $\alpha \colonequals \max\{0, 1 - \theta^{-2}\}$. This gives}
      &\geq \max\left\{0, \theta - \frac{1}{\theta}\right\} + 2\min\left\{1, \frac{1}{\theta}\right\} \cdot \sqrt{\frac{d - 1}{d}} \\
      &\geq B(\theta) - \frac{2}{\sqrt{d}}.
    \end{align*}

    In summary, we have
    \[ B(\theta) - \frac{2}{\sqrt{d}} \leq \lambda_{\max}(\mathfrak{X}) \leq B(\theta) + \frac{3}{\sqrt{d}}, \]
    so taking $d \to \infty$ gives the result.
\end{proof}

\begin{remark}
    \label{rem:bbp-S}
    We note that the $\bS$ we constructed in the second part of the proof above is of the form $\bS = (B^{\prime}(\theta) + o(1))\bv\bv^{\top} + \beta(\theta) (\bm I_d - \bv\bv^{\top})$, which is precisely what we expect from the second claim \eqref{eq:goe-bbp-evec} of Theorem~\ref{thm:bbp}, along with the sensitivity interpretation of the dual optimizer $\bS^{\star} \approx \EE[\bv_{\max}(\bX)\bv_{\max}(\bX)^{\top}]$ from Section~\ref{sec:dual-interp}.
\end{remark}

\subsection{General isotropic Gaussian noise: Proof of eigenvalue limit of Theorem~\ref{thm:BCSVH-BBP}}
\label{sec:bbp-iso}

We now consider the general case of $\bW^{(d)} = \sum_{i = 1}^n g_i \bA_i^{(d)}$ a jointly Gaussian random matrix with $\EE [(\bW^{(d)})^2] = \bm I_d$ for each $d \geq 1$.
We often omit the ``$(d)$'' superscripts below for the sake of brevity.
As before, we define the associated operator
\[ \mathfrak{X} = \mathfrak{X}^{(d)} \colonequals \theta \bv\bv^{\top} \otimes \mathfrak{1} + \sum_{i = 1}^n \bA_i \otimes \mathfrak{s}_i. \]
Again, by using Theorem~\ref{thm:BCSVH}, the eigenvalue limit theorem part of Theorem~\ref{thm:BCSVH-BBP} is immediate from the following.
\begin{theorem}[Special case of Theorem 2.7 of \cite{BCSVH-2024-MatrixConcentrationFreeProbability2}]
    \label{thm:bbp-iso-free}
    Suppose that $\sigma_*(\bX^{(d)}) = o(1)$, for $\sigma_*(\cdot)$ as defined in \eqref{eq:BBVH-sigma-star}.
    Then, for any $\theta \geq 0$, we have
    \[ \lim_{d \to \infty} \lambda_{\max}(\mathfrak{X}^{(d)}) = B(\theta). \]
\end{theorem}
\noindent
We note that $\sigma_*(\bX) \leq v(\bX)$ always, and so the assumption of Theorem~\ref{thm:bbp-iso-free} follows from that of Theorem~\ref{thm:BCSVH-BBP}.

And, again, our goal is to give a proof of this through constructing explicit primal (Lemma~\ref{lem:bbp-iso-free-upper}) and dual (Lemma~\ref{lem:bbp-iso-free-lower}) certificates.
As we will see, this is possible though subtler in this case; in particular, we can describe a dual certificate only somewhat indirectly using the method from Lemma~\ref{lem:approx-slackness} and Proposition~\ref{prop:Psi}.

Let us first establish some preliminary tools.
We again use the operator
\[ \Phi(\bS) \colonequals \sum_{i = 1}^n \bA_i \bS \bA_i. \]
Note that in this much more general case, unlike in Theorem~\ref{thm:bbp-free}, we do not have a closed form for this operator.
We are, however, guaranteed by the isotropy condition that
\begin{equation}
    \Phi(\bm I_d) = \sum_{i = 1}^n \bA_i^2 = \bm I_d, \label{eq:Phi-Id}
\end{equation}
i.e., that the identity is an eigenmatrix of $\Phi$ with eigenvalue 1, which will be important repeatedly below.

The key role of the $\sigma_*$ quantity for us is that it controls how this operator acts on rank-one matrices (or, low-rank matrices more generally):
\begin{proposition}
    \label{prop:Phi-rank-one}
    Allowing $\Phi$ to act by the same definition on non-Hermitian matrices, we have for any $\bu, \bv \in \CC^{d}$
    \[ \|\Phi(\bu\bv^*)\| \leq \|\bu\| \cdot \|\bv\| \cdot \sigma_*(\bX)^2. \]
\end{proposition}
\begin{proof}
    We expand directly,
    \begin{align*}
      \|\Phi(\bu\bv^*)\|
      &= \sup_{\|\ba\| = \|\bm b\| = 1} |\ba^*\Phi(\bu\bv^*)\bm b| \\
      &= \sup_{\|\ba\| = \|\bm b\| = 1} \left|\sum_{i = 1}^n (\ba^*\bA_i\bu)(\bv^*\bA_i\bm b)\right| \\
      &\leq \sup_{\|\ba\| = \|\bm b\| = 1} \left(\sum_{i = 1}^n |\ba^*\bA_i\bu|^2\right)^{1/2} \left(\sum_{i = 1}^n |\bv^*\bA_i\bm b|^2\right)^{1/2} \\
      &\leq \|\bu\| \cdot \|\bv\| \cdot \sigma_*(\bX)^2,
    \end{align*}
    using the Cauchy-Schwarz inequality.
\end{proof}
\noindent
The upper bound of Theorem~\ref{thm:bbp-iso-free} then follows relatively easily:
\begin{lemma}[Upper bound on $\lambda_{\max}(\mathfrak{X}^{(d)})$]
    \label{lem:bbp-iso-free-upper}
    In the above setting, $\lambda_{\max}(\mathfrak{X}^{(d)}) \leq B(\theta) + o(1)$.
\end{lemma}
\begin{proof}
    For the upper bound, we follow essentially the same argument as for Theorem~\ref{thm:bbp-free}.
    We use Lehner's formula \eqref{eq:lehner-sc} and again consider feasible points $\bZ$ of the form
    \[ \bZ = \alpha \bP + \beta \bQ = (\alpha - \beta)\bP + \beta \bm I_d \]
    for $\alpha, \beta > 0$.
    The second form above will be useful in this case because it gives
    \[ \Phi(\bZ) = (\alpha - \beta)\Phi(\bP) + \beta \bm I_d \preceq (\beta + (\alpha + \beta)\sigma_*(\bX^{(d)})^2) \bm I_d \]
    where we use linearity and complete positivity of $\Phi$, \eqref{eq:Phi-Id}, and Proposition~\ref{prop:Phi-rank-one} on the first term, since $\bP = \bv\bv^*$.
    We then have
    \begin{align*}
      \lambda_{\max}(\mathfrak{X})
      &\leq \lambda_{\max}\bigl(\theta \bP + \bZ^{-1} + \Phi(\bZ)\bigr) \\
      &\leq \lambda_{\max}\bigl(\theta \bP + \alpha^{-1}\bP + \beta^{-1}\bQ + (\beta + (\alpha + \beta)\sigma_*(\bX^{(d)})^2) \bm I_d\bigr) \\
      &\leq \max\left\{\theta + \alpha^{-1} + \beta, \beta + \beta^{-1}\right\} + (\alpha + \beta)\sigma_*(\bX^{(d)})^2 \\
      &= B(\theta) + o(1)
    \end{align*}
    provided that we choose
    \[ \beta \colonequals \min\{1, \theta^{-1}\} \]
    not depending on $d$, and $\alpha = \alpha(d)$ such that
    \begin{equation}
        \omega(1) \leq \alpha(d) \leq o(\sigma_*(\bX^{(d)})^{-2}). \label{eq:alpha-cond}
    \end{equation}
    Since $\sigma_*(\bX^{(d)}) \to 0$ by assumption, this is always possible.
\end{proof}

For the lower bound, we use the tools we established in Lemma~\ref{lem:approx-slackness} and Proposition~\ref{prop:Psi}.
We first establish a few preliminaries about these.
Recall that, to $\bZ \succ \bm 0$, we associate a self-adjoint linear operator on $\CC^{d \times d}_{\herm}$ defined by
\[ \widetilde{\Psi}_{\bZ}(\bS) \colonequals \bZ^{1/2}\Phi(\bZ^{1/2}\bS\bZ^{1/2})\bZ^{1/2}. \]
We work with the same class of $\bZ$ as before, $\bZ = \alpha \bP + \beta \bQ$, and we still take $\beta \colonequals \min\{1, \theta^{-1}\}$, but let us consider $\alpha$ as a variable for this part of the argument and view $\bZ = \bZ(\alpha)$.
(Note that while morally we expect the best $\bZ$ here and from the previous part of the argument in Lemma~\ref{lem:bbp-iso-free-upper} to be close or identical, formally for the lower bound we are allowed to choose an entirely new $\bZ$; it is just an input into Lemma~\ref{lem:approx-slackness} and Proposition~\ref{prop:Psi}.)
For fixed $\beta$, we have the following:
\begin{proposition}
    \label{prop:lam-max-mon}
    Assume $\Phi(\bm I_d) = \bm I_d$.
    Fix $\beta > 0$, and for $\alpha > 0$ set $\bZ(\alpha) \colonequals \alpha\bP + \beta\bQ$ and
    \[
        f(\alpha) \colonequals \lambda_{\max}(\widetilde{\Psi}_{\bZ(\alpha)}).
    \]
    Then $f$ is continuous and non-decreasing on $(0, \infty)$.
    If $\beta < 1$, then there is a unique $\alpha > 0$ such that $f(\alpha) = 1$.
\end{proposition}
\begin{proof}
    Write $\bH(\alpha) \colonequals \bZ(\alpha)^{1/2}$.
    By Proposition~\ref{prop:Psi}, the top eigenvalue of $\widetilde{\Psi}_{\bZ(\alpha)}$ has a positive semidefinite eigenmatrix.
    Thus, after the change of variables $\bT = \bH(\alpha) \bS \bH(\alpha)$ in the Rayleigh quotient,
    \[
        f(\alpha)
        =
        \sup_{\substack{\bT \succeq \bm 0 \\ \bT \neq \bm 0}}
        \frac{\langle \bT, \Phi(\bT) \rangle}
        {\|\bH(\alpha)^{-1}\bT\bH(\alpha)^{-1}\|_F^2}.
    \]
    For $\bT \succeq \bm 0$,
    \[
        \langle \bT, \Phi(\bT) \rangle
        =
        \sum_i \Tr(\bT \bA_i \bT \bA_i)
        =
        \sum_i \Tr((\bT^{1/2}\bA_i\bT^{1/2})^2)
        \geq 0.
    \]
    On the other hand,
    \begin{align*}
      \|\bH(\alpha)^{-1}\bT\bH(\alpha)^{-1}\|_F^2
      &= \alpha^{-2}\|\bP\bT\bP\|_F^2
        + 2(\alpha\beta)^{-1}\|\bP\bT\bQ\|_F^2
        + \beta^{-2}\|\bQ\bT\bQ\|_F^2.
    \end{align*}
    For fixed $\bT$, this denominator is non-increasing in $\alpha$, while the numerator does not depend on $\alpha$.
    Therefore $f$ is non-decreasing.
    Continuity follows from the continuity of the entries of a matrix representation of the operator $\widetilde{\Psi}_{\bZ(\alpha)}$ as functions of $\alpha > 0$.

    We next prove uniqueness of $\alpha$ satisfying $f(\alpha) = 1$ under the assumption $\beta < 1$.
    First note that $\Phi$ is a contraction for the Frobenius norm.
    Indeed, $\Phi(\bm I_d) = \bm I_d$ and the $\bA_i$ are Hermitian, so $\Phi$ is unital and trace-preserving.
    Kadison's inequality \cite[Theorem 2.3.2]{Bhatia-2009-PositiveDefiniteMatrices} then gives $\Phi(\bY)^2 \preceq \Phi(\bY^2)$ for Hermitian $\bY$.
    Taking traces gives $\|\Phi(\bY)\|_F \leq \|\bY\|_F$.
    Since $\Phi(\bm I_d) = \bm I_d$, it follows that $\lambda_{\max}(\Phi) = 1$ as a self-adjoint operator on $\CC^{d \times d}_{\herm}$.
    Hence
    \[
        f(\beta) = \lambda_{\max}(\widetilde{\Psi}_{\beta \bm I_d}) = \beta^2\lambda_{\max}(\Phi) = \beta^2 < 1.
    \]

    We next show that $f$ is strictly increasing wherever $f > \beta^2$.
    Let $0 < \alpha_1 < \alpha_2$, and suppose $f(\alpha_1) > \beta^2$.
    Let $\bT \succeq \bm 0$ attain the above supremum for $\alpha_1$.
    If $\bT = \bQ\bT\bQ$, then
    \[
        \frac{\langle \bT, \Phi(\bT) \rangle}
        {\|\bH(\alpha_1)^{-1}\bT\bH(\alpha_1)^{-1}\|_F^2}
        \leq
        \frac{\|\bT\|_F^2}{\beta^{-2}\|\bT\|_F^2}
        =
        \beta^2,
    \]
    contradicting $f(\alpha_1) > \beta^2$.
    Thus $\bT \neq \bQ\bT\bQ$.
    The denominator in the quotient is therefore strictly smaller at $\alpha_2$ than at $\alpha_1$.
    Since this quotient has value $f(\alpha_1) > 0$, its numerator is positive.
    This gives $f(\alpha_2) > f(\alpha_1)$.

    Also, using the test matrix $\bm I_d/\sqrt{d}$ in the Rayleigh quotient, for $\alpha \geq \beta$ we have
    \[
        f(\alpha)
        \geq
        \frac{1}{d}\langle \bZ(\alpha), \Phi(\bZ(\alpha)) \rangle
        \geq
        \frac{\alpha\beta}{d},
    \]
    so in particular $f(\alpha) \to \infty$ as $\alpha \to \infty$.
    The last inequality uses $\alpha \geq \beta$, $\bZ(\alpha) \succeq \beta\bm I_d$, $\Phi(\bZ(\alpha)) \succeq \alpha\Phi(\bP)$, and $\Tr(\Phi(\bP)) = \Tr(\bP) = 1$.
    By continuity, there exists $\alpha > 0$ such that $f(\alpha) = 1$, and since $\beta^2 < 1$ this $\alpha$ is unique.
\end{proof}

The dual certificate we will use in the proof below, also that prescribed by Lemma~\ref{lem:approx-slackness}, is then as follows.
\begin{definition}[Spiked matrix model dual certificate]
    \label{def:dual-cert}
    Suppose that $\bA_1, \dots, \bA_n \in \RR^{d \times d}_{\sym}$ have $\sum_{i = 1}^n \bA_i^2 = \bm I_d$ and $\bA_0 = \theta \bv\bv^{\top}$ for $\bv \in \SS^{d - 1}$.
    Define the operator $\Phi: \CC^{d \times d}_{\herm} \to \CC^{d \times d}_{\herm}$ by
    \begin{align*}
      \Phi(\bS) &\colonequals \sum_{i = 1}^n \bA_i \bS \bA_i
                  \intertext{and, for each $\bZ \succ \bm 0$, define the operator $\widetilde{\Psi}_{\bZ}: \CC^{d \times d}_{\herm} \to \CC^{d \times d}_{\herm}$ by}
                  \widetilde{\Psi}_{\bZ}(\bS) &\colonequals \bZ^{1/2}\Phi(\bZ^{1/2}\bS\bZ^{1/2})\bZ^{1/2},
    \end{align*}
    both self-adjoint operators with respect to the Frobenius inner product.
    Suppose that $\theta > 1$ and let $\bZ(\alpha) \colonequals \alpha \bv\bv^{\top} + \theta^{-1} (\bm I_d - \bv\bv^{\top})$.
    There is a unique $\alpha > 0$ such that $\lambda_{\max}(\widetilde{\Psi}_{\bZ(\alpha)}) = 1$ (Proposition~\ref{prop:lam-max-mon}), in the above viewpoint on $\widetilde{\Psi}_{\bZ(\alpha)}$ as a self-adjoint operator.
    Up to rescaling, with this choice of $\alpha$, the eigenvalue 1 of $\widetilde{\Psi}_{\bZ(\alpha)}$ has a positive semidefinite associated eigenmatrix $\widetilde{\bS} \succeq \bm 0$.
    Writing $\bZ \colonequals \bZ(\alpha)$, set
    \[
        \bS \colonequals \frac{\bZ^{1/2}\widetilde{\bS}\bZ^{1/2}}{\Tr(\bZ^{1/2}\widetilde{\bS}\bZ^{1/2})}.
    \]
    Then $\bS \succeq \bm 0$, $\Tr(\bS) = 1$, and $\bS = \Psi_{\bZ}(\bS) = \bZ\Phi(\bS)\bZ$ by Proposition~\ref{prop:Psi}.
    We call this $\bS \equalscolon \bS^{\mathsf{est}} = \bS^{\mathsf{est}}(\theta, \bv, \bA_1, \dots, \bA_n)$.
\end{definition}

We note that $\bS^{\mathsf{est}}$ may not be well-defined, because the eigenvalue 1 above may have several positive semidefinite eigenmatrices.
This is not an issue for the argument below: one may simply pick any such eigenmatrix.
However, this might cast some doubt on the formulation of Conjectures~\ref{conj:spiked-matrix} and~\ref{conj:spiked-matrix-fluct}, which refer to a single $\bS^{\mathsf{est}}$.
Thus, before proceeding to the rest of the proof, we give the following simple condition for uniqueness of this eigenmatrix.

\begin{proposition}
    \label{prop:Phi-irreducible-transfer}
    In the above setting, suppose that $\Phi$ is irreducible (Definition~\ref{def:cp-irred}) and that $\bm I_d \in \operatorname{span}\{\bA_1, \dots, \bA_n\}$.
    Then, for every $\bZ \succ \bm 0$, $\widetilde{\Psi}_{\bZ}$ is irreducible, and thus by Theorem~\ref{thm:quantum-pf} the eigenvalue $\lambda_{\max}(\widetilde{\Psi}_{\bZ})$ is simple.
\end{proposition}
\begin{proof}
    Write $\bB_i \colonequals \bZ^{1/2}\bA_i\bZ^{1/2}$, so that
    \[
        \widetilde{\Psi}_{\bZ}(\bS) = \sum_{i = 1}^n \bB_i \bS \bB_i^*.
    \]
    Suppose $L \subset \CC^d$ is invariant under every $\bB_i$.
    Since $\bm I_d \in \operatorname{span}\{\bA_1, \dots, \bA_n\}$, the matrix $\bZ$ belongs to $\operatorname{span}\{\bB_1, \dots, \bB_n\}$.
    Thus $L$ is invariant under $\bZ$.
    Since $\bZ$ is Hermitian, $L$ is also invariant under $\bZ^{1/2}$ and $\bZ^{-1/2}$.
    Therefore every $\bA_i = \bZ^{-1/2}\bB_i\bZ^{-1/2}$ leaves $L$ invariant.
    By irreducibility of $\Phi$, $L$ is either $\{\bm 0\}$ or all of $\CC^d$.
\end{proof}

\begin{lemma}
    \label{lem:bbp-iso-free-lower}
    In the above setting, $\lambda_{\max}(\mathfrak{X}^{(d)}) \geq B(\theta) - o(1)$.
\end{lemma}
\begin{proof}
    In the case $\theta \leq 1$, choosing the dual certificate $\bS \colonequals \frac{1}{d}\bm I_d$ in the dual form \eqref{eq:sc-sdp-small-dual-2} immediately gives $\lambda_{\max}(\mathfrak{X}^{(d)}) \geq 2 + O(\frac{1}{d})$, giving the result since $B(\theta) = 2$ in this case.
    So, we restrict our attention below to the case $\theta > 1$.

    Let $\beta \colonequals \theta^{-1} < 1$, and $\bZ = \alpha \bP + \beta \bQ$ for that value of $\alpha = \alpha(d) > 0$ that, according to Proposition~\ref{prop:lam-max-mon}, is the unique such $\alpha$ such that $\lambda_{\max}(\widetilde{\Psi}_{\bZ(\alpha)}) = 1$.
    Let $\widetilde{\bS} \succeq \bm 0$ be a positive semidefinite eigenmatrix of $\widetilde{\Psi}_{\bZ}$ with eigenvalue $1$, normalized so that $\Tr(\widetilde{\bS}) = 1$.

    We first show that $\alpha(d) = \omega(1)$ as $d \to \infty$.
    Note that, as observed in the proof of Proposition~\ref{prop:lam-max-mon}, we will always have $\alpha > \beta$.
    We then have
    \begin{align*}
      1
      &= \lambda_{\max}(\widetilde{\Psi}_{\bZ(\alpha)}) \\
      &= \Tr(\widetilde{\Psi}_{\bZ(\alpha)}(\widetilde{\bS})) \\
      &= \langle \bZ, \Phi(\bZ^{1/2} \widetilde{\bS} \bZ^{1/2}) \rangle \\
      &= \langle \Phi(\bZ), \bZ^{1/2} \widetilde{\bS} \bZ^{1/2} \rangle \\
      &\leq \langle \alpha \Phi(\bP) + \beta \bm I_d, (\alpha \bP + \beta \bQ)^{1/2} \widetilde{\bS} (\alpha \bP + \beta \bQ)^{1/2} \rangle \\
      &\leq \beta \langle \alpha \bP + \beta \bQ, \widetilde{\bS} \rangle + \alpha \cdot \|\Phi(\bP)\| \cdot \Tr((\alpha \bP + \beta \bQ)^{1/2} \widetilde{\bS} (\alpha \bP + \beta \bQ)^{1/2}) \\
      &\leq (\beta + \alpha \sigma_*^2)\langle \alpha \bP + \beta \bQ, \widetilde{\bS} \rangle \\
      &= (\beta + \alpha \sigma_*^2)\langle (\alpha - \beta) \bP + \beta \bm I_d, \widetilde{\bS} \rangle \\
      &\leq (\beta + \alpha \sigma_*^2)(\beta + \alpha\langle \bP, \widetilde{\bS} \rangle).
    \end{align*}
    We are aiming for a contradiction if $\alpha$ were bounded as $d \to \infty$.
    To this end, we control $\langle \bP, \widetilde{\bS} \rangle$, using that $\widetilde{\Psi}_{\bZ}(\widetilde{\bS}) = \widetilde{\bS}$:
    \begin{align*}
      \langle \bP, \widetilde{\bS} \rangle
      &= \langle \widetilde{\Psi}_{\bZ}(\widetilde{\bS}), \bP \rangle \\
      &= \langle \bZ^{1/2}\Phi(\bZ^{1/2}\widetilde{\bS}\bZ^{1/2})\bZ^{1/2}, \bP \rangle \\
      &= \alpha \langle \Phi((\alpha \bP + \beta \bQ)^{1/2}\widetilde{\bS}(\alpha \bP + \beta \bQ)^{1/2}), \bP \rangle \\
      &= \alpha \langle (\alpha \bP + \beta \bQ)^{1/2}\widetilde{\bS}(\alpha \bP + \beta \bQ)^{1/2}, \Phi(\bP) \rangle \\
      &\leq \alpha \|\Phi(\bP)\| \Tr((\alpha \bP + \beta \bQ)^{1/2}\widetilde{\bS}(\alpha \bP + \beta \bQ)^{1/2}) \\
      &\leq \alpha \sigma_*^2 \langle \alpha \bP + \beta \bQ, \widetilde{\bS} \rangle \\
      &= \alpha \sigma_*^2 \langle (\alpha - \beta) \bP + \beta \bm I_d, \widetilde{\bS} \rangle \\
      &= \alpha \sigma_*^2 (\beta + (\alpha - \beta) \langle \bP, \widetilde{\bS} \rangle),
    \end{align*}
    and rearranging this gives
    \[ \langle \bP, \widetilde{\bS} \rangle \leq \frac{\alpha \beta \sigma_*^2}{1 - \alpha(\alpha - \beta) \sigma_*^2} \]
    provided $\alpha^2 \sigma_*^2 < 1$.
    Thus, since $\sigma_*^2 \to 0$ as $d \to \infty$, if $\alpha(d)$ were bounded as $d \to \infty$ then we would have $\langle \bP, \widetilde{\bS} \rangle \to 0$ and thus $1 \leq \beta^2 = \theta^{-2}$, a contradiction.

    We have $\bZ \succ \bm 0$, and thus $\bZ$ is invertible.
    Write $c \colonequals \lambda_{\max}(\theta \bP + \bZ^{-1} + \Phi(\bZ))$.
    For sufficiently large $d$ that $\alpha(d) \geq \beta = \theta^{-1}$, we have $\bZ \succeq \beta\bm I_d = \theta^{-1}\bm I_d$ and $\bZ^{-1} \succeq \beta^{-1}\bQ = \theta\bQ$, so this satisfies
    \begin{align*}
      c
      &\geq \lambda_{\max}(\theta \bP + \theta \bQ + \Phi(\theta^{-1} \bm I_d)) \\
      &= \lambda_{\max}(\theta \bm I_d + \theta^{-1} \bm I_d) \\
      &= \theta + \theta^{-1} \\
      &= B(\theta).
    \end{align*}
    By Proposition~\ref{prop:Psi}, since $\lambda_{\max}(\widetilde{\Psi}_{\bZ}) = 1$, we may choose $\bS \succeq \bm 0$ with $\Tr(\bS) = 1$ such that $\bS = \Psi_{\bZ}(\bS) = \bZ \Phi(\bS)\bZ$.
    Now, we plug this choice of $\bZ$ and $\bS$ into Lemma~\ref{lem:approx-slackness}.
    This gives the lower bound, again for sufficiently large $d$,
    \begin{align*}
      \lambda_{\max}(\mathfrak{X})
      &\geq c - \langle c \bm I_d - \theta \bP - \Phi(\bZ) - \bZ^{-1}, \bS \rangle \\
      &= \theta \langle \bP, \bS \rangle + \langle \Phi(\bZ), \bS \rangle + \langle \bS, \bZ^{-1} \rangle
        \intertext{and further, using that $\bS = \bZ \Phi(\bS)\bZ$, we have $\langle \Phi(\bZ), \bS \rangle = \langle \bZ, \Phi(\bS) \rangle = \Tr(\bZ\Phi(\bS)) = \Tr(\bS\bZ^{-1}) = \langle \bS, \bZ^{-1} \rangle$, whereby}
      &= \theta \langle \bP, \bS \rangle + 2\langle \bS, \bZ^{-1} \rangle \\
      &= \theta \langle \bP, \bS \rangle + 2\left\langle \bS, \frac{1}{\alpha} \bP + \frac{1}{\beta} \bQ \right\rangle \\
      &= \frac{2}{\beta} + \left(\theta + 2\left(\frac{1}{\alpha} - \frac{1}{\beta}\right)\right) \langle \bP, \bS \rangle.
    \end{align*}

    We now identify $\langle \bP, \bS \rangle$ from the fixed point property of $\bS$.
    Indeed, we have
    \begin{align*}
      \langle \bP, \bS \rangle
      &= \langle \bP, \bZ \Phi(\bS) \bZ \rangle \\
      &= \langle \bZ\bP\bZ, \Phi(\bS) \rangle \\
      &= \alpha^2 \langle \bP, \Phi(\bS) \rangle \\
      &= \alpha^2 \langle \Phi(\bP), \bS \rangle,
    \end{align*}
    and also
    \begin{align*}
      1
      &= \Tr(\bS) \\
      &= \Tr(\bZ\Phi(\bS)\bZ) \\
      &= \langle \bZ^2, \Phi(\bS) \rangle \\
      &= \langle \Phi((\alpha^2 - \beta^2)\bP + \beta^2 \bm I_d), \bS \rangle \\
      &= \beta^2 + (\alpha^2 - \beta^2) \langle \Phi(\bP), \bS \rangle \\
      &= \beta^2 + \frac{\alpha^2 - \beta^2}{\alpha^2} \langle \bP, \bS \rangle,
    \end{align*}
    and solving gives the exact expression
    \[ \langle \bP, \bS \rangle = \frac{\alpha^2(1 - \beta^2)}{\alpha^2 - \beta^2}. \]

    Substituting into our lower bound,
    \begin{align*}
      \lambda_{\max}(\mathfrak{X})
      &\geq \frac{2}{\beta} + \left(\theta + 2\left(\frac{1}{\alpha} - \frac{1}{\beta}\right)\right)\frac{\alpha^2(1 - \beta^2)}{\alpha^2 - \beta^2}
        \intertext{and using that $\alpha = \alpha(d) = \omega(1)$ as we established above, this is}
      &= B(\theta) + o(1),
    \end{align*}
    completing the proof.
\end{proof}

\subsection{Eigenvector fluctuations: Evidence for Conjecture~\ref{conj:spiked-matrix-fluct}}
\label{sec:conj-evidence}

\begin{figure}[p]
    \centering
    \begingroup
    \setlength{\tabcolsep}{0.4pt}
    \renewcommand{\arraystretch}{0.55}
    \newcommand{\spikedcovrowlabel}[1]{\rotatebox{90}{\parbox{0.22\textwidth}{\centering\scriptsize #1}}}
    {\scriptsize
      \begin{tabular}{@{}c c c c c c@{}}
        & Noise Model 1: & Noise Model 2: & Noise Model 3: & Noise Model 4: \\
        & GOE & Block Independent & Smooth Independent & Correlated & \\[0.375em]
        \spikedcovrowlabel{Variance Profile}
        & \includegraphics[width=0.22\textwidth]{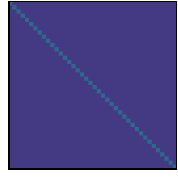}
        & \includegraphics[width=0.22\textwidth]{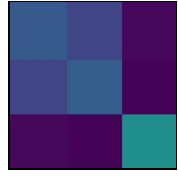}
        & \includegraphics[width=0.22\textwidth]{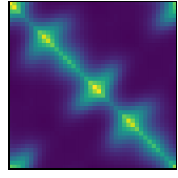}
        & \includegraphics[width=0.22\textwidth]{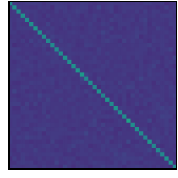}
        & \raisebox{0.009\textwidth}[0pt][0pt]{\makebox[0.075\textwidth][l]{\includegraphics[height=0.210\textwidth]{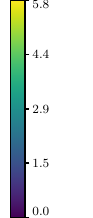}}} \\
        \spikedcovrowlabel{Off-Diagonal Predicted\\$d \cdot \bQ\,\bm\Sigma^{\mathsf{est}}\,\bQ^{\top}$}
        & \includegraphics[width=0.22\textwidth]{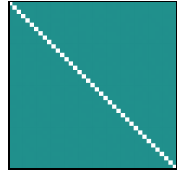}
        & \includegraphics[width=0.22\textwidth]{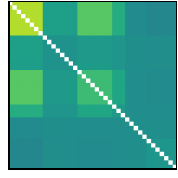}
        & \includegraphics[width=0.22\textwidth]{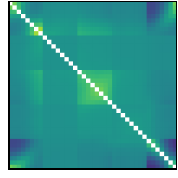}
        & \includegraphics[width=0.22\textwidth]{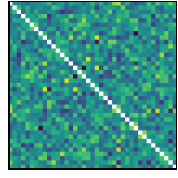}
        & \raisebox{-0.106\textwidth}[0pt][0pt]{\makebox[0.075\textwidth][l]{\includegraphics[height=0.210\textwidth]{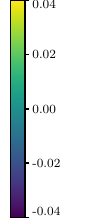}}} \\
        \spikedcovrowlabel{Off-Diagonal Empirical\\$d \cdot \bQ\Cov[\widehat{\bv}]\,\bQ^{\top}$}
        & \includegraphics[width=0.22\textwidth]{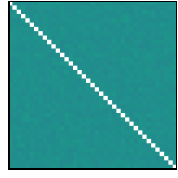}
        & \includegraphics[width=0.22\textwidth]{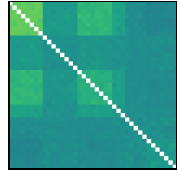}
        & \includegraphics[width=0.22\textwidth]{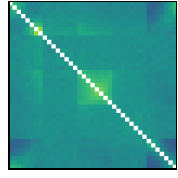}
        & \includegraphics[width=0.22\textwidth]{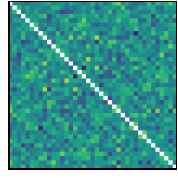}
        & \\
        \spikedcovrowlabel{Diagonal Predicted\\$d \cdot \bQ\,\bm\Sigma^{\mathsf{est}}\,\bQ^{\top}$}
        & \includegraphics[width=0.22\textwidth]{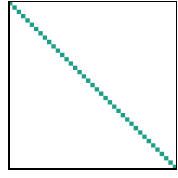}
        & \includegraphics[width=0.22\textwidth]{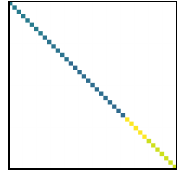}
        & \includegraphics[width=0.22\textwidth]{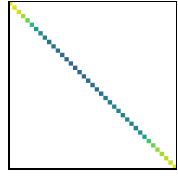}
        & \includegraphics[width=0.22\textwidth]{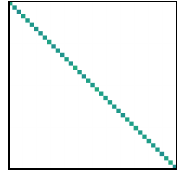}
        & \raisebox{-0.106\textwidth}[0pt][0pt]{\makebox[0.075\textwidth][l]{\includegraphics[height=0.210\textwidth]{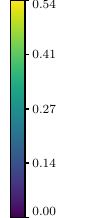}}} \\
        \spikedcovrowlabel{Diagonal Empirical\\$d \cdot \bQ\Cov[\widehat{\bv}]\,\bQ^{\top}$}
        & \includegraphics[width=0.22\textwidth]{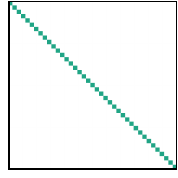}
        & \includegraphics[width=0.22\textwidth]{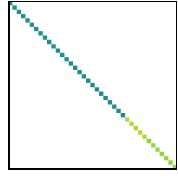}
        & \includegraphics[width=0.22\textwidth]{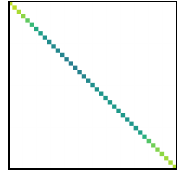}
        & \includegraphics[width=0.22\textwidth]{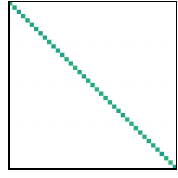}
        & \\
    \end{tabular}}
    \endgroup
    \caption{\textbf{Eigenvector fluctuation covariance estimates.}
      We compare an empirical estimate of the eigenvector covariance $d \cdot \bQ\Cov[\widehat{\bv}]\,\bQ^{\top}$ with the prediction $\Cov[\widehat{\bv}] \approx \bm\Sigma^{\mathsf{est}}$ of Conjecture~\ref{conj:spiked-matrix-fluct} under various noise models.
      See Section~\ref{sec:conj-evidence} for description of the noise models and choice of signal vector $\bv$.
      Here $\bQ \in \RR^{(d - 1) \times d}$ has in its rows an orthonormal basis for the orthogonal complement of the $\bv$ direction.
      The top row shows the variance profile $d \cdot \sigma_{ij}^2$ for the noise model.
      The two middle rows compare the predicted and empirical off-diagonal entries, and the bottom two rows the predicted and empirical diagonal entries.}
    \label{fig:spiked-covariance-table}
\end{figure}

\begin{figure}[p]
    \centering
    \begingroup
    \setlength{\tabcolsep}{0.4pt}
    \renewcommand{\arraystretch}{0.55}
    \newlength{\spikedqqcellsize}
    \setlength{\spikedqqcellsize}{0.22\textwidth}
    \newlength{\spikedqqplotsize}
    \setlength{\spikedqqplotsize}{0.211\textwidth}
    \newlength{\spikedqqrowshift}
    \setlength{\spikedqqrowshift}{4pt}
    \newcommand{\spikedqqrowlabel}[1]{\rotatebox{90}{\parbox{0.22\textwidth}{\centering\scriptsize #1}}}
    {\scriptsize
      \begin{tabular}{@{}c c c c c c@{}}
        & Noise Model 1: & Noise Model 2: & Noise Model 3: & Noise Model 4: \\
        & GOE & Block Independent & Smooth Independent & Correlated & \\[0.375em]
        \spikedqqrowlabel{Variance Profile}
        & \makebox[\spikedqqcellsize][l]{\includegraphics[width=\spikedqqcellsize]{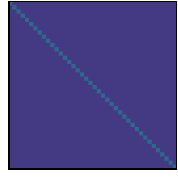}}
        & \makebox[\spikedqqcellsize][l]{\includegraphics[width=\spikedqqcellsize]{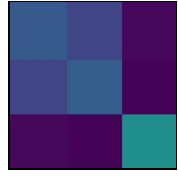}}
        & \makebox[\spikedqqcellsize][l]{\includegraphics[width=\spikedqqcellsize]{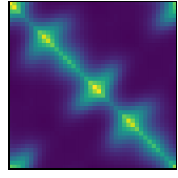}}
        & \makebox[\spikedqqcellsize][l]{\includegraphics[width=\spikedqqcellsize]{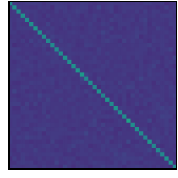}}
        & \raisebox{0.009\textwidth}[0pt][0pt]{\makebox[0.075\textwidth][l]{\includegraphics[height=0.210\textwidth]{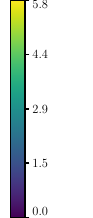}}} \\[-0.5em]
        \spikedqqrowlabel{Sparse Direction 1\\Standardized Q--Q}
        & \makebox[\spikedqqcellsize][l]{\hspace*{\spikedqqrowshift}\includegraphics[width=\spikedqqplotsize]{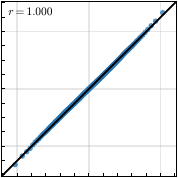}}
        & \makebox[\spikedqqcellsize][l]{\hspace*{\spikedqqrowshift}\includegraphics[width=\spikedqqplotsize]{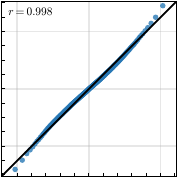}}
        & \makebox[\spikedqqcellsize][l]{\hspace*{\spikedqqrowshift}\includegraphics[width=\spikedqqplotsize]{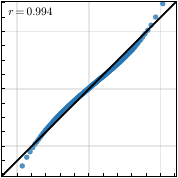}}
        & \makebox[\spikedqqcellsize][l]{\hspace*{\spikedqqrowshift}\includegraphics[width=\spikedqqplotsize]{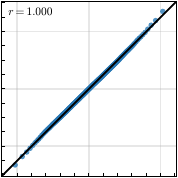}}
        & \\
        \spikedqqrowlabel{Sparse Direction 2\\Standardized Q--Q}
        & \makebox[\spikedqqcellsize][l]{\hspace*{\spikedqqrowshift}\includegraphics[width=\spikedqqplotsize]{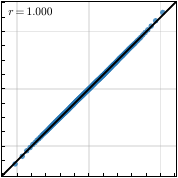}}
        & \makebox[\spikedqqcellsize][l]{\hspace*{\spikedqqrowshift}\includegraphics[width=\spikedqqplotsize]{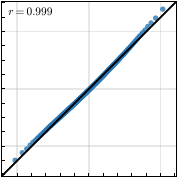}}
        & \makebox[\spikedqqcellsize][l]{\hspace*{\spikedqqrowshift}\includegraphics[width=\spikedqqplotsize]{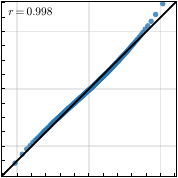}}
        & \makebox[\spikedqqcellsize][l]{\hspace*{\spikedqqrowshift}\includegraphics[width=\spikedqqplotsize]{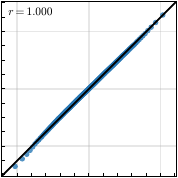}}
        & \\
        \spikedqqrowlabel{Dense Sine Direction\\Standardized Q--Q}
        & \makebox[\spikedqqcellsize][l]{\hspace*{\spikedqqrowshift}\includegraphics[width=\spikedqqplotsize]{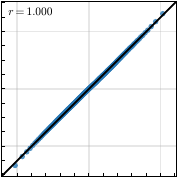}}
        & \makebox[\spikedqqcellsize][l]{\hspace*{\spikedqqrowshift}\includegraphics[width=\spikedqqplotsize]{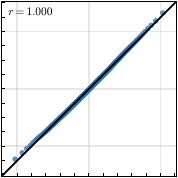}}
        & \makebox[\spikedqqcellsize][l]{\hspace*{\spikedqqrowshift}\includegraphics[width=\spikedqqplotsize]{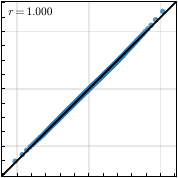}}
        & \makebox[\spikedqqcellsize][l]{\hspace*{\spikedqqrowshift}\includegraphics[width=\spikedqqplotsize]{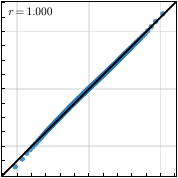}}
        & \\
        \spikedqqrowlabel{Dense Random Direction\\Standardized Q--Q}
        & \makebox[\spikedqqcellsize][l]{\hspace*{\spikedqqrowshift}\includegraphics[width=\spikedqqplotsize]{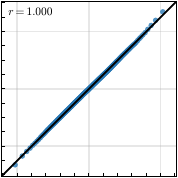}}
        & \makebox[\spikedqqcellsize][l]{\hspace*{\spikedqqrowshift}\includegraphics[width=\spikedqqplotsize]{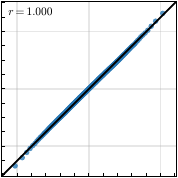}}
        & \makebox[\spikedqqcellsize][l]{\hspace*{\spikedqqrowshift}\includegraphics[width=\spikedqqplotsize]{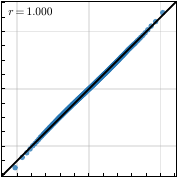}}
        & \makebox[\spikedqqcellsize][l]{\hspace*{\spikedqqrowshift}\includegraphics[width=\spikedqqplotsize]{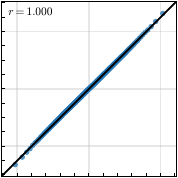}}
        & \\
    \end{tabular}}
  \endgroup
  \vspace{0.125em}
    \caption{\textbf{Eigenvector fluctuation Gaussianity.}
      We demonstrate the near-Gaussianity of inner products $\langle \bx, \widehat{\bv}\rangle$ per Conjecture~\ref{conj:spiked-matrix-fluct} for various choices of test direction $\bx$ and various noise models.
      See Section~\ref{sec:conj-evidence} for description of the noise models, choice of signal vector $\bv$, and choice of test directions.
      The top row repeats the four variance profiles from Figure~\ref{fig:spiked-covariance-table}.
      In the remaining rows, each panel is a Q--Q plot for $\langle \bx,\widehat{\bv}\rangle$ after subtracting its empirical mean and dividing by its empirical standard deviation.
      The comparison distribution is the standard Gaussian $\sN(0, 1)$.}
    \label{fig:spiked-gaussianity-table}
\end{figure}

We now provide evidence for various aspects of Conjecture~\ref{conj:spiked-matrix-fluct}, which we recall proposes a limiting distribution for fluctuations of the leading eigenvector $\what{\bv} = \bv_{\max}(\bX)$ of a spiked matrix model $\bX = \theta \bv\bv^{\top} + \bW \in \RR^{d \times d}_{\sym}$ around the component $\sqrt{1 - \theta^{-2}}\,\bv$ correlated with the signal $\bv$, in directions orthogonal to $\bv$.
We normalize such models in accordance with Theorem~\ref{thm:BCSVH-BBP}, such that $\|\bv\| = 1$ and $\EE \bW^2 = \bm I_d$.

We consider the following four Gaussian noise models for $\bW$.
We write $\sigma_{ij}^2 \colonequals \Var[W_{ij}]$ for the variance profile, though we recall that we allow the entries of $\bW$ to be correlated, so this matrix does not contain all of the information of the law of $\bW$ unless the entries are independent (up to symmetry).
\begin{enumerate}
\item \textbf{Gaussian orthogonal ensemble (GOE):} For reference, we include the case of GOE noise (Definition~\ref{def:goe}) that has been studied by many previous works, rescaled such that $\EE \bW^2 = \bm I_d$. Recall that, as prior work has shown (e.g., \cite{CH-2012-FluctuationsEigenvectorSpikedMatrix, LCC-2024-GaussianFluctuationsEigenvectors, CK-2025-EigenvectorFluctuationsSpikedMatrix}), here by the rotational symmetry of this model we expect the fluctuations of $\what{\bv}$ in the direction orthogonal to $\bv$ to be isotropic Gaussian, with no correlations.
\item \textbf{Independent entries, block variance profile:} We consider a model with entries independent up to symmetry and with the matrix of $\sigma_{ij}^2$ having a block structure.
    Namely, for $b_i \colonequals \lfloor 3(i - 1)/d \rfloor + 1$, we set $\sigma_{ij}^2 = \frac{1}{d} B_{b_i b_j}$ with values roughly
    \[
        \bB
        =
        \begin{bmatrix}
        1.6165 & 1.2102 & 0.1258 \\
        1.2102 & 1.7156 & 0.0580 \\
        0.1258 & 0.0580 & 2.8834
    \end{bmatrix}.
\]
We note that the normalization is such that $\sum_{j = 1}^d \sigma_{ij}^2 = 1$ for all $i$, which gives $\EE \bW^2 = \bm I_d$.
The quality of results is not very sensitive to this choice; we work with an example that makes the presence of covariance in the eigenvector fluctuations clearer.
\item \textbf{Independent entries, smooth variance profile:} We again consider independent entries, but now with smoothly varying variances $\sigma_{ij}^2 = \frac{1}{d} D_iD_j f(\frac{i - 1}{d}, \frac{j - 1}{d})$, where we choose the $D_i$ to ensure that again we have the isotropy condition $\sum_{j = 1}^d \sigma_{ij}^2 = 1$ (these may be computed by the Sinkhorn algorithm).
    In our experiments, we use the function
    \begin{align*}
      f(x, y) &= \num{0.025} + \exp\big(\num{-8.69}\min\{|x - y|, 1 - |x - y|\} + \num{0.65}\cos(4\pi(x + y) + \num{0.2}) \\
              &\hspace{3cm} + \num{0.35} \sin(6\pi(x + y) + \num{0.7}) + \num{0.15}\cos(2\pi(x - y))\big).
    \end{align*}
    Again, this rather convoluted example is chosen purely for the sake of visual illustration.
\item \textbf{Correlated noise:} Finally, we propose a model such that the entries of $\bW$ do have a non-trivial correlation structure.
    For a parameter $R$ and orthonormal bases $\{\bq_{r,1}, \dots, \bq_{r,d}\}$ of $\RR^d$ for each $r \in [R]$, we set $\bW \colonequals \frac{1}{\sqrt{R}}\sum_{r = 1}^R \sum_{k = 1}^d g_{r,k} \bq_{r,k}\bq_{r,k}^{\top}$, normalized so that $\EE \bW^2 = \bm I_d$.
    In our experiments we take $R = 16$ and the orthonormal bases are chosen uniformly at random from the Haar measure on orthogonal matrices.
\end{enumerate}
We illustrate the variance profiles $\sigma_{ij}^2$ of these models in the top row of Figures~\ref{fig:spiked-covariance-table} and~\ref{fig:spiked-gaussianity-table}, but we emphasize that for Noise Model 4 this does not contain the full information of the model.

In our experiments, we consider a signal vector $\bv$ that also has a block structure, but with a different number of blocks from the matrix in Noise Model 2, again to emphasize the intricate covariance structures that can appear and that depend on both the signal vector and the noise model.
For $c_i \colonequals \lfloor 5(i - 1)/d \rfloor + 1$, we set $v^{(0)}_i = a_{c_i}$ for $a = [\,6 \,\,\, 1 \,\,\, 4 \,\,\, 2 \,\,\, 8\,]$, and then normalize $\bv = \bv^{(0)} / \|\bv^{(0)}\|$.

We consider relatively small experiments, with $d = 40$ and signal strength $\theta = \num{1.8}$.
We note that it is not obvious that Conjecture~\ref{conj:spiked-matrix-fluct} should yet ``kick in'' at such small sizes, but we will see that it does quite clearly.

In Figure~\ref{fig:spiked-covariance-table}, we first consider the covariance structure of $\what{\bv}$, the leading eigenvector of $\bX$, drawn with each noise model above.
For each noise model, we sample \num{150000} such matrices. In the third and fifth rows of Figure~\ref{fig:spiked-covariance-table}, we plot the associated empirical estimate of $\Cov[\what{\bv}]$, separating diagonal and off-diagonal entries since the former are much larger.
In the second and fourth rows we plot our prediction of the same matrix from Conjecture~\ref{conj:spiked-matrix-fluct}, using the version with $\bm\Sigma^{\mathsf{est}}$ as defined in Definition~\ref{def:dual-cert} for the sake of efficiency (since computing this requires merely solving an eigenvalue problem as opposed to a full SDP).
Already at this small size and for several quite different noise models, we see remarkable agreement between the predicted and empirical values, the prediction capturing all major qualitative phenomena of the actual covariance.
We emphasize also that the empirical covariances show that the eigenvector fluctuations are indeed non-universal in these models and can have a broad range of structures for different noise distributions.
We also remark that there seems to be a systematic finite-size effect of the prediction of Conjecture~\ref{conj:spiked-matrix-fluct} overestimating the range of variances and covariances, making larger values larger and smaller values smaller.
Thus, perhaps if these estimates were to be used in practice a shrinkage procedure might be useful.

Second, in Figure~\ref{fig:spiked-gaussianity-table}, we consider the predicted Gaussianity of finite-dimensional projections of $\what{\bv}$.
For four choices of test vector $\bx$ fixed in advance, for each noise model we again draw \num{30000} independent samples of $\what{\bv}$ and plot Q--Q plots of a normalization of $\langle \bx, \what{\bv} \rangle$ to have mean zero and variance 1.
(Conjecture~\ref{conj:spiked-matrix-fluct} gives a prediction of the variance, but as we see in the first set of experiments this estimate seems to have a systematic bias, so we normalize separately to isolate the test of Gaussianity from this bias.)
The first two choices of $\bx$ are sparse two-coordinate directions orthogonal to $\bv$: if $i \neq j$, we use the normalized vector proportional to $v_j\be_i - v_i\be_j$.
In the experiment we take the pairs $(i,j) = (1,9)$ and $(17,33)$.
The other two choices are dense directions obtained by projecting the $\bv$ direction away from a vector and then normalizing.
The first initial dense vector has oscillating entries $\sin(2\pi(i - 1)/d) + 0.35\cos(6\pi(i - 1)/d)$, and the second is a standard Gaussian random vector.
We see an excellent fit to a Gaussian distribution, with correlations of Q--Q plots at least \num{0.99} in all cases.
For the cases of a sparse test direction $\bx$ having slightly mismatched tail distributions, we see (in additional experiments omitted here) improvement and clear convergence to Gaussianity as we increase $d$, with an essentially perfect fit once $d = 120$.

\section{Contributions by artificial intelligence}
\label{sec:ai}

The GPT-5.5 family of large language models (specifically, GPT-5.5~Pro accessed via the ChatGPT interface as well as the version of GPT-5.5 available in the Codex interface) was used for certain parts of this work.

Their most significant role was in the proof of Lemma~\ref{lem:rad-sdp-2}.
As mentioned above, we first derived the bound in Lemma~\ref{lem:rad-sdp-1} independently.
We believed that the SDP treated there (also that appearing in Theorem~\ref{thm:rad-sdp}) could only give a lower bound on $\Lambda_{\rad}$ and that this bound should sometimes be loose.
This incorrect belief was based on experiments with a very slightly different and also valid SDP relaxation of $\Lambda_{\rad}$ that \emph{is} indeed loose in some cases.
When prompted to verify and look for simple examples of this looseness for the SDP of Theorem~\ref{thm:rad-sdp}, GPT-5.5~Pro insisted that the bound was never loose and instead produced the proofs given in Section~\ref{sec:pf-rad-sdp}, presented here after editing for notation and clarity of exposition.

The other uses of these models involved more extensive interaction with and guidance by the author.
Beyond routine uses in programming numerical experiments and the above proof, the largest role the models played was in developing the details of the proof ideas in Section~\ref{sec:bbp-iso}.

\section*{Acknowledgments}
\addcontentsline{toc}{section}{Acknowledgments}

Thanks to Afonso Bandeira, Ramon van Handel, Lucas Pesenti, and Robert Wang for helpful discussions and comments over the course of this project.

\addcontentsline{toc}{section}{References}
\bibliography{main}
\bibliographystyle{alpha}

\clearpage

\appendix

\section{Lagrangian duality calculations}

We give the omitted calculations of Lagrangian duals of semidefinite programs from Theorems~\ref{thm:sc-sdp-small} and~\ref{thm:sc-sdp-large}.
Both are entirely routine semidefinite programming calculations.

\subsection{Small semicircular SDP: Omitted calculation from Theorem~\ref{thm:sc-sdp-small}}
\label{app:dual-small}
  Let
    \[
        \bM(c, \bZ)
      \colonequals
      \left[\begin{array}{cc} c \bm I_d - \bA_0 - \Phi(\bZ) & \bm I_d \\ \bm I_d & \bZ \end{array}\right].
    \]
    For a Hermitian dual variable
    \[
      \bm \bY
      = \bY(\bS, \bT, \bU) = \left[\begin{array}{cc} \bm S & -\bm T \\ -\bm T^{*} & \bm U \end{array}\right] \succeq \bm 0,
    \]
    the Lagrangian is
    \[
      \sL(c, \bZ; \bm S, \bm T, \bm U)
      \colonequals c - \langle \bm \bY(\bS, \bT, \bU), \bM(c, \bZ) \rangle.
    \]
    We note that $\Phi$ is self-adjoint with respect to the Frobenius inner product:
    \[
      \langle \bm S, \Phi(\bZ) \rangle
      = \sum_{i = 1}^n \Tr(\bm S \bA_i \bZ \bA_i)
      = \sum_{i = 1}^n \Tr(\bA_i \bm S \bA_i \bZ)
      = \langle \Phi(\bm S), \bZ \rangle.
    \]
    Using this, we calculate
    \[ \sL(c, \bZ; \bm S, \bm T, \bm U) = c(1 - \Tr(\bm S)) + \langle \bA_0, \bm S \rangle + 2\Re(\Tr(\bm T)) + \langle \Phi(\bm S) - \bm U, \bZ \rangle. \]
    Because $c$ is unconstrained and $\bZ$ ranges over all Hermitian matrices, the infimum of $\sL$ over $(c, \bZ)$ is finite if and only if $\Tr(\bS) = 1$ and $\bU = \Phi(\bS)$.
    In that case, the infimum of $\sL$ over $(c, \bZ)$ is $\langle \bA_0, \bm S \rangle + 2\Re(\Tr(\bm T))$.
    Therefore the dual problem is
    \[
\left\{ \begin{array}{ll} \text{supremum of} & \langle \bA_0, \bm S \rangle + 2\Re(\Tr(\bm T)) \\ \text{subject to} & \left[\begin{array}{cc} \bm S & -\bm T \\ -\bm T^{*} & \Phi(\bS) \end{array}\right] \succeq \bm 0, \\ & \bS \in \CC^{d \times d}_{\herm}, \bT \in \CC^{d \times d}, \\ & \Tr(\bS) = 1 \end{array}\right\},
\]
the same as \eqref{eq:sc-sdp-small-dual} from Theorem~\ref{thm:sc-sdp-small}.

\subsection{Large semicircular SDP: Omitted calculation from Theorem~\ref{thm:sc-sdp-large}}
\label{app:dual-large}

Let
\[ \bH(c, \bZ)
      \colonequals
      \left[\begin{array}{ccccc} c\bm I_d - \bA_0 - \bZ & \bA_1 & \bA_2 & \cdots & \bA_n \\ \bA_1 & \bZ & \bm 0 & \cdots & \bm 0 \\ \bA_2 & \bm 0 & \bZ & \cdots & \bm 0 \\
                                                                                     \vdots & \vdots & \vdots & \ddots & \vdots \\ \bA_n & \bm 0 & \bm 0 & \cdots & \bZ \end{array}\right]. \]
Let $\bm \Gamma = [\bm \Gamma^{[i, j]}]_{i, j = 0}^n \succeq \bm 0$ be a Hermitian block matrix of the same size as $\bH(c, \bZ)$, the dual variable for the positivity constraint on $\bH$.
    The Lagrangian is
    \begin{align*}
      \sL(c, \bZ; \bm \Gamma)
      &\colonequals c - \langle \bm \Gamma, \bH(c, \bZ) \rangle
        \intertext{and expanding the inner product blockwise gives}
      &= c\bigl(1 - \Tr({\bm \Gamma}^{[0, 0]})\bigr)
      + \langle \bA_0, {\bm \Gamma}^{[0, 0]} \rangle - \sum_{i = 1}^n \bigl(\langle \bA_i, {\bm \Gamma}^{[0, i]} \rangle + \langle \bA_i, {\bm \Gamma}^{[i, 0]} \rangle\bigr) \\
      &\hspace{1cm}
      + \left\langle {\bm \Gamma}^{[0, 0]} - \sum_{i = 1}^n {\bm \Gamma}^{[i, i]}, \bZ \right\rangle.
    \end{align*}

    As in the proof of Theorem~\ref{thm:sc-sdp-small}, the infimum of this Lagrangian over unconstrained $c \in \RR$ and Hermitian $\bZ$ is finite if and only if $\Tr({\bm \Gamma}^{[0, 0]}) = 1$ and ${\bm \Gamma}^{[0, 0]} = \sum_{i = 1}^n {\bm \Gamma}^{[i, i]}$.
    In this case, the infimum of $\sL$ over $(c, \bZ)$ is
    \[
      \langle \bA_0, {\bm \Gamma}^{[0, 0]} \rangle
      - \sum_{i = 1}^n \bigl(\langle \bA_i, {\bm \Gamma}^{[0, i]} \rangle + \langle \bA_i, {\bm \Gamma}^{[i, 0]} \rangle\bigr) =
      \left\langle \left[\begin{array}{cccc} \bA_0 & -\bA_1 & \cdots & -\bA_n \\ -\bA_1 & \bm 0 & \cdots & \bm 0 \\ \vdots & \vdots & \ddots & \vdots \\ -\bA_n & \bm 0 & \cdots & \bm 0 \end{array}\right], \bm \Gamma \right\rangle.
  \]
  Therefore the dual problem is
  \[ \left\{\begin{array}{ll} \text{supremum of} & \left\langle \left[\begin{array}{cccc} \bA_0 & -\bA_1 & \cdots & -\bA_n \\ -\bA_1 & \bm 0 & \cdots & \bm 0 \\ \vdots & \vdots & \ddots & \vdots \\ -\bA_n & \bm 0 & \cdots & \bm 0 \end{array}\right], \left[\begin{array}{cccc} \bm \Gamma^{[0, 0]} & \bm \Gamma^{[0, 1]} & \cdots & \bm \Gamma^{[0, n]} \\ \bm \Gamma^{[1, 0]} & \bm \Gamma^{[1, 1]} & \cdots & \bm\Gamma^{[1, n]} \\ \vdots & \vdots & \ddots & \vdots \\ \bm\Gamma^{[n, 0]} & \bm\Gamma^{[n, 1]} & \cdots & \bm \Gamma^{[n, n]} \end{array}\right]\right\rangle
              \\ \text{subject to} & \bm \Gamma = [\bm \Gamma^{[i, j]}]_{i, j = 0}^n \succeq \bm 0, \\ & \bm \Gamma^{[0, 0]} = \sum_{i = 1}^n \bm \Gamma^{[i, i]}, \\ & \Tr(\bm \Gamma^{[0, 0]}) = 1 \end{array}\right\}, \]
  the same as \eqref{eq:sc-sdp-large-dual} from Theorem~\ref{thm:sc-sdp-large}.

\subsection{Rademacher SDP: Omitted calculation from Theorem~\ref{thm:rad-sdp}}
\label{app:rad}

Let
\[
  \bK(c, \bZ, \bY_1, \dots, \bY_n)
  \colonequals c \bm I_d - \bA_0 - \bZ - \sum_{i = 1}^n \bY_i
\]
and, for each $i \in [n]$, let
\[
  \bM_i(\bY_i, \bZ)
  \colonequals
  \left[\begin{array}{cc}
    \bY_i & \bA_i \\
    \bA_i & \bY_i + \bZ
  \end{array}\right].
\]
Then the primal problem \eqref{eq:rad-sdp-primal} can be written as
\[
  \left\{
  \begin{array}{ll}
    \text{minimum of} & c \\
    \text{subject to} & \bK(c, \bZ, \bY_1, \dots, \bY_n) \succeq \bm 0, \\
    & \bZ \succeq \bm 0, \\
    & \bM_i(\bY_i, \bZ) \succeq \bm 0 \text{ for all } i \in [n].
  \end{array}
  \right\}.
\]

For Hermitian dual variables
\[
  \bm \Delta \succeq \bm 0,
  \qquad
  \bm U \succeq \bm 0,
  \qquad
  \bm \Gamma_i
  =
  \left[\begin{array}{cc}
    \bS_i & -\bR_i \\
    -\bR_i^{*} & \bT_i
  \end{array}\right]
  \succeq \bm 0
  \text{ for all } i \in [n],
\]
the Lagrangian is
\begin{align*}
  &\sL(c, \bZ, \bY_1, \dots, \bY_n; \bm \Delta, \bm U, \bm \Gamma_1, \dots, \bm \Gamma_n) \\
  &\colonequals
  c
  - \langle \bm \Delta, \bK(c, \bZ, \bY_1, \dots, \bY_n) \rangle
  - \langle \bm U, \bZ \rangle
  - \sum_{i = 1}^n \langle \bm \Gamma_i, \bM_i(\bY_i, \bZ) \rangle.
\end{align*}
Expanding the inner products gives
\begin{align*}
  \sL
  &=
  c
  - \left\langle \bm \Delta, c \bm I_d - \bA_0 - \bZ - \sum_{i = 1}^n \bY_i \right\rangle
  - \langle \bm U, \bZ \rangle \\
  &\hspace{1cm}
  - \sum_{i = 1}^n
  \left\langle
    \left[\begin{array}{cc}
      \bS_i & -\bR_i \\
      -\bR_i^{*} & \bT_i
    \end{array}\right],
    \left[\begin{array}{cc}
      \bY_i & \bA_i \\
      \bA_i & \bY_i + \bZ
    \end{array}\right]
  \right\rangle \\
  &=
  c(1 - \Tr(\bm \Delta))
  + \langle \bA_0, \bm \Delta \rangle
  + \left\langle \bm \Delta - \bm U - \sum_{i = 1}^n \bT_i, \bZ \right\rangle \\
  &\hspace{1cm}
  + \sum_{i = 1}^n \left\langle \bm \Delta - \bS_i - \bT_i, \bY_i \right\rangle
  + \sum_{i = 1}^n \langle \bA_i, \bR_i + \bR_i^{*} \rangle.
\end{align*}
Indeed, for each $i \in [n]$ we have
\[
  \left\langle
    \left[\begin{array}{cc}
      \bS_i & -\bR_i \\
      -\bR_i^{*} & \bT_i
    \end{array}\right],
    \left[\begin{array}{cc}
      \bY_i & \bA_i \\
      \bA_i & \bY_i + \bZ
    \end{array}\right]
  \right\rangle
  =
  \langle \bS_i + \bT_i, \bY_i \rangle
  + \langle \bT_i, \bZ \rangle
  - \langle \bA_i, \bR_i + \bR_i^{*} \rangle.
\]

Because $c$ is unconstrained and $\bZ, \bY_1, \dots, \bY_n$ range over all Hermitian matrices, the infimum of $\sL$ over $(c, \bZ, \bY_1, \dots, \bY_n)$ is finite if and only if
\[
  \Tr(\bm \Delta) = 1,
  \qquad
  \bS_i + \bT_i = \bm \Delta \text{ for all } i \in [n],
  \qquad
  \bm U = \bm \Delta - \sum_{i = 1}^n \bT_i.
\]
Since $\bm U$ is constrained to be positive semidefinite, this last condition is equivalent to
\[
  \bm \Delta - \sum_{i = 1}^n \bT_i \succeq \bm 0.
\]
In that case, the infimum of $\sL$ over $(c, \bZ, \bY_1, \dots, \bY_n)$ is
\[
  \langle \bA_0, \bm \Delta \rangle
  + \sum_{i = 1}^n \langle \bA_i, \bR_i + \bR_i^{*} \rangle.
\]
Therefore the dual problem is
\[
  \left\{
  \begin{array}{ll}
    \text{supremum of}
    &
    \langle \bA_0, \bm \Delta \rangle
    + \displaystyle\sum_{i = 1}^n \langle \bA_i, \bR_i + \bR_i^{*} \rangle
    \\
    \text{subject to}
    &
    \left[\begin{array}{cc}
      \bS_i & -\bR_i \\
      -\bR_i^{*} & \bT_i
    \end{array}\right]
    \succeq \bm 0
    \text{ for all } i \in [n],
    \\
    &
    \bS_i, \bT_i, \bm \Delta \in \CC^{d \times d}_{\herm},
    \qquad
    \bR_i \in \CC^{d \times d},
    \\
    &
    \bS_i + \bT_i = \bm \Delta
    \text{ for all } i \in [n],
    \\
    &
    \bm \Delta - \sum_{i = 1}^n \bT_i \succeq \bm 0,
    \\
    &
    \Tr(\bm \Delta) = 1.
  \end{array}
  \right\},
\]
the same as \eqref{eq:rad-sdp-dual} from Theorem~\ref{thm:rad-sdp}.

\end{document}